\documentclass[twoside,11pt]{article}

\usepackage{blindtext}

\usepackage{jmlr2e}

\newcommand{\R}{\mathbb{R}}
\newcommand{\E}{\mathbb{E}}
\newcommand{\Prob}{\mathbb{P}}

\DeclareMathOperator*{\argmax}{arg\,max}  
\DeclareMathOperator*{\argmin}{arg\,min}

\usepackage{lastpage}

\ShortHeadings{Geometry-dependent matching pursuit}{Moucer, Taylor, Bach}
\firstpageno{1}

\begin{document}

\title{Geometry-dependent matching pursuit: a transition phase for convergence on linear regression and LASSO}

\author{\name Celine Moucer \email celine.moucer@inria.fr \\
       \addr Inria, Département d'Informatique de l'Ecole Normale Supérieure,
       PSL Research University. \\
       \addr Ecole Nationale des Ponts et Chaussées, Marne-la-Vallée, France.
       \AND
       \name Adrien B. Taylor \email adrien.taylor@inria.fr \\
       \addr Inria, Département d'Informatique de l'Ecole Normale Supérieure,
       PSL Research University.
       \AND
       \name Francis Bach \email francis.bach@inria.fr \\
       \addr Inria, Département d'Informatique de l'Ecole Normale Supérieure,
       PSL Research University.}

\editor{Editors}

\maketitle

\begin{abstract}
Greedy first-order methods, such as coordinate descent with Gauss-Southwell rule or matching pursuit, have become popular in optimization due to their natural tendency to propose sparse solutions and their refined convergence guarantees. In this work, we propose a principled approach to generating (regularized) matching pursuit algorithms adapted to the geometry of the problem at hand, as well as their convergence guarantees. 
Building on these results, we derive approximate convergence guarantees and describe a transition phenomenon in the convergence of (regularized) matching pursuit from underparametrized to overparametrized models.
\end{abstract}


\section{Introduction}
Many natural problems from machine learning and data science take the form of an $\ell_1$-regularized minimization problem:
\begin{align}
   \min_{\alpha \in \R^d} \ \{F(\alpha) + H(\alpha) \triangleq f(P\alpha) + \lambda \|\alpha\|_1\},
   \label{eq:problem}
\end{align}
where $f: \R^n \rightarrow \R$ is a smooth strongly convex function, $P \in \R^{n \times d}$ and $n,d$ respectively denote the number of samples and the dimension of the problem. Typically, in the vanilla least-squares regression problem, $H(\alpha) = 0$ and $F(\alpha) = f(P\alpha) = \frac{1}{2n}\|P\alpha - y\|_2^2$, and $P$ corresponds to the input data, $y \in \R^n$ to the labels, $d$ to the number of features (or parameters) and $n$ the number of observations. If in addition $H(\alpha) = \lambda \|\alpha\|_1$, Problem~\eqref{eq:problem} is exactly the LASSO problem~\citep{1996Tibshirani}, that belongs to more general variational problems appearing in Fenchel duality theory~\citep[Section 15.3]{Bauschke2017}. Problem~\eqref{eq:problem} is often compared to its constrained counterpart, 
\begin{equation}
\label{eq:problem_d_constrained}
    \min_{\alpha \in \R^d} F(\alpha), {\rm \ such \ that \ } \|\alpha\|_1 \leqslant R,
\end{equation}
where $\lambda$ may be seen as the Lagrange multiplier associated to the constraint $\|\alpha\|_1 \leqslant R$ with $R > 0$. Problems~\eqref{eq:problem}~and~\eqref{eq:problem_d_constrained} arise when looking for sparsity patterns, such as in signal processing where we aim for models depending on a small number of variables, or for trace-norm regularized problems, when looking for low-rank solutions~\citep{2012dudik}. In particular, Problem~\eqref{eq:problem} is a popular way to induce sparsity on the solution for a well-chosen range of $\lambda$. Thus, $\ell_1$-penalization (or constraint) is strongly connected to sparsity and can be seen as a convex substitute for $\ell_0$-penalization problems~\citep[Section 1.2]{Candes2005}for performing feature selection.

First-order methods have become popular to solve optimization Problems~\eqref{eq:problem} and~\eqref{eq:problem_d_constrained}, due to their low cost per iteration and to the limited accuracy requirements in machine learning~\citep[Section 7 and 8]{Bottou2018}. Within these methods, different algorithms might be used, whose choice depends on the properties of functions $F$ and $H$. For instance, a first-order method may rely on the gradient or the proximal operator~\citep{Parikh2013ProximalA} as in the proximal gradient method, or the linear minimization oracle as in the Frank-Wolfe algorithm~\citep{jaggi2013} for the constrained version~\eqref{eq:problem_d_constrained}. These methods often benefit from convergence guarantees.

In the context of sparsity, traditional first-order methods, such as the proximal gradient, do not always lead to sparse solutions~\citep{iutzeler2020nonsmoothness}. Boosting strategies (also known as matching pursuit) have been developed to ensure sparse representations of approximate solutions~\citep{Mallat1993, 2004Tropp}. At each iteration, a possibly new atom (also referred to as a weak-learner in the boosting literature, or a coordinate in the context of coordinate descent) is greedily selected as a best candidate among a set of atoms, and combined to past iterates. While boosting benefits from strong statistical properties~\citep{2004Tropp}, from an optimization perspective, their convergence analyses often rely on extra statistical assumptions~\citep{Zhang2012_OMP}. More recently, randomized and greedy coordinate descent methods have gained interest due to their low-cost per iteration even in high dimension~\citep{Nesterov2012} and to their implicit induced sparsity~\citep{2013Beck, Fang2020}.

Correspondences have been highlighted  between first-order methods and boosting strategies for non-regularized minimization problems ($\lambda=0$), leading to convergence guarantees independent of traditional statistical assumptions. For example, coordinate descent has been interpreted as matching pursuit~\citep{2018Locatello}, as well as Frank-Wolfe algorithms~\citep{jaggi2013, Locatello2017} for constrained Problems~\eqref{eq:problem_d_constrained}, by formulating them as minimizers of well-chosen quadratic upper approximations. These analyses strongly rely on a well-chosen geometry, characterized by a gauge function~\citep{2014friedlander}. To our knowledge, this comparison was only drawn for non-regularized problems for which $\lambda = 0$~\citep{2018Locatello} and for constrained problems~\citep{Sun2020_1}.

Problem~\eqref{eq:problem} in $\R^d$ can be naturally formulated as an optimization problem in $\R^n$, letting the gauge geometry appear,
\begin{equation}
     \label{eq:problem_n}
\begin{aligned}
   & \min_{\alpha \in \R^d, x\in \R^n} \  f(x) + \lambda \|\alpha\|_1, \ {\rm such \ that } \ x = P\alpha,\\
    &= \min_{x \in \R^n}  \ f(x) + \lambda \inf_{\alpha \in \R^d, \ x = P\alpha} \|\alpha\|_1, \\
    &=  \min_{x \in \R^n} f(x) + \lambda \gamma_{\mathcal{P}}(x),
    \end{aligned}
\end{equation}
where the gauge function is defined as $\gamma_{\mathcal{P}}(x) \triangleq \inf_{\alpha \in \R^d, \ x = P\alpha} \|\alpha\|_1$ with $\mathcal{P} = {\rm conv}(\{\pm P_{:, i}, i=1, \ldots, d\})$ the centrally symmetric convex hull of the columns of $P$. Gauge functions may be seen as generalized versions of the $\ell_1$-norm, providing a sparse representation $\alpha \in \R^d$ of a vector $x \in \R^n$ with respect to a set of atoms. Under some assumptions on $P$, the gauge function may be a norm, as we will see in Section~\ref{sec:linear_regression}. Let us for example take $\mathcal{P} = {\rm conv}(\pm e_i)$, then $\gamma_{\mathcal{P}}(x) = \|x\|_1$. 

Due to the connection between optimizing in $\R^d$ and in $\R^n$, it is possible to derive algorithms adapted to one geometry or to the other, and to formulate geometry-adapted convergence guarantees. For the $\ell_1$-geometry, \citet[Section 4]{2018Nutini} analyzed greedy coordinate descent by considering strong convexity with respect to the $\ell_1$-norm, and formulated the strong convexity parameter as an optimization problem~\citep[Appendix 4.1]{2018Nutini}. More generally, \citet[Section 2]{Daspremont2018_affine_invariant} extended smoothness and strong convexity with respect to the gauge $\gamma_{\mathcal{P}}$, which led to formulations of the smoothness parameter as an optimization problem in the work of \citet[Section 2.5]{Sun2020_1}. These optimization problems are often hard to solve (yet, they have closed-form reformulation in some cases).

The main idea of this work is to propose a principled view on gradient boosting methods, that are obtained by minimizing a smoothness upper bound with respect to the $\ell_1$-norm. This methodology leads to a new boosting strategy for regularized problems, that benefits from (sub)linear convergence properties. Unlike former methods, such as orthogonal matching pursuit under restricted isometry property (RIP) by~\citet{Zhang2012greedy}, convergence analysis is performed without statistical assumptions on the data. Convergence guarantees let appear parameters characterizing the class of functions and the geometries of optimization problems~\eqref{eq:problem} in $\R^d$ ~and~\eqref{eq:problem_n} in $\R^n$, but remain mostly intractable. To this end, we compute \textit{a priori} refined estimates of convergence rates for boosting methods applied to a particular least-squares problem. We develop two approaches for computing on the one hand deterministic estimates using SDP relaxations~\citep{Goemans95}, and on the other hand high probability bounds using random matrix theory. As a result, we observe a transition phase in the convergence rate of gradient descent (resp.~coordinate descent), depending on $(n, d)$. Surprisingly, we conclude that for a fixed number of samples $n$, adding features (dimension $d$) improves their convergence, which may be compared to the double descent phenomenon~\citep{Belkin2018} for the generalization error. Building on these results, we experimentally highlight a transition phase for the proximal gradient and regularized matching pursuit on a LASSO problem, depending on the value for $\lambda$. Finally, we define an \emph{ultimate method}, enjoying linear convergence both in the underparametrized $(n \gg d)$ and in the overparametrized $(n \ll d)$ regime, that is nonetheless not a boosting method (it may indeed add more than one atom per iteration).

\subsection{Prior works}

\textbf{Boosting algorithms. }Boosting strategies, also known as matching pursuit in signal processing, have been initiated in the context of sparse recovery~\citep{Mallat1993}, and extended to the fitting of weak-learners with `gradient boosting' techniques such as Adaboost by \citet{1999Freund}. Matching pursuit (MP) algorithms produce sparse combinations of atoms by picking a direction from a set of atoms using information on the gradient. Boosting algorithms are suited to both constrained models, with for example orthogonal matching pursuit~\citep{Chen1989, 2004Tropp, Zhang2012greedy} or greedy algorithms~\citep{2012Tewari}, as well as to unconstrained (penalized) optimization problems, with for example the vanilla boosting strategy of \citep{Zhang2012_OMP}, that minimizes a well-chosen quadratic upper-bound. Recently, \citet{Locatello2017} have unified the framework for matching pursuit and Frank-Wolfe algorithms~\citep{1956Frank} leading to non-statistical convergence guarantees for matching pursuit.

\noindent \textbf{Coordinate descent. }Coordinate descent has gained interest due to the increasing access to large amounts of data, and thereby to the use of large-scale optimization models. \citet{Tseng2001} opened the path to convergence guarantees for proximal coordinate descent on composite minimization problems~\citep{Tseng2009}. \citet{Nesterov2012} derived global guarantees for coordinate gradient descent applied on convex objectives, paving the way to families of randomized coordinate updates~\citep{2014Richtarik}, and greedy updates~\citep{2013Beck}. Yet, these analyses often lead to dimension-dependent convergence guarantees. \citet{2018Nutini} provided the first convergence guarantee of greedy coordinate descent (or coordinate descent with Gauss-Southwell rule) without dependence in the dimension, formulating the update as the minimization of a smoothness upper bound with respect to the $\ell_1$-norm. More precisely, they showed a significantly better performance of greedy coordinate descent compared to randomized coordinate descent. However, the analysis did not extend well to proximal coordinate descent, letting a dependence in the dimension appear in the convergence bound. This led to refined techniques such as the greedy update of \citet{Karimireddy2019}, with dimension-independent convergence guarantees. Finally, these methods often present the benefit of an induced sparsity, that can be linked to the $\ell_1$-norm. \citet{Locatello2017} interpreted steepest coordinate descent as a matching pursuit algorithm, where the atoms corresponds to the unitary directions. More precisely, steepest coordinate descent may be seen as the minimization of a smoothness upper bound with respect to the $\ell_1$-norm. Considering gauge functions, coordinate descent can be extended to producing solutions sparse with respect to atoms, as \citet{Sun2020_1} did with the generalized conditional gradient method~\citep{Bach2015}.

\noindent \textbf{Refined convergence guarantees. }Sparse optimization often reveals a gap between theoretical convergence guarantees and observed behaviors. The LASSO has been widely studied for statistical recovery. From an optimization point of view, most of the analyses depend on the statistical recovery efficiency. For constrained optimization problems, \citep{Zhang2012greedy} proposed a forward-backward greedy algorithm for which he derived convergence guarantees under RIP. Similarly, \citet{2010agarwal} analyzed the proximal gradient and the projected gradient under restricted strong convexity and smoothness, that comes directly from restricted eigenvalue conditions~\citep{raskutti2010}, that appear for example for random Gaussian matrices. A recent focus on average-case analysis of optimization methods under random matrices was initiated by \citet{2020Pedregosa_acceleration}, coming from the convergence analysis of the simplex method~\citep{1987_Borgwardt, Spielman_2001}. On the contrary, other works improved global convergence guarantees considering well-chosen geometries. For separable quadratics, \citet[Section 4.1]{2018Nutini} have computed explicitly the strong convexity parameter in the $\ell_1$-geometry. Generalizing unitary atoms from the $\ell_1$-geometry to atoms, \citet[Section 2.5]{Sun2020_1} formulated smoothness and strong convexity with respect to gauge functions as optimization problems. However in most cases, since these parameters are hard to compute, both strong convexity and smoothness parameters remains formulated in the $\ell_2$-norm. This often leads to additional terms in convergence guarantees, coming from the norm equivalence~\citep[Appendix 4]{2018Nutini} or from the geometry such as the pyramidal width~\citep{Lacoste2015} or the directional width~\citep{Locatello2017} in Frank-Wolfe techniques, or to the Hoffman constant~\citep{1957Hoffman} for linear mappings with strongly convex functions~\citep{2015Necoara, 2016Karimi, 2020guilleescuret}.

\subsection{Assumptions}
\label{sec:assumptions}
\textbf{Convex optimization framework.} In this work, functions $f$ into consideration are convex, differentiable and Problem~\eqref{eq:problem_n} admits at least one global minimizer $x_\star \in \R^n$. Functions $F(\cdot) = f(P\cdot): \R^d \rightarrow \R$ benefit from the same properties. We restrict ourselves to the analysis of first-order methods (linear combinations of past iterates and gradients).
\smallbreak
In this paper, functions $f$ may be smooth with respect to a generic norm $\|\cdot\|_{\R^n}$, if they verify for all $x, y \in \R^n$,
\begin{equation} \label{eq:smoothness} 
f(y) \leqslant f(x) + \langle \nabla f(x), y - x \rangle + \frac{L^f}{2} \|y - x\|_{\R^n}^2.
\end{equation}
Functions $F(\cdot) = f(P\cdot)$ are therefore smooth with respect for any norm $\|\cdot\|_{\R^d}$ with $L^F \leqslant L^f L^{\mathcal{P}}$,  where $L$ is defined such that for all $\alpha, \beta \in \R^d$, $\|P( \alpha - \beta)\|_{\R^n}^2 \leqslant L^{\mathcal{P}} \|\alpha - \beta\|_{\R^d}^2$, that is $L^{\mathcal{P}} = \sup_{\|\beta\|_{\R^d} \leqslant 1} \|P\beta\|_{\R^n}^2$. For least-squares, functions $F$ are exactly smooth with $L^F = L^f L^{\mathcal{P}}$. In addition, functions $f$ are strongly convex with respect to a norm $\|\cdot\|_{\R^n}$, if for all $x, y \in \R^n$,\begin{equation}
    f(y) \geqslant f(x) + \langle \nabla f(x), y-x \rangle + \frac{\mu^f}{2}\|y - x\|_{\R^n}^2.
\label{eq:strongconv}
\end{equation}
Functions $F(\cdot) = f(P\cdot)$ do not always inherit strong convexity. For example, for least-squares, functions $F$ are not strongly convex as soon as the number of samples $n$ is lower than the dimension $d$. The `natural' strong convexity parameter of functions $F$ is given by $\mu^F = \mu^f \mu^{\mathcal{P}}$, with $\mu^{\mathcal{P}} = \inf_{\|\beta\|_{\R^d} \geqslant 1} \|P\beta\|_{\R^n}^2$ and may indeed be zero. As we will see in Section~\ref{section_matching_pursuit}, $F$ however inherits the \L ojasiewicz property with parameter~$\mu^{F}_L>0$, such that for all $\beta \in \R^d$,
\begin{equation}
    \frac{1}{2}\|\nabla F(\beta)\|_{\R^d, \star}^2 \geqslant \mu^{F}_L(F(\beta) - F_\star),
    \label{mu_PL}
\end{equation}
where $\|\cdot\|_{\R^d, \star}$ is the dual norm for $\|\cdot\|_{\R^d}$.
\bigbreak
\noindent \textbf{Random matrices}. A part of this work is devoted to approximating the strong convexity and smoothness parameters of $f$ and $F$. We consider on the one hand relaxed formulations for strong convexity and smoothness parameters with respect to the data~$P$ (i.e., geometry). On the other hand, we propose high probability bounds of these parameters, relying on random matrix theory. Random matrices often appears in statistical assumptions, such as with restricted isometry property~\citep{Candes2005} or the restricted eigenvalue condition~\citep{raskutti2010}. In the machine learning literature, random matrices appear in average-case analysis for quadratics~\citep{2020Pedregosa_acceleration} with the Marchenko-Pastur distribution~\citep{1967Marchenko}, or when studying the double descent phenomena for the generalization error~\citep{Belkin2018, 2022Mei, Bach2023_double_descnet} for Gaussian data. Most of the time, these analyses let two regimes appear, depending (among others) on the number of samples $n$ and the dimension $d$. Throughout this work, we thus consider two regimes depending on the linear mapping structure $P \in \R^{n \times d}$: the \emph{underparametrized} (respectively \emph{overparametrized}) regime, characterized by matrices $P \in \R^{n \times d}$ for which $n \geqslant d$ (resp.~$d \geqslant n$) and $P^\top P$ (resp.~$PP^\top$) is invertible. Note that the invertibility of $PP^\top$ (resp.~$P^\top P$) in the overparametrized (resp.~underparametrized) regime can be obtained by adding sufficiently random noise. More assumptions on $P$ and $P^\top P$ will be made across this study.

\section{A transition phase for linear regression}
\label{sec:linear_regression}

We begin with the study of a linear regression problem, where problem~\eqref{eq:least_square} is a special case of the optimization Problem~\eqref{eq:problem} with $\lambda = 0$,
\begin{equation}
\label{eq:least_square}
    \min_{\alpha \in \R^d} \ \{ F(\alpha) = f(P\alpha) = \frac{1}{2n}\|P\alpha - y\|_2^2 \},
\end{equation}
where $P \in \R^{n \times d}$, and $n, d$  respectively denotes the number of samples and the dimension. 

In this section, we focus on describing the convergence regimes of gradient descent in the $\ell_2$-geometry and coordinate descent with the Gauss-Southwell (GS) rule~\citep{2016Karimi, 2018Nutini} in the $\ell_1$-geometry. More precisely, we interpret gradient descent and coordinate descent as the minimizers of smoothness upper bound with respect to well-chosen norms, that is, as optimization problems in the geometry under consideration. This interpretation leads to linear convergence both in the underparametrized and the overparametrized regime, letting smoothness and strong convexity parameters appear, that are adapted to the geometry. For characterizing convergence properties of these methods, we provide estimates of these quantities. A first technique developed in this work is based on an SDP relaxation, and leads to deterministic estimates. A second technique, inspired from statistical assumptions and average-case analysis, leads to high probability bounds under statistical assumptions on the data. These estimates let a transition phase appear between the underparametrized and overparametrized regimes, that we illustrate in particular in a random feature experiment. Finally, we interpret coordinate descent as a matching pursuit algorithm depending on the geometry $P$. 
\smallbreak
First, let us compute estimates of smoothness and strong convexity parameters by formulating their computation as optimization problems in a generic norm for the least-squares minimization~\eqref{eq:least_square}. In this context, $f$ is $\frac{1}{n}$-smooth $\frac{1}{n}$-strongly convex with respect to the norm~$\|\cdot\|_2$. Thus, $F$ is $L^F$-smooth with respect to an arbitrary norm $\|\cdot\|$ in $\R^d$, with $L^F = \frac{1}{n}\sup_{\|\beta\|^2 \leqslant 1} \|P\beta\|_2^2$. In addition, the function is (possibly) $\mu^F$-strongly convex, with a parameter $\mu^F$ explicited in Lemma~\ref{mu_F_formulation} and possibly equal to 0 (especially when dimension $d < n$).
\begin{lemma}
\label{mu_F_formulation}
    Let $F = \frac{1}{n}\|P\alpha - y\|_2^2$, where $P \in \R^{n \times d}$. Then, $F$ is $\mu^F$-strongly convex with respect to a norm~$\|\cdot\|$ with,
    \begin{align*}
        \mu^F &= \frac{1}{n}\inf_{\|\beta\|^2 \geqslant 1} \|P\beta\|_2^2 \ \ \ \ \ {\rm and} \ \ \ \ \ \frac{1}{\mu^F} = n \sup_{\|P\beta\|^2 \leqslant 1} \|\beta\|_2^2.
    \end{align*}
\end{lemma}
\begin{proof}
    Let us recall the definition for strong convexity~\eqref{eq:strongconv}, for all $\alpha, \nu \in \R^d$, $F(\alpha) \geqslant F(\nu) + \langle \nabla F(\nu), \alpha - \nu\rangle + \frac{\mu^F}{2}\|\alpha - \nu\|^2$. Since $F$ is a quadratic, the left-hand side of the inequality can be rephrased into, for all $\beta \in R^d$, $\|P\beta\|_2^2 \geqslant \mu^F\|\beta\|^2$, from which both formulations follow.
\end{proof}
In Lemma~\ref{mu_F_formulation}, we formulate $\mu^F$, the strong convexity parameter for $F$, as a nonconvex minimization problem, with a convex objective and concave constraints. Such a problem is usually costly to solve. The function $F$ also verifies the \L ojasiewicz inequality~\eqref{mu_PL} with~$\mu^{F_L}$. Again $\mu^{F}_L$ is formulated as an optimization problem.
\begin{lemma}
    \label{mu_PLF_formulation}
    Let $F = \frac{1}{n}\|P\alpha - y\|_2^2$, where $P \in \R^{n \times d}$. Then, $F$ verifies the \L ojasiewicz inequality~\eqref{mu_PL} with respect to a (dual) norm~$\|\cdot\|_\star$, with
    \begin{align*}
        \mu^{F}_L &= \frac{1}{n}\inf_{\|P\beta\|_2^2 \geqslant 1} \|P^\top P\beta\|_\star^2\ \ \ \ \ {\rm and} \ \ \ \ \ \frac{1}{\mu^{F}_L} = n \sup_{\|P^\top P\beta\|_\star^2 \leqslant 1} \|P\beta\|_2^2.
    \end{align*}
\end{lemma}
\begin{proof}
    The proof follows from the \L ojasiewicz inequality. Given that $y = P\alpha_\star$, where $\alpha_\star$ is an optimal point for $F$, we have for all $\alpha \in \R^d$: $\|\nabla F(\alpha)\|_\star^2 = \frac{1}{n}\|P^\top P(\alpha - \alpha_\star)\|_\star$ and $F(\alpha) - F_\star = \frac{1}{2n}\|P(\alpha - \alpha_\star)\|_2^2$. Then, for all $\beta \in \R^d$, $\|P^\top P(\alpha - \alpha_\star)\|_\star \geqslant \mu^{F}_L \|P\beta\|_2^2$.
\end{proof}
Again in Lemma~\ref{mu_PLF_formulation}, $\mu^{F}_L$ is formulated as a (nonconvex) minimization problem. The two quantities $\mu^F$ and $\mu^{F}_L$ are compared in Lemma~\ref{comparison_mu_pl}, with equality in the underparametized regime in which $P^\top P$ is invertible.

\begin{lemma}
\label{comparison_mu_pl}
    Let $F = \frac{1}{2n}\|P\alpha - y\|_2^2$. Then, we have that $\mu^{F}_L \geqslant \mu^F$ for $\mu^F$ (resp.~$\mu^{F}_L$) defined in Lemma~\ref{mu_F_formulation} (resp.~Lemma~\ref{mu_PLF_formulation}).
If $P^\top P$ is invertible, $\mu^{F}_L = \mu^F$.
\end{lemma}
\begin{proof}
    Let us consider the squared-root formulations of $\mu^F$ and $\mu^{F}_L$ given in Lemma~\ref{mu_F_formulation} and Lemma~\ref{mu_PLF_formulation}.
    \begin{align*}
        \frac{1}{\sqrt{n \mu^F}} &= \sup_{\|P\beta\|_2 \leqslant 1} \|\beta\| = \sup_{ \|z\|_\star \leqslant 1, \|P\beta\|_2 \leqslant 1} \langle \beta, z \rangle, \\
        \frac{1}{\sqrt{n \mu^{F}_L}} &= \sup_{\|P^\top P \nu\|_{\star} \leqslant 1} \|P\nu\|_2 = \sup_{\|P^\top P \nu\|_{\star} \leqslant 1, \|P\beta\|_2 \leqslant 1} \langle P\beta, P\nu \rangle = \sup_{\|P^\top P \nu\|_{\star} \leqslant 1, \|P\beta\|_2 \leqslant 1} \langle \beta, P^\top P\nu \rangle.
    \end{align*}
Since ${\rm Im}(P^\top P) \subset \R^d$, we have $\frac{1}{\sqrt{n \mu^F}} \geqslant \frac{1}{\sqrt{n\mu^{F}_L }}$, and therefore $\mu^{F}_L \geqslant \mu^F$. In the special case where $P^\top P$ is invertible, ${\rm Im}(P^\top P) = \R^d$, and $\mu^{F}_L = \mu^F$.
\end{proof}

In the next sections, we see the role of these parameters in the convergence guarantees of gradient descent and steepest coordinate descent, both in the underparametrized and overparametrized regime. We then propose deterministic estimates for $\mu^F$ and $\mu^{F}_L$, as well as high probability bounds based on a simple random model for $P$.

\subsection{Gradient descent in the $\ell_2$-geometry}
We are interested in the convergence of gradient descent in the underparametrized and the overparametrized regimes. Assume~$\R^d$ is equipped with the $\ell_2$-norm. The function $F$ is convex, $L_2^F$-smooth with respect to the norm~$\ell_2$, with $L_2^F = \frac{1}{n}\lambda_{\max}(P^\top P)$. A common interpretation of gradient descent with fixed step size $\gamma = \frac{1}{L_2^F}$ comes from the minimization of a quadratic (smoothness) upper bound on~$F$:\begin{equation}
    \alpha_1 = \alpha_0 - \frac{1}{L_2^F} \nabla F(\alpha_0) = \alpha_0 - \frac{1}{L_2^F} P^\top(P\alpha_0 - y). 
    \label{eq:gd}
\end{equation}
In the underparametrized regime, the function $F$ is $\mu_2^F$-strongly convex with respect to the $\ell_2$-norm, with $\mu_2^F = \lambda_{\min}(\frac{P^\top P}{n}) > 0$. As a result, gradient descent~\eqref{eq:gd} converges linearly. However, in the overparametrized regime in which $d \geqslant n$, $\mu_2^F = 0$, $F$ is not strongly convex. Yet, gradient descent still converges linearly~\citep{2010Bolte}, since quadratics benefit from the \L ojasiewicz inequality, with $\mu^{F}_{L,2} = \frac{1}{n}\lambda_{\min}(PP^\top) > 0$.

\begin{proposition}
\label{theolinreg_gd}
    Let $F$ be convex, $L_2^F$-smooth with respect to the norm $\|\cdot\|_2$, be $\mu_2^F$-strongly convex and verify a \L ojasiewicz inequality with parameter  $\mu_{2, L}^{F}$, with $0 \leqslant \mu_2^F \leqslant L_2^F$ and $0 \leqslant \mu_{2, L}^{F} \leqslant L_2^F$. Let $(\alpha_k)_{k\in \mathrm{N}}$ be generated by gradient descent in~\eqref{eq:gd} starting from $\alpha_0 \in \R^d$. The sequence verifies:
    \begin{equation*}
        F(\alpha_k) - F_\star \leqslant \left(1 - \frac{\max(\mu_2^F, \mu_{2, L}^{F})}{L_2^F}\right)^{k}(F(\alpha_0) - F_\star),
    \end{equation*}
where $\mu_2^F = \lambda_{\min}(PP^\top/n)$, $\mu_{2, L}^{F} = \lambda_{\max}(P^\top P / n)$ and $L_2^F = \lambda_{\max}(P^\top P / n)$
\end{proposition}
\begin{proof}
    See appendix \ref{proof_linreg_gd_coord}.
\end{proof}
 Convergence speeds obtained in Proposition~\ref{theolinreg_gd} depend on $\lambda_{\max}(P^\top P / n)$, $\lambda_{\max}(P^\top P / n)$ and $\lambda_{\min}(PP^\top/n)$. In the case where $P$ is generated randomly, we can derive estimates of these extremal eigenvalues, avoiding a full computation of the extremal eigenvalues, and thus, an approximate convergence guarantee of the method. In the following, we consider random data $P$, with i.i.d.~entries having the same variance, so that $P^\top P$ and $PP^\top$ have a limiting Marchenko-Pastur distribution~\citep{1967Marchenko}, whose extremal eigenvalues are known. This distribution generalizes the Wishart distribution of $P^\top P$ and $PP^\top$, obtained from Gaussian data $P$. 

\begin{theorem}[Limits of extreme eigenvalues - Theorem 5.11~\citep{Bai2010}]

    \label{MP_theorem}  Assume $P \in \R^{n \times d}$, where each entry is an i.i.d.~random variable with mean 0, variance $\sigma^2$, $\E[P_{i, j}^4] < + \infty$ and let $H = \frac{1}{n}P^\top P$. If $\frac{d}{n} \rightarrow r \in (0, \infty)$, then we have almost surely that
    \begin{align*}
        \lambda_{\min}(H) &\rightarrow \sigma^2(1 - \sqrt{r})^2, \\
        \lambda_{\max}(H) &\rightarrow \sigma^2(1 + \sqrt{r})^2.
    \end{align*}
\end{theorem} Combining Theorems~\ref{MP_theorem}~with Proposition~\ref{theolinreg_gd}, we obtain natural estimates of the convergence properties in the underparametrized and the overparametrized regimes for random~$P$.

\begin{corollary}
\label{under_over_corollary}
    Under the same assumptions than Theorem~\ref{MP_theorem} and assuming $\E[\|P\|^4] < + \infty$, gradient descent with step size $\gamma = \frac{1}{L^F_2}$ converges linearly to the optimum. Then, if $\frac{d}{n} \rightarrow r \in (0, \infty)$, we have almost surely that
    \begin{itemize}
        \item in the underparametrized regime $(r \ll 1)$: $1-\frac{\lambda_{\min}(P^\top P)}{\lambda_{\max}(P^\top P)} \rightarrow 1-\frac{(1 - \sqrt{r})^2}{(1 + \sqrt{r})^2}$,
        \item in the overparametrized regime $(r \gg 1)$: $1-\frac{\lambda_{\min}(PP^\top)}{\lambda_{\max}(P^\top P)} \rightarrow 1- r\frac{(1 - \sqrt{1/r})^2}{(1 + \sqrt{r})^2}$.
    \end{itemize}
\end{corollary}
\begin{proof}
    Applying the Marchenko-Pastur Theorem~\ref{MP_theorem} with Proposition~\ref{theolinreg_gd} leads to the result.
\end{proof}

Given that $\frac{d}{n} \rightarrow r$, we deduce approximate convergence guarantees from Corollary~\ref{under_over_corollary}: in the underparametrized regime ($\frac{d}{n} \rightarrow r \ll 1$), we have $1-\frac{\lambda_{\min}(P^\top P)}{\lambda_{\max}(P^\top P)} = 4 \sqrt{r} + o(\sqrt{r}) \approx 4 \sqrt{\frac{d}{n}}$, and in the overparametrized regime ($\frac{d}{n} \rightarrow r \gg 1$), $1-\frac{\lambda_{\min}(PP^\top)}{\lambda_{\max}(P^\top P)} =4 \sqrt{1/r} + o(\sqrt{1/r}) \approx 4 \sqrt{\frac{n}{d}}$. These approximate convergence rates should be to compared to the average-case analysis of \citet{2020Pedregosa_acceleration} for least-squares problems, and to the polynomial-based analysis for convergence of gossip developed by \citet{Berthier2020}. Depending on the distribution under consideration, Scieur~et~Pedregosa developed average-case optimal accelerated methods, whose limit in the number of iterations happens to be the Polyak-Momentum~\citep{2020Scieur_Polyak}. Its worst-case convergence guarantee verifies $\frac{\sqrt{L_2^F} - \sqrt{\mu_2^F}}{\sqrt{L_2^F} + \sqrt{\mu_2^F}} \approx \sqrt{\frac{d}{n}}$, can be compared to Nesterov's accelerated gradient method~\citep{Nesterov1983} with $1 - \sqrt{\frac{\mu_2^F}{L_2^F}} \approx 2\sqrt{\frac{d}{n}}$ and to gradient descent with $1 - \frac{\mu_2^F}{L_2^F} \approx 4 \sqrt{\frac{d}{n}}$. When considering their average-case guarantees, only a polynomial sublinear term is added~\citep[Table 2]{2023Paquette} to the worst-case guarantee, without major modifications in the linear term. In other words, Polyak-Momentum (resp.~Nesterov's accelerated gradient method) converges four times (resp.~twice) as fast as gradient descent with fixed step sizes.

We conclude from Corollary~\ref{under_over_corollary} that convergence of gradient descent depends on the degree of underparametrization (resp.~overparametrization). The more independent samples (resp.~features), the better the convergence. While the advantage of adding independent samples is well known for improving both learning and convergence speed, it appears that adding features improves the convergence speed too. We are now going the study approximate convergence guarantees of coordinate descent, that appears in the $\ell_1$-geometry.

\subsection{Gauss-Southwell coordinate descent in the $\ell_1$-geometry}

Similar to gradient descent, we study convergence guarantees of coordinate descent based on the Gauss-Southwell (GS) rule. The GS-rule can be obtained from a smoothness upper bound with respect to the $\ell_1$-norm, as shown by \citet[Section 4]{2018Nutini}, for all $\alpha_0, \alpha \in \R^d$,
\begin{equation}
\label{smoothness_F}
    F(\alpha) \leqslant F(\alpha_0) + \langle\nabla F(\alpha_0), \alpha - \alpha_0 \rangle + \frac{L_1^F}{2}\|\alpha - \alpha_0\|^2_1.
\end{equation}
From this inequality, we compute $L_1^F = \frac{1}{n}\max_{\alpha \in \R^d, \|\alpha\|_1=1} \|Pz\|_2^2 = \frac{1}{n}\max_{i=1, \ldots, d} \|P_{:,i}\|_2^2$ (the maximization problem attains its optimum on the extremal point of the simplex). Gauss-Southwell coordinate descent follows by minimizing over $\alpha \in \R^d$, for a fixed $\alpha_0 \in \R^d$,
\begin{equation}
\begin{aligned}
    i_0 &= \argmax_{k = 1, \ldots , d} |\nabla_{i_k} F(\alpha_0)|, \\
    \alpha_1 &= \alpha_0 - \frac{1}{L_1^F}\nabla_{i_0} F(\alpha_0) e_{i_0}.
\end{aligned}
\label{GS_coordinate}
\end{equation}

As for gradient descent, its convergence speed depends on the parametrization regime. Depending on the $(n,d)$, $F$ may be $\mu_1^F$-strongly convex, or verify the \L ojasiewicz inequality with parameter $\mu_{1, L}^{F}$. Both $\mu_1^F$ and $\mu_{1, L}^{F}$ can be formulated as optimization problems for computing explicit estimates. It follows from the strong convexity characterization given in Lemma~\ref{mu_F_formulation} with the norm~$\| \cdot\|_1$, that $\mu_1^F = \frac{1}{n}\inf_{\|z\|_1^2 \geqslant 1} \|Pz\|_2^2$, and from Lemma~\ref{mu_PLF_formulation} with norm~$\|\cdot\|$ for the \L ojasiewicz inequality $\mu_{1, L}^{F} = \inf_{\|P\beta\|_2^2 \geqslant 1} \|P^\top P \beta\|_{\infty}^2$.

In the regimes under consideration, Proposition~\ref{conv_coordinate} states that coordinate descent converges linearly, as already proven by \citet[Theorem 1]{2016Karimi}.
\begin{proposition}{\citep[Theorem 1]{2016Karimi}}
\label{conv_coordinate}
    Let $F$ be convex, $L_1^F$-smooth with respect to the norm $\|\cdot\|_1$, be $\mu_1^F$-strongly convex and verify the \L ojasiewicz inequality with $\mu_{1, L}^{F}$, where $0 \leqslant \mu_{1, L}^{F} \leqslant L_1^F$ and $0 \leqslant \mu_1^F \leqslant L_1^F$. Let $(\alpha_k)$ be generated by coordinate gradient descent~\eqref{GS_coordinate} starting from $\alpha_0 \in \R^d$. The sequence verifies:
    \begin{equation*}
        F(\alpha_k) - F_\star \leqslant \left( 1 - \frac{\max(\mu_1^F, \mu_{1, L}^{F})}{L_1^F}\right)^{k}(F(\alpha_0) - F_\star)
    \end{equation*}
\end{proposition}
\begin{proof}
    See Appendix~\ref{proof_linreg_gd_coord}.
\end{proof}
The convergence guarantee provided in Proposition~\ref{conv_coordinate} depends on $\mu_1^F$ and $\mu_{1, L}^F$ and hence complicated to compute. Although $L_1^F = \frac{1}{n} \max_{i=1, \ldots ,d } \|P_{:,i}\|_2^2$ has a closed-form solution, $\mu_1^F= \frac{1}{n}\inf_{\|z\|_1^2 \geqslant 1} \|Pz\|_2^2$ and $\mu_{1, L}^{F}= \inf_{\|P\beta\|_2^2 \geqslant 1} \|P^\top P \beta\|_{\infty}^2$ are formulated as nonconvex minimization problems. 

In the following, we construct estimates to these quantities, so that they may be computed a priori. First, we provide SDP relaxations for the optimization problems defining $\mu_1^F$ and $\mu_{1, L}^F$, that may differ from the exact solution. We thus propose to construct high probability bounds for $\mu_1^F$ and $\mu_{1, L}^{F}$, assuming randomly generated data.
\bigbreak
\noindent \textbf{SDP relaxations.} Building on the formulation of $\mu_1^F$ and $\mu_{1, L}^{F}$ as optimization problems, we rephrase them into relaxed SDPs.

\begin{proposition}
\label{Exact_approx_GS}
    Let $P \in \R^{n \times d}$, and let us define $\mu_1^F= \frac{1}{n}\inf_{\|z\|_1^2 \geqslant 1} \|Pz\|_2^2$  and $\mu_{1, L}^{F}= \inf_{\|P\beta\|_2^2 \geqslant 1} \|P^\top P \beta\|_{\infty}^2$. Then the following inequality holds 
    \begin{itemize}
        \item in the underparametrized regime, $\frac{1}{\tilde{n \mu}_1^F} = \sup_{X \succcurlyeq 0} {\rm Tr}((P^\top P)^{-1}X), \ \ {\rm s.t. } \ {\rm diag}(X) \leqslant 1$,
        \begin{equation*}
            1 - \frac{\pi}{2}\frac{\tilde{\mu}_1^F}{L_1^F} \leqslant 1 - \frac{\mu_1^F}{L_1^F} \leqslant 1 - \frac{\tilde{\mu}_1^F}{L_1^F},
        \end{equation*}
        \item in the overparametrized regime, $\frac{1}{n \tilde{\mu}_{1, L}^{F}} = \sup_{X \succcurlyeq 0} {\rm Tr}(P^\top P X) \ \ {\rm s.t.} \ \|P^\top P X P^\top P\|_{\infty} \leqslant 1$,
        \begin{equation*}
            1 - \frac{\mu_{1, L}^{F}}{L_1^F} \leqslant 1 - \frac{\tilde{\mu}_{1, L}^{F}}{L_1^F},
        \end{equation*}
    \end{itemize}
    where $L_1^F = \frac{1}{n} \max_{i=1, \ldots ,d } \|P_{:,i}\|_2^2$. In addition, we still have that $\tilde{\mu}_1^F \leqslant \tilde{\mu}_{1, L}^{F}$.
\end{proposition}
\begin{proof}
    See Appendix~\ref{exact_approximation}.
\end{proof}
In Proposition~\ref{Exact_approx_GS}, we find out SDP relaxations that yield an deterministic estimate for $\mu_1^F$, and an exact lower bound for $\mu_{1, L}^{F}$. Yet, the larger $n, d$, the longer the computation of these~SDPs. 
\bigbreak
\noindent \textbf{High probability bounds.} We now assume that $P$ is randomly generated, as in the $\ell_2$-geometry. Under subgaussian assumptions, we derive in Proposition~\ref{big_theo_approx} high probability bounds for $\mu_1^F$, $\mu_{1, L}^{F}$ and $L_1^F$. More precisely, we prove that $L_1^F$ concentrates around the variance $\sigma^2$, $\mu_1^F$ around $\frac{\sigma^2}{d}$ and $\mu_{1, L}^{F}$ around $\frac{\sigma^2}{n}$ with subgaussian tails.

\begin{proposition}
    \label{big_theo_approx}
    Let $P\in \R^{n \times d}$, with $P_i \in \R^d$ i.i.d.~subgaussian such that $\E[P_{i, j}] = 0$, $\E[P_{i, j}] = \sigma^2$. There exists absolute constants $C, C_1, C_2, C_3, C_4, K > 0$ such that, 
    \begin{itemize}
        \item For all $u \geqslant 2 K^2 \sqrt{\frac{C_1 \log(d)}{n}}$, 
        \begin{equation*}
        \left(1 + C_2K^2\frac{1}{\sqrt{n}} - t \right)^2\leqslant \frac{L_1^F}{\sigma^2} \leqslant \left(1 + 2 K^2 \sqrt{\frac{C_1\log(d)}{n}} + t\right)^2,
         \end{equation*}
        holds with probability $1 - e^{-\frac{C}{\sigma^2K^4}\min(u_1(t), u_2(t))}$ where $u_1(t) = \log(d)\sigma^2(t + \frac{C_2K^2}{\sqrt{n}})^2$ and $u_2(t) = d\sigma^2(t - 2 K^2 \sqrt{\frac{C_1 \log(d)}{n}})^2$.
        \item For all $t \geqslant 0$, it holds with probability $1 - 2\exp(-t^2)$,
        \begin{equation*}
           \left(1 - C_3K^2\left(\sqrt{\frac{d}{n}} + \frac{t}{\sqrt{n}}\right)\right)^2\leqslant  \mu_1^F\frac{d}{\sigma^2} \leqslant \left(1 + C_3K^2\left(\sqrt{\frac{1}{n}} + \frac{t}{\sqrt{dn}} \right) \right)^2.
        \end{equation*}
        \item For all $t \geqslant 2\sigma K^2 \sqrt{\frac{C_1\log(d)}{n}}$, it holds with probability $ 1 - 2\exp(-\min(t^2, u_2(t))$,
    \begin{equation*}
    \left(1 - C_4K^2\left( \sqrt{\frac{n}{d}} + \frac{t}{\sqrt{d}}\right)\right)^2 \leqslant \mu_{1, L}^{F}\frac{n}{\sigma^2} \leqslant \left( 1 +  2 K^2 \sqrt{\frac{C_1\log(d)}{n}} + t\right)^2.
    \end{equation*}
    \end{itemize}
   The constant $K> 0$ characterizes subgaussian vectors of $P$ (and defined in Appendix~\ref{proof:big_theo_approx}).
\end{proposition}
\begin{proof}
    See Appendix~\ref{proof:big_theo_approx}.
\end{proof}

Compared with Proposition~\ref{theolinreg_gd}, Proposition~\ref{big_theo_approx} provides concentration inequalities for $L_1^F$, $\mu_1^F$ and $\mu_{1, L}^{F}$ depending on dimension $d$, the variance $\sigma^2$, the number of samples $n$ and absolute constants. 

From Proposition~\ref{conv_coordinate}, we have seen that coordinate descent with GS-rule~\eqref{GS_coordinate} converges in function values with a rate $1 - \frac{\max(\mu_{1, L}^F, \mu_1^F)}{L_1^F}$. In the overparametrized regime (resp.~underparametrized), we conclude in Coroallary~\ref{coro:limiting_mu_1_L_1} with limiting concentration of the convergence rate for large dimensions (resp.~large number of samples).

\begin{corollary}
\label{coro:limiting_mu_1_L_1}
    Let $P\in \R^{n \times d}$, with $P_i \in \R^d$ i.i.d.~subgaussian such that $\E[P_{i, j}] = 0$, $\E[P_{i, j}] = \sigma^2$. Then, 
    \begin{itemize}
        \item in the underparametrized regime, when $n \rightarrow \infty$, the quantity $1 - \frac{\mu_1^F}{L_1^F}$ concentrates in $1 - \frac{1}{d} + O(\frac{1}{\sqrt{n}})$ with subgaussian tails,
        \item in the overparametrized regime, when $d \rightarrow \infty$ and $\frac{\log(d)}{n} \rightarrow 0$, the quantity $1 - \frac{1}{\mu_{1, L}^{F}}$ concentrates in $1 - \frac{1}{n} + O(\frac{1}{\sqrt{d}}) + O(\sqrt{\frac{\log(d)}{n}})$ with subgaussian tails.
    \end{itemize}
\end{corollary}
\begin{proof}
    See the proof in Appendix~\ref{proof:coro_limiting_mu_1_L_1}.
\end{proof}
For large overparametrized models (resp.~underparametrized), the convergence guarantee of coordinate descent with GS rule concentrates to $(1 - \frac{1}{n})$ (resp.~$(1 - \frac{1}{d})$), that is independent of the dimension $d$ (resp.~of the number of samples). Note that the condition $\log(d) \ll n$ is indeed reasonable, since $e^n$ grows quickly (when $n = 50$, $e^n \approx 5 \times 10^{21}$). A numerical comparison for the expected and exact lower bounds for $\mu_1^F$ and $\mu_{1}^{L, F}$ is provided in~Appendix~\ref{proof_linreg_gd_coord}. Unlike the approximate convergence guarantees for gradient descent in the underparametrized regime (resp.~overparametrized) detailed in Corollary~\ref{under_over_corollary}, coordinate descent with GS-rule does not improve when adding samples (resp.~features).

As for gradient descent in the $\ell_2$-geometry, we have formulated coordinate descent with GS rule as the minimization of the smoothness upper bound with respect to the $\ell_1$-norm, leading its linear convergence in both the underparametrized and overparametrized regime. For a linear regression problem, nor the strong convexity parameter neither the \L ojasiewicz in the $\ell_1$-geometry benefit from a closed-form formulation (but it did in the $\ell_2$-geometry). In a first approach, we approximate these quantities by SDPs, that may take longer computation in large models (either in the number of samples or the dimension). Instead of that, we consider randomly generated matrices $P$ to approximate these parameters. Under subgaussian data, it appears $\mu_1^F$, $\mu_{1, L}^{F}$ and $L_1^F$ concentrate to there expectation with subgaussian tails. In the next section, we perform numerical experiments showing the transition phase between the two regimes.

\subsection{A transition phase phenomenon: experimental results}
\label{sec:random_features}

We compare approximate convergence guarantees to numerical experimental convergence for gradient descent from Corollary~\ref{under_over_corollary} and for coordinate descent from Corollary~\ref{coro:limiting_mu_1_L_1}. More precisely, we verify the expected transition phase in $(n, d)$: in the overparametrized (respectively underparametrized) regime, the larger the dimension (resp.~the number of samples), the better the convergence. To this end, we perform a few experiments on several datasets: least-squares problems obtained either with synthetic Gaussian vectors or from the Leukemia dataset or from random features, described below. 

\noindent \textbf{Synthetic quadratics}. We consider several least-squares problems~\eqref{eq:least_square}, where the number of samples $n=50$ is fixed, and the number of dimension varies so that both the overparametrized and the underparametrized regimes are explored. In this model, the feature matrix $P$ into consideration is generated such that $P_{:,i} \sim \mathcal{N}(0, I_d)$ are i.i.d, $\alpha_\star \in \{-1, 1\}^d$ has a sparsity (that is, the number of non-zero entries) equal to $s = 8 < d$, and $y = P\alpha_\star + \epsilon$, where $\epsilon_i \sim \mathcal{N}(0, \sigma$).

\noindent \textbf{Leukemia dataset}. We consider the standard Leukemia dataset~\citep{Golub1999}, where $n=72$ and $d=7129$. Again, we consider submatrices, so that the dimensions vary from the underparametrized regime to the overparametrized one. For each model, $P$ has zero mean and features with unit variance.

\noindent \textbf{Random features. }We consider the example of random features for a fixed prediction model. We consider the regression model $\hat{a} = \argmin_{a \in \R^d} \frac{1}{2n}\|y - f(P, a, \theta)\|_2^2$, where the family of models is given by $\mathcal{F}(\theta) = \{ f(P, a, \theta) = \sum_{i=1}^d a_i \sigma(\langle \theta_{:, i}, P_{:, j}\rangle) = \phi_P(\theta)^\top a, a \in \R^N\}$, where $\theta \in \R^{d \times m} \sim \mathcal{N}(0, \nu^2)$, and $\sigma(\cdot) = \max(0, \cdot)$. In this experiment, we increase the number of features $d$ (from 10 to 1000) while the initial data taken from the leukemia dataset is such that $n = 72$, $m=200$ and $\theta_i \sim \mathcal{N}(0, I_m)$. In comparison to experiments on synthetic quadratics and on the leukemia dataset, the model does not vary in random features: all models converge to the same optimal solution $y$.

\begin{figure}[!h]
    \centering
    \includegraphics[height=260pt]{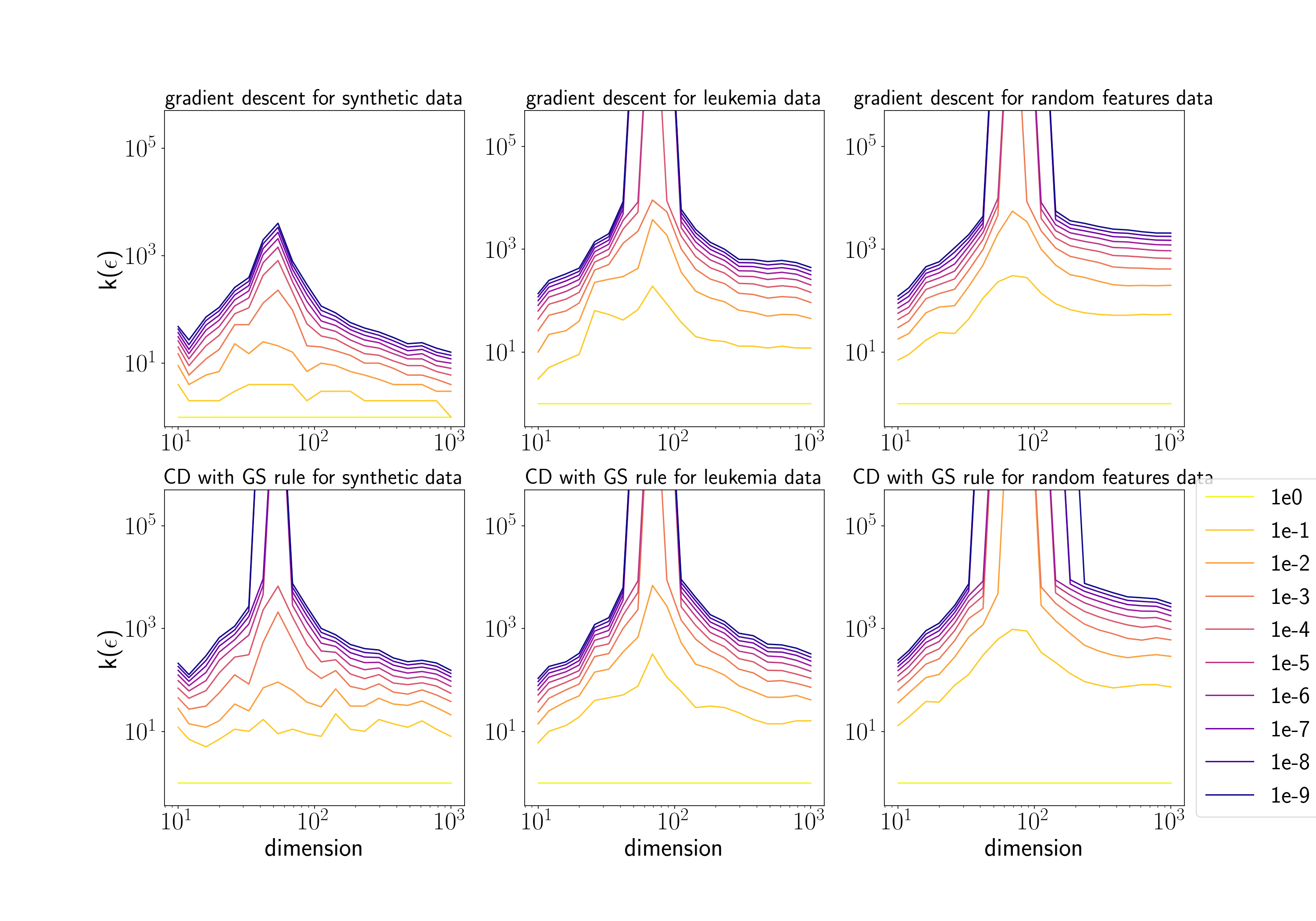}
    \vspace{-20pt}
    \caption{$\epsilon$-curve for gradient descent (top) and coordinate descent with the GS-rule (bottom), for the three models: synthetic quadratics (on the left) with $n=50$, the leukemia dataset (in the middle) with $n=72$, a random feature model (on the right) with $n=72$.}
\label{fig:RF_lireg}
\end{figure}

In Figure~\ref{fig:RF_lireg}, we plot the iteration number $k(\epsilon)$ at which a certain accuracy $\epsilon$ is reached for the three models describes above, both for gradient descent and coordinate descent with GS-rule. We consider accuracies $\epsilon = \{10^{-0}, \ldots, 10^{-10}\}$ and we refer to these curves by $\epsilon$-curves. 
For the three models, both steepest coordinate descent and gradient descent converge faster  when $n \gg d$ (resp.~$d \gg n$) in the underparametrized (resp.~overparametrized) regime in Figure~\ref{fig:RF_lireg}. For $n \approx d$, convergence slows down and tends to be sublinear, as expected from the theory for smooth convex functions. In other words, we observe a transition phenomenon for dimensions $d \approx n$. For the random feature models, a double descent phenomena was empirically highlighted by \citet{Belkin2018}, and formalized by~\citet{2022Mei}. For a fixed prediction model, as the number of features increase, the excess risk follows is $U$-shaped for underparametrized optimization models and goes down for overparametrized models. As for the excess rick, we observe a transition at $d \approx n$ as well as a better precision for overparametrized models. Contrary to the generalization error, underparametrized models $(d \ll n)$ performs well even when $\frac{d}{n} \rightarrow 0$ and are not $U$-shaped. We refer to this phenomenon as a transition phase for gradient and coordinate descent.
\begin{figure}[h]
    \centering
    \includegraphics[height=240pt]{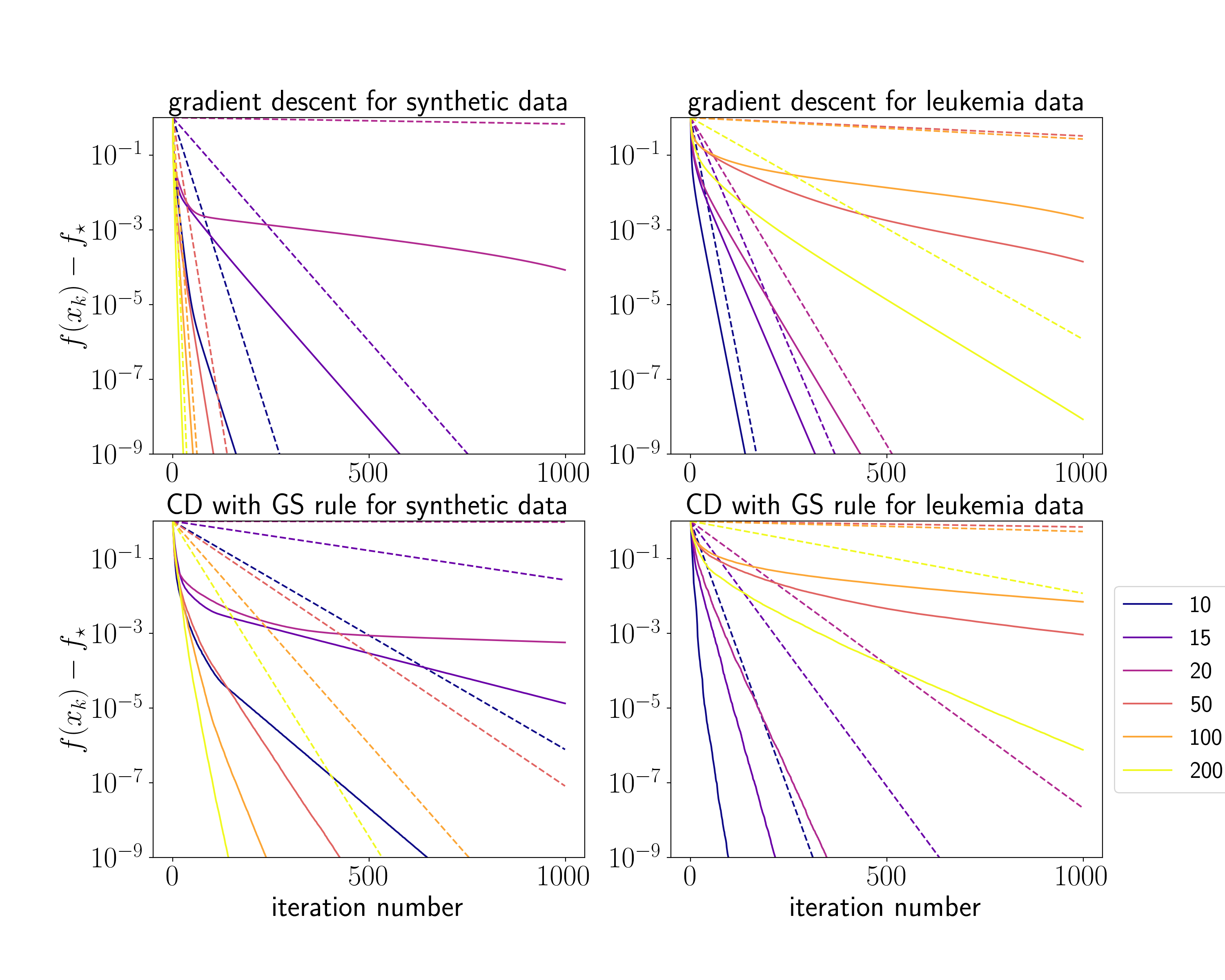}
    \vspace{-20pt}
    \caption{Convergence in function value for gradient descent and coordinate descent with GS rule, on synthetic quadratics $(n=20)$ and on the leukemia dataset ($n=72$), for different values for $d$. Dashed lines: comparison to the approximate convergence guarantees from~Corllary~\ref{under_over_corollary} for synthetic quadratics, and to high probability estimates for the leukemia dataset from Proposition~\ref{big_theo_approx}.}
    \label{fig:comparison_linreg_bound}
\end{figure}

In Figure~\ref{fig:comparison_linreg_bound}, we compare the exact and approximated upper bound to the convergence guarantee in function value for gradient descent and coordinate descent with the GS rule. For gradient descent, the theoretical approximation guarantee from Corollary~\ref{under_over_corollary} matches the observed convergence behavior of gradient descent. For steepest coordinate descent, we compare its convergence in function values to the exact upper bound obtained from the SDP relaxation in Proposition~\ref{Exact_approx_GS} for `small' values of $d$ and $n$, and to its high probability bound otherwise. In both cases, we numerically recover that convergence is improved as the dimension increases.

\subsection{Coordinate descent is an instance of matching pursuit}
\label{section_matching_pursuit}

Coordinate descent, as well as gradient descent, converge linearly both in the underparametrized and overparametrized regime, as provided by Proposition~\ref{conv_coordinate}, respectively due to strong convexity and the \L ojasiewicz property. Given a certain structure on $f$, we prove that $F(\cdot) = f(P\cdot)$ inherits some regularity properties from $f$ and that coordinate descent can be interpreted as a matching pursuit algorithm in either $\R^d$ or $\R^n$. Now, let us consider the more general formulation,
\begin{equation*}
    \min_{\alpha \in \R^d} F(\alpha) = f(P\alpha),
\end{equation*}
where $f$ is $L^f$-smooth, $\mu^f$-strongly convex and $F$ is $L_1^F$-smooth with respect to the $\ell_1$-norm.
\bigbreak
\noindent \textbf{Underparametrized regime.} $F$ is $\mu_1^F$-strongly convex (since $P^\top P$ is invertible). The connection between coordinate descent with GS-rule~\eqref{GS_coordinate} and matching pursuit was highlighted by~\citet{2018Locatello}. Considering the set of unitary direction $\mathcal{A} = {\rm conv}(\{\pm e_i, i=1,\ldots, d\})$, that is the $\ell_1$ unit ball, coordinate descent may be rewritten as a matching pursuit:
\begin{align*}
    e_{i_0} & \in - {\rm LMO}_{\mathcal{A}}(\nabla F(\alpha_0)) = - \argmin_{i= 1, \ldots, d} \nabla F(\alpha_0)^\top e_i, \\
    \alpha_1 &= \alpha_0 - \frac{1}{L_1^F}\nabla F(\alpha_0)^\top e_{i_0},
\end{align*}
where ${\rm LMO}_{\mathcal{A}}(\nabla F(\alpha_0)) = \inf_{z \in \mathcal{A}} \langle \nabla F(\alpha_0), z\rangle $. Steepest coordinate descent converges linearly from Proposition~\ref{conv_coordinate}, with the same convergence guarantee as in the context of matching pursuit~\cite[Theorem 5]{2018Locatello}.

\bigbreak
\noindent \textbf{Overparametrized regime.} $F$ does not inherit strong convexity. Yet, for least-squares, $F$ but does inherit some structure from $f$ (see Lemma~\ref{mu_PLF_formulation}). We prove that coordinate descent can be interpreted as a matching pursuit algorithm in $\R^n$. 
Recall the gauge function, for $x \in \R^n$, $\gamma_{\mathcal{P}}(x) = \inf_{\alpha \in \R^d, x = P\alpha} \|\alpha\|_1$. Lemma~\ref{norm_lemma} ensures~$\gamma_{\mathcal{P}}(\cdot)$ is a norm in the overparametrized regime.
\begin{lemma}
\label{norm_lemma}
Let $\alpha \in \R^d \rightarrow P \alpha \in \R^n$ be a surjection in $\R^d$, and $\mathcal{P} = {\rm conv}(P)$ be centrally symmetric. The function $\gamma_{\mathcal{P}}(\cdot)$ is a norm, and its dual norm is $\gamma^\star_{\mathcal{P}}(\cdot) = \sup_{s \in \mathcal{P}} \langle s, \cdot\rangle =\|P^\top\cdot\|_\infty$.
\end{lemma}
\begin{proof}
    See Appendix~\ref{proof_norm_dual}.
\end{proof}

Let $f$ be convex, $L_2^f$-smooth and $\mu_2^f$-strongly convex with respect to the $\ell_2$-norm. We define $L_\mathcal{P}^f = L_2^f\sup_{j=1, \ldots, d} \|P_j\|_2^2$ and $\mu_{\mathcal{P}}^f = \mu_2^f \inf_{z \in \R^n} \|P^\top z\|_{\infty}^2, \ {\rm such \ that } \ \|z\|_2^2 = 1$. Then, $f$ is convex, $L_\mathcal{P}^f$-smooth and $\mu_{\mathcal{P}}^f$-strongly convex with respect to the norm $\gamma_{\mathcal{P}}(\cdot)$. By definition, for all $ x \in \R^d, \ L_2^f\|x\|_2^2 \leqslant L_{\mathcal{P}}^f\gamma_{\mathcal{P}}(x)^2$ and $\mu_2^f \|x\|_2^2 \geqslant \mu_{\mathcal{P}}^f \gamma_{\mathcal{P}}(x)^2$. Thus $f$ is $L_\mathcal{P}^f$-smooth and $\mu_{\mathcal{P}}^f$-strongly convex with respect to $\gamma_{\mathcal{P}}$ (that is a norm in this regime).

As before, our estimates for smoothness (resp. strong convexity) parameters are obtained by an optimization problem. They appear to be closely related to the parameters for least-squares from Lemma~\ref{mu_F_formulation}~and~\ref{mu_PLF_formulation}. In the context of least-squares where $f(x) = \frac{1}{2n}\|x - y\|_2^2$ for $x \in \R^n$, we indeed have that $L_2^f = \mu_2^f = \frac{1}{n}$. For the $\ell_1$-norm, that $L_\mathcal{P}^f = L_1^F$ as defined in~\eqref{smoothness_F}, and $\mu_{\mathcal{P}}^f = \mu_{1, L}^{F}$ as soon as $PP^\top$ is invertible (which is the case here). Multiplying~\eqref{GS_coordinate} by $P$, noticing that $\min_{e, \|e\|_1 = 1} \langle \nabla F(\alpha), e \rangle = \min_{p \in \mathcal{P}} \langle \nabla f(x), p \rangle$, coordinate descent with the GS-rule on $F$ can be formulated as matching pursuit on $f$,
\begin{equation}
\begin{aligned}
    z_0 &\in {\rm LMO}_{\mathcal{P}}(\nabla f(x_0)), \\
    x_{1} &= x_0 - \frac{1}{L_\mathcal{P}^f} \langle \nabla f(x_0), z_0\rangle z_0.
\end{aligned}
\label{MP_over}
\end{equation}
Let $x_k$ be generated by matching pursuit \eqref{MP_over}, starting from $x_0 \in \R^n$ for $L_{\mathcal{P}}^f$-smooth and $\mu_{\mathcal{P}}^f$-strongly convex functions, then, Locatello~et~al~\cite[Theorem 5]{2018Locatello} proved linear convergence of the sequence with
\begin{equation}
\label{lin_conv_GS_overparam}
    f(x_k) - f_\star \leqslant \left(1 - \frac{\mu_{\mathcal{P}}^f}{L_{\mathcal{P}}^f}\right)(f(x_0) - f_\star).
\end{equation}
By construction, since $x_k = P\alpha_k$, we have that $F(\alpha_k) - F_\star \leqslant (1 - \frac{\mu_{\mathcal{P}}^f}{L_{\mathcal{P}}^f})(F(\alpha_0) - F_\star) = (1 - \frac{\mu_{1, L}^{F}}{L_1^F})(F(\alpha_0) - F_\star)$. The same result could have been derived from  Proposition~\ref{conv_coordinate} and the observation that strongly convex functions composed with a linear mapping verify a \L ojasiewicz-inequality.

\begin{lemma}
\label{PL_F_from_f}
    Let $f$ be $\mu_\mathcal{P}^f$-strongly convex with respect to the norm $\gamma_{\mathcal{P}}(\cdot)$. Then, $F$ verifies a \L ojasiewicz inequality with parameters $\mu_\mathcal{P}^f$, that is for all $\alpha \in \R^d$,
\begin{equation*}
    \frac{1}{2}\|\nabla F(\alpha)\|_{\infty} \geqslant \mu_{\mathcal{P}}^f(F(\alpha) - F_\star).
\end{equation*}
\end{lemma}
\begin{proof}
Let $x \in \R^n$, $f_\star \geqslant f(x) - \sup_{z} \langle -\nabla f(x), y-x \rangle - \frac{\mu_\mathcal{P}^f}{2} \gamma_{\mathcal{P}}^2(y-x) \geqslant f(x) - $ \\ $(\frac{\mu_\mathcal{P}^f}{2}\gamma_{\mathcal{P}}^2(\cdot))^\star(-\nabla f(x)) \geqslant f(x) - \frac{1}{2 \mu_\mathcal{P}^f} \|P^\top \nabla f(x)\|_{\infty}^2.$ Since $F(\cdot) = f(P\cdot)$, the inequality is obtained by taking $x = P\alpha$ and since $\nabla F(\alpha) = P^\top \nabla f(P\alpha)$.
\end{proof}

Lemma~\ref{PL_F_from_f} corresponds to the result of Karimi~et~al.~\cite[Appendix B]{2016Karimi}, that let a Hoffman constant appear (that is in their context equal to the smallest non-zero eigenvalue of $P$), as defined in~\cite[Section 3 and 4.1]{2015Necoara} by $\theta(P) = \max_{z, \|P^\top z\|_{\infty} = 1} \|z\|_2^2$. They indeed proved that $F$ verifies a \L ojasiewicz inequality, for all $\alpha \in \R^d$, $\frac{1}{2}\|\nabla F(\alpha)\|_2^2 \geqslant \theta(P)\mu^F (F(\alpha_k) - F_\star)$. 

\bigbreak
Depending on the parametrization regime, we have proven that coordinate descent may be formulated as a (possibly rebased) matching pursuit method. In the underparametrized regime on the one hand, since $F$ inherits all regularity properties from $f$, the atoms are defined by the Euclidean basis and the matching pursuit is formulated in $\R^d$. On the other hand in the overparametrized regime, the introduction of a well-chosen gauge function $\gamma_{\mathcal{P}}$ allows to formulated coordinate descent as a matching pursuit algorithm in $\R^n$, and to perform a convergence analysis using the strong convexity assumption on $f$. Again, global values of the smoothness and strong-convexity parameters can be formulated as optimization problems depending on the gauge. The gauge let also appear how $F$ inherits some structure from $f$. In the next sections, we generalize this framework for analyzing penalized linear models.

\section{Transition phase for penalized linear models}
\label{sec:regularized_models}
We now consider the penalized linear model,
\begin{equation}
\label{penalized_model}
    \min_{\alpha \in \R^d} \{G(\alpha) \triangleq f(P\alpha) + \lambda \|\alpha\|_1\},
\end{equation}
where $f: \R^n \rightarrow \R$ is $L_2^f$-smooth, $\mu_2^f$-strongly convex, (and thus, the function $F: \R^d \rightarrow \R$ such that $F(\cdot) = f(P\cdot)$, is $L_2^F$ smooth), $P\in \R^{n \times d}$, $\lambda > 0$ and where $H(\alpha) = \lambda\|\alpha\|_1$ is closed convex and proper. In this section, we derive a new matching pursuit algorithm for a $\ell_1$-regularized model, that we compare to proximal coordinate descent with GS rule and to the proximal gradient descent. Building on the results of Section~\ref{sec:linear_regression}, we derive convergence guarantees depending on the properties of $f$ and~$P$, and notice a strong connection to the proximal coordinate descent with GS rule. Yet, in the overparametrized regime, neither the proximal gradient nor the regularized matching pursuit benefits from linear convergence. Instead of that, we describe experimentally the role of $\lambda$ in the LASSO, as a continuous mapping between low-rank solutions and full-rank solution to the least-squares.

\smallbreak
\noindent \textbf{Proximal gradient descent.} Proximal gradient descent, a.k.a. forward-backward (see e.g. ~\citep{2005Combettes}) was developed for such `composite' convex optimization problems. Given a starting point $\alpha_0 \in \R^d$, each iterate is obtained by minimizing a smooth quadratic upper bound on $F$:
\begin{align}
\label{eq:prox_grad}
    G(\alpha) \leqslant F(\alpha_{k}) + \langle \nabla F(\alpha_{k}), \alpha - \alpha_k \rangle + \frac{L_2^F}{2}\|\alpha_k - \alpha\|_2^2 + \lambda \|\alpha\|_1.
\end{align}
Minimizing the right side of the inequality yields the proximal gradient method as follows:
\begin{align*}
    \alpha_{k+1} &= \  \argmin_{\alpha \in \R^d} \  \langle \nabla F(\alpha_k), \alpha - \alpha_k\rangle + \frac{L_2^F}{2}\|\alpha - \alpha_k\|_2^2 + \lambda \|\alpha\|_1.
\end{align*}
The proximal gradient method converges sublinearly if $F$ is smooth and convex, and linearly if $F$ is in addition $\mu_2^F$-strongly convex, such as in the underparametrized regime. Then, the sequence $\alpha_k$ starting from $\alpha_0 \in \R^d$ verifies (\cite[Theorem 2.1]{Taylor2018}),
\begin{equation*}
        G(\alpha_k) - G_\star \leqslant \left(1 - \frac{\mu_2^F}{L_2^F}\right)^{k}(G(\alpha_0) - G_\star).
    \end{equation*}

\smallbreak
\noindent \textbf{Coordinate descent.} In practice, (randomized) coordinate gradient descent is widely used to avoid computing the full gradient (that costs $O(d)$), and is particularly suited to sparse regression problems. \citet{2018Nutini} analyzed coordinate descent with the Gauss-Southwell selection rule, that tends to perform better than randomized coordinate descent.
\begin{align}
\label{eq:coordinate_GS_rule}
        \alpha_{k+1} &= \argmin_{\alpha \in \R^d} \ \nabla_{i_k} F(\alpha_k) (\alpha^{(i_k)} - \alpha_k^{(i_k)}) + \frac{L_2^F}{2}(\alpha^{(i_k)} - \alpha_k^{(i_k)})^2 + \lambda |\alpha^{(i_k)}|,
\end{align}
where $i_k = \argmin_l \min_{t \in \R} \ \nabla_{l} F(P\alpha_k) (t - \alpha^{(l)}) + \frac{L_2^F}{2} (t - \alpha^{(l)})^2 + \lambda |t|$ corresponds to the GS rule. Nutini~et~al.~\cite[Appendix 8]{2018Nutini} proved that coordinate descent with the Gauss-Southwell rule makes at least as much progress as randomized coordinate descent,
\begin{equation*}
    G(\alpha_{k+1}) - G_\star \leqslant \left(1 - \frac{\mu_2^F}{d L_2^F}\right)(G(\alpha_k) - G_\star).
\end{equation*}
A refinement, that let a sublinear dependence in the parameter $\mu_1^F$ appear, is mentioned in~\cite[Appendix 8]{2018Nutini}. Coordinate descent with GS-rule is closely related to matching pursuit as for nonpenalized models, as detailed in~Appendix~\ref{GS_matching_pursuit_appendix}, where we formulate this method as a `nearly' matching pursuit algorithm.

In the following, we derive a matching pursuit procedure for $\ell_1$-regularized problems~\eqref{penalized_model}, that we compare to classical boosting algorithm and coordinate descent with GS-rule. After that, we compute convergence guarantees for smooth (possibly) strong convex functions. Finally, we interpret the convergence regimes as a function of the penalty $\lambda$.

\subsection{Regularized matching pursuit}
\label{sec:section_regu_MP}

We propose a new regularized matching pursuit algorithm based on the $\ell_1$-geometry. The main idea is to replace the $\ell_2$-norm in the minimization Problem~\eqref{eq:prox_grad} leading to the proximal gradient by a $\ell_1$-norm. Let $F$ be convex, $L_1^F$-smooth, as for coordinate descent with GS rule in the linear regression problem from Section~\ref{sec:linear_regression}. We define the penalized matching pursuit method starting from $\alpha_0 \in \R^d$ as the sequence minimizing smoothness with respect to the $\ell_1$-norm at each iteration:
\begin{equation}
\label{L_1_smooth_sotopo}
    \alpha_{k+1} = \argmin_{\alpha \in \R^d}  \ \langle P^\top \nabla f(P\alpha), \alpha - \alpha_k\rangle + \frac{L_1^F}{2}\|\alpha - \alpha_k\|_1^2 + \lambda \|\alpha\|_1.
\end{equation}
Whereas the optimization steps in proximal gradient descent~\eqref{eq:prox_grad} and proximal coordinate descent with the GS rule~\eqref{eq:coordinate_GS_rule} can be decomposed coordinate-wise, the function $\alpha \rightarrow \|\alpha\|_1^2$, that is not separable. Based on the same upper bound~\eqref{L_1_smooth_sotopo},~\citet[Algorithm 1]{Song2017} generalized greedy coordinate descent with the ``SOft ThresOlding PrOjection'' (SOTOPO) algorithm using a reweighted least-squares formulation~(Appendix~\ref{ap:math_tools}). However, their methods is neither a coordinate-based method, nor a boosting method. We propose instead a regularized matching pursuit algorithm that draws a clean connection to boosting and proximal coordinate descent.

In the following, we formulate this optimization step~\eqref{L_1_smooth_sotopo} as a matching pursuit algorithm, that only calls for a linear minimization oracle. Using a variational trick detailed in~Appendix~\ref{ap:math_tools} to approach $\|\beta\|_1^2$, we begin by formulating Problem~\eqref{L_1_smooth_sotopo} starting from $\alpha_k \in \R^d$ as a separable optimization problem,
\begin{align*}
    V_\star &=\min_{\beta \in \R^d} \ \langle \nabla F(\alpha_k), \beta\rangle + \frac{L_1^F}{2}\|\beta\|_1^2 + \lambda \|\beta + \alpha_k\|_1, \\ 
    &= \min_{\beta \in \R^d} \max_{z \geqslant 0} \ \langle \nabla F(\alpha_k), \beta\rangle -\frac{z^2}{2L_1^F} + z\|\beta\|_1 + \lambda \|\beta + \alpha_k\|_1, \\
    &= \max_{z \geqslant 0} \ -\frac{z^2}{2L_1^F} + \min_{\beta \in \R^d}\sum_{i=1}^d \{\nabla_i F(\alpha_k)\beta^{(i)}  + z|\beta^{(i)}| + \lambda |\beta^{(i)} + \alpha_k^{(i)}|\}.
\end{align*}
At the optimum, $z = L_1^F\|\beta\|_1$. The problem is now separable in each coordinate~$\beta^{(i)}$, and can be reduced to an optimization problem in $z \geqslant 0$ in Lemma~\ref{lemma:final_V}.

\begin{lemma}
\label{lemma:final_V}
    The optimization step~\eqref{L_1_smooth_sotopo} can be reformulated as
    \begin{equation}
\label{eq:penalized_MP_each_step}
   V_\star = \max_{z_{\min} \leqslant z} \ \left\{ h(z) \triangleq  -\frac{z^2}{2L_1^F} + \sum_{i\in I} \min\left(\lambda |\alpha_k^{(i)}|, -\nabla_i F(\alpha_k) \alpha_k^{(i)} + z |\alpha_k^{(i)}|\right)\right\}, 
\end{equation}
where $z_{\min} = (\max_i|\nabla_i F(\alpha_k)| - \lambda)_+$ and where $I = \{i, \alpha_k^{(i)} \neq 0\}$ is the set of active atoms.
\end{lemma}
\begin{proof}
    The function $\phi_i(z, \beta^{(i)}) = \nabla_i F(\alpha_k)\beta^{(i)}  + z|\beta^{(i)}| + \lambda |\beta^{(i)} + \alpha_k^{(i)}|$ is lower bounded if $z \geqslant |\nabla_i F(\alpha_k)| - \lambda$ for all $i \in I$. Then, $\phi_i(z, \cdot)$ attains its minima in $\beta^{(i)} = 0$ with $\phi_i(z, 0) = \lambda |\alpha_k^{(i)}|$ or in $\beta^{(i)} = -\alpha_k^{(i)}$ with $\phi_i(z, -\alpha_k^{(i)})= -\nabla_i F(\alpha_k)\alpha_k^{(i)} + z |\alpha_k^{(i)}|$ or in every possible value for $\beta^{(i)}$ if $z = \pm \nabla_i F(\alpha_k) - \lambda$ with $\phi_i(z, 0) = \lambda |\alpha_k^{(i)}|$.  
\end{proof}
Lemma~\ref{lemma:final_V} leads to a convex constrained optimization problem in $\R^+$, whose objective is the sum of a quadratic and piecewise linear functions whose slope coefficients changing at $ z_i = \lambda + \nabla_i F(\alpha_k)\alpha_k^{(i)} / |\alpha_k^{(i)}|$. Its constraints in $z_{\min}$ includes the LMO. By construction, the LMO given by $z_{\min}$ may correspond to several atoms $\beta_{j}$ such that $j \in \argmin_{i} |\nabla_i F(\alpha_k)|$. Since we aim at solving Problem~\eqref{L_1_smooth_sotopo} by constructing a solution as sparse as possible, we introduce Assumption~\ref{sparsity_assumption}.
\begin{assumption}
    \label{sparsity_assumption} The algorithm only selects one atom corresponding to the LMO, that is $i_{\min} \in \argmax_i |\nabla_i F(\alpha_k)|$, such that  $z_{\min} = (\max_i|\nabla_i F(\alpha_k)| - \lambda)_+$.
\end{assumption}
In the following lemmas, we compute the minimum of $h$ explicitly. Under Assumption~\ref{sparsity_assumption}, Lemma~\ref{lem:adding_direction} first deals with the situation in which the objective is quadratic, that is for all $i \in I$, $z_i \geqslant z_{\min}$. 
\begin{lemma}
\label{lem:adding_direction}
    Let $(z_\star, \beta_\star)$ be a solution to~\eqref{eq:penalized_MP_each_step}, assume $\{i, z_i \geqslant z_{\min}\} = \emptyset$ and verify Assumption~\ref{sparsity_assumption}. Then $z_\star = z_{\min}$, $\beta^{(i_{\min})}_{\star} = -{\rm sign}(\nabla_{i_{\min}} F(\alpha_k))\frac{z_{\min}}{L_1^F}$ and $\beta^{(i)}_{\star} = 0$ for $i \neq i_{\min}$.
\end{lemma}
\begin{proof}
    The objective is quadratic and attains its minimum at $z_{\min} = L_1^F |\beta^{(i_{\min})}|$. 
\end{proof}
In the context of Lemma~\ref{lem:adding_direction} and Assumption~\ref{sparsity_assumption}, only the atom given by the LMO in $z_{\min} = (\max_i|\nabla_i F(\alpha_k)| - \lambda)_+$ can be added to the set of active atoms. Now, we assume the objective is piecewise quadratic, that is $\mathcal{S} = \{i, z_i \geqslant z_{\min}\} \neq \emptyset$. 

\begin{lemma}
    \label{cases_regularized_matching_pursuit}
    Let $z_\star, \beta_\star$ be a solution of~\eqref{eq:penalized_MP_each_step} and assume $\mathcal{S} = \{i, z_i \geqslant z_{\min}\} \neq \emptyset$ and verify Assumption~\ref{sparsity_assumption}. There are four possible solutions to Problem~\eqref{eq:penalized_MP_each_step},
    \begin{itemize}
        \item If $h'(z_{\min}) \leqslant 0$, then $z_{\star} = z_{\min}$. \\ In addition, $\beta_{\star}^{(i)} = \left\{
    \begin{array}{ll}
        -\alpha_k^{(i)} & \mbox{ if } z_i \geqslant z_{\star}, \\
        0 & \mbox{ if } z_i \leqslant z_{\star}, \\
        -{\rm sign}(\nabla_{i_{\min}} F(\alpha_k))\frac{z_{\min} - \sum_{i \in \mathcal{S}} |\alpha_k^{(i)}|}{L_1^F} & \mbox{ if } i = i_{\min}.
    \end{array}\right.$
        \item If there exists $k \in \mathcal{S}$ such that $h'(z_k^+) \geqslant 0$ and $h'(z_{k+1}^-) \leqslant 0$, then $z_{\star} \in ]z_{k}, z_{k+1}[$. In addition, 
        $\beta_{\star}^{(i)} = \left\{\begin{array}{ll}
        -\alpha_k^{(i)} & \mbox{ if } z_i \geqslant z_{\star}, \\
        0 & \mbox{ if } z_i \leqslant z_{\star}.
    \end{array} \right.$
        \item If there exists $k \in \mathcal{S}$ such that $h'(z_k^{-}) \geqslant 0$ and $h'(z_k^{+}) \leqslant 0$ then $z_{\star} = z_{k}$. In addition, \\
        $\beta_{\star}^{(i)} = \left\{ \begin{array}{ll}
        -\alpha_k^{(i)} & \mbox{ if } z_i > z_{\star}, \\
        0 & \mbox{ if } z_i < z_{\star}, \\
       -{\rm sign}(\alpha_k^{(i)})\left(\frac{z_k}{L_1^F} - \sum_{i, z_i > z_{k}}|\alpha_k^{(i)}| \right) & \mbox{ if } i = k.
    \end{array}\right.$
        \item If $h'(z_{|I|}) > 0$, then for all $i \in I$, $z_{\star} > z_i$ and
        $\beta^{(i)}_{\star} = \left\{\begin{array}{ll}
        - {\rm sign}(\alpha_k^{(i)})\frac{z_i}{L_1^F} & \mbox{ if } i = |I|, \\
        0 & \mbox{ otherwise.}
    \end{array} \right.$
    \end{itemize}
\end{lemma}
\begin{proof}
The function $h$ is strictly concave and piecewise quadratic on $[z_i, z_{i+1}]$. The solution to the optimization Problem~\eqref{eq:penalized_MP_each_step} is thus obtained by studying the sign of $h'(\cdot)$ at $z_{i}^-$ and $z_{i}^+$. By construction of the solution given in the proof of Lemma~\ref{lemma:final_V}, for all $i$ such that $ z_\star > z_i$ (resp.~ $z_\star < z_i$), then $\beta^{(i)} = 0$ (resp.~ $\beta^{(i)} = -\alpha_k^{(i)}$). Finally, we have $z_{\star} = L_1^F \|\beta_{\star}\|_1$ which gives the solution for $z_\star = z_{\min}$ or $z_{\star} = z_k$.
\end{proof}

Lemmas~\ref{lem:adding_direction}~and~\ref{cases_regularized_matching_pursuit} provides a closed form solution by calling only for the linear minimization oracle $\min_i |\nabla_i F(\alpha_k)|$, and performing $O(|I|)$ operations on the active atoms. From that, we deduce Algorithm~\ref{Reguralized_MP}. 
\begin{algorithm}
\caption{Regularized matching pursuit (RMP)}
\label{Reguralized_MP}
\begin{algorithmic}
\STATE{$\alpha \in \mathbf{R}^d$, $N \in \mathbf{N}$}
\FOR{$k \in [0, \ldots, N]$}
\STATE{$z_{\min} = (\max_i |\nabla_iF(\alpha_k)| - \lambda)_+ $ and $ i_{\min} = \argmax_i |\nabla_{i} F(\alpha_k)| $}
\STATE{For $\alpha_k^{(i)} \neq 0$, compute $z_i = \lambda + \frac{\alpha_k^{(i)}}{|\alpha_k^{(i)}|}\nabla_iF(\alpha_k)$ such that $z_{i+1} \geqslant z_i$}
\IF{$\{i, z_i \geqslant z_{\min}\} = \emptyset$}
\STATE{$\beta^{(i_{\min})} = -\text{sign}(\nabla_{i_{\min}} F(\alpha_k))\frac{z_{\min}}{L_1^F}$}
\ELSE
\STATE{Compute $u = \text{argmin}_i \{z_i \geqslant z_{\min}\}$ and for $i \in [u, v]$, compute $h'(z_i)$}
\IF{$h'(z_{\min}) \leqslant 0$ or $h'(z_u) \leqslant 0$}
\STATE{\textbf{For} $i \in [u, v]$, $\beta^{(i)} = -\alpha_k^{(i)}$}
\STATE{\textbf{If} $h'(z_{\min}) \leqslant 0$, \textbf{then} $\beta^{(i_{\min})} = -\text{sign}(\nabla_{i_{\min}} F(\alpha_k))(\frac{z_{\min}}{L_1^F} - \sum_{i=u}^{v} |\alpha_k^{(i)}|)$}
\ELSE
\STATE{$n = \argmax\{i, i \in [u, v-1], \ h'(z_i^+), h'(z_{i+1}^-) \geqslant 0\}$}
\IF{$n = v-1$}
\STATE{$\beta^{(v)} = -\text{sign}(\alpha_v)\frac{1}{L_1^F}(\lambda + \frac{\alpha_v}{|\alpha_v|} \nabla_v F(\alpha_k))$}
\ELSE
\STATE{\textbf{For} $i \in [n+1, v]$, $\beta^{(i)} = -\alpha_k^{(i)}$}
\STATE{\textbf{If} $h'(z_n^{+}) \leqslant 0$, \textbf{then} $\beta^{(n)} = -\text{sign}(\alpha^{(n)}_k)(\frac{1}{L_1^F}\left(\lambda + \frac{\alpha^{(n)}_k}{|\alpha^{(n)}_k|} \nabla_n F(\alpha_k)\right) - \sum_{i=n+1}^v |\alpha_k^{(i)}|)$}
\ENDIF
\ENDIF
\ENDIF
\STATE{$\alpha_{k+1} = \alpha_k + \beta$}
\ENDFOR
\end{algorithmic}
\end{algorithm} 
In short, at each iteration, Algorithm~\ref{Reguralized_MP} performs one of the three possible actions: either one new atom is added (at most) by calling the ${\rm LMO}(\nabla F(\alpha_k))=$ $\argmax_{i}|\nabla_i F(\alpha_k)|= \argmax_{p \in \mathcal{P}}p^\top \nabla f(P\alpha_k)$ while some active atoms may be set to zero, or one active atom may be optimized while some active atoms may be set to zero, or some active atoms are set to zero (but none is added nor optimized). To sum it up, at each iteration, it constructs the next iterate using only past active atoms plus possibly a new one generated by the LMO. Therefore, Algorithm~\ref{Reguralized_MP} belongs to the family of boosting algorithms. We refer to it as the regularized matching pursuit (RMP). 

Compared to the boosting approach of \citet{2012Zhang_matrixboosting} for metric-norm regularization or to the generalized conditional gradient~\citep{Bach2015, Sun2020_1}, active atoms are not modified uniformly since only some of them may be reduced to zero. The SOTOPO method of \citet{Song2017} minimizes the same upper bound with respect to the $\ell_1$-norm~\eqref{L_1_smooth_sotopo}. Yet, it is resolved with a different variational formulation, that does not let a linear minimization oracle appear. Compared to proximal coordinate descent with GS-rule~\eqref{eq:coordinate_GS_rule} applied with $L_1^F$ (instead of $L_2^F$), the regularized matching pursuit happens to often follow exactly the same path when starting from zero (but not when starting from a nonzero point), as observed in Figure~\ref{fig:sotopo_gsq_comparison}. This suggests a connection between regularized matching pursuit and proximal coordinate descent, as proven by \citet{2018Locatello} for gradient descent and steepest coordinate descent.

 \begin{figure}[h]
\begin{subfigure}[h]{.32\textwidth}
  \centering
  \includegraphics[height=95pt]{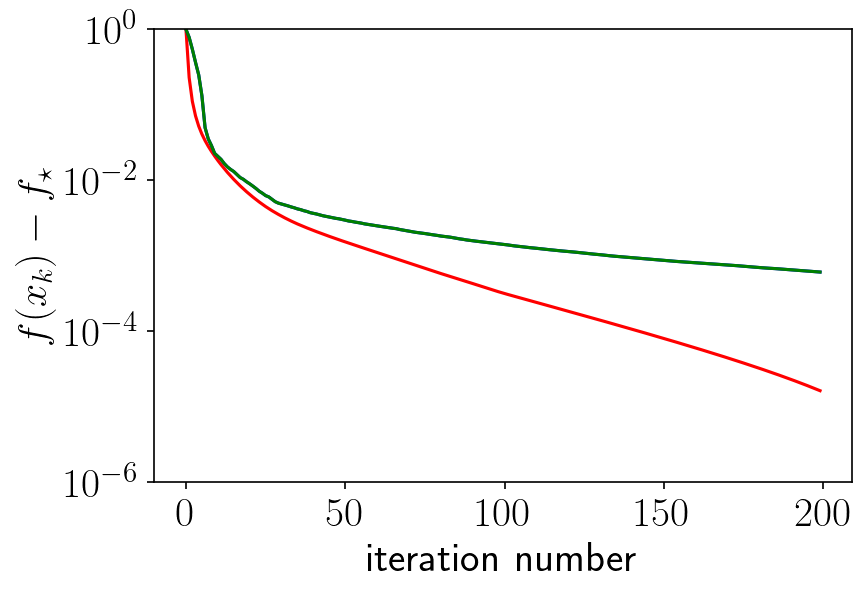}
\end{subfigure}
  \hfill
  \begin{subfigure}[h]{.32\textwidth}
  \centering
  \includegraphics[height=95pt]{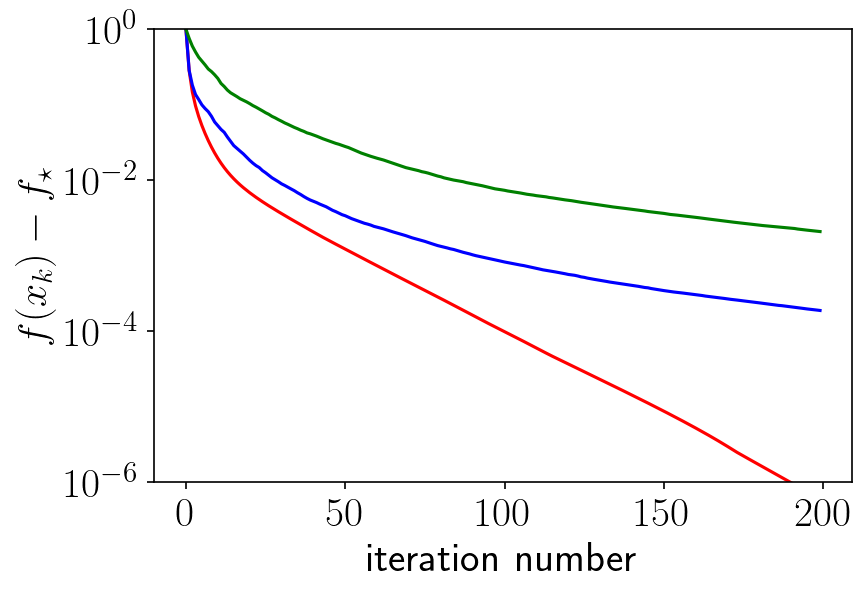}
\end{subfigure}
  \hfill
\begin{subfigure}[h]{.32\textwidth}
  \centering
  \includegraphics[height=95pt]{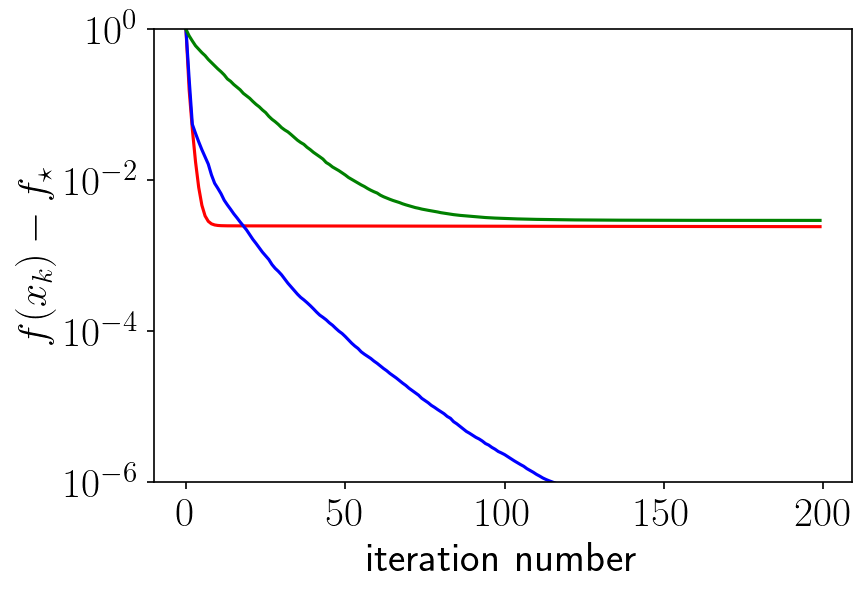}
\end{subfigure}
\caption{Convergence in function value of the proximal gradient descent, coordinate descent with Gauss-Soutwhell rule and with $L=L_1^F$ (instead of $L=L_2^F$) and of the regularized matching pursuit, for synthetic quadratics (see Section~\ref{sec:random_features}) with $n=50$, $s=8$, $\lambda = 0.001$, $\sigma = 0.5$ and for $d=30$ starting from zero on the left (underparametrized regime), from a non zero point in the middle (underparametrized regime) and for $d=500$ on the right (overparametrized regime). RMP and coordinate descent with GS-rule matche exactly in these examples.}
\label{fig:sotopo_gsq_comparison}
\end{figure}

In Figure~\ref{fig:sotopo_gsq_comparison}, the RMP appears to converge linearly in the underparametrized regime, and sublinearly in the overparametrized regime. We compute some convergence guarantees in the next section.

\subsection{Convergence guarantee}

We now establish convergence guarantees for the RMP, both for strongly convex and non-strongly convex functions. We consider a more general composite minimization problem, \begin{equation}
\label{composite_min_problem}
    \min_{\alpha \in \R^d} \left\{G(\alpha) \triangleq F(\alpha) + H(\alpha)\right\},
\end{equation}
where $H$ is closed, convex, proper, and where $F$ is $L_1^F$-smooth and (possibly) $\mu_1^F$-strongly convex with respect to the $\ell_1$-norm. If in addition, $F$ is a linear mapping, and $H(\cdot) = \|\cdot\|_1$, this is exactly the original optimization Problem~\eqref{eq:problem}. We evaluate the convergence guarantee a generalized version of the RMP (the GRMP), that is not always a boosting method,\begin{equation}
\label{prox_1_algo}
    \alpha_{k+1} = \argmin_{\alpha \in \R^d} \  \langle \nabla F(\alpha_{k}), \alpha - \alpha_k \rangle + \frac{L_1^F}{2}\|\alpha - \alpha_k\|_1^2 + H(\alpha).
\end{equation}
As we will see, our proofs are similar to those for randomized coordinate descent~\cite[Theorem 5, 7]{2014Richtarik}.

\subsubsection{Strongly convex functions}

Let us assume that $F$ is $L_1^F$-smooth and $\mu_1^F$-strongly convex, typically in the underparametrized regime. Similarly to coordinate gradient descent with GS rule which converges linearly in this context~\citep{2018Nutini}, regularized matching pursuit is formulated as the minimization of the smoothness upper bound with respect to the $\ell_1$-norm. Therefore, it benefits from linear convergence guarantees, detailed below.

\begin{proposition}{\citep[Appendix A.8]{Nutini_2018thesis}}
\label{regularized_MP_proof_strconv}
If $F$ be convex, $L_1^F$-smooth with respect to the $\ell_1$-norm, and $\mu_1^F$-strongly convex with respect to the $\ell_1$-norm. Then, the sequence $(\alpha_k)$ generated by~\eqref{prox_1_algo} verifies,
\begin{equation*}
    G(\alpha_{k+1}) - G_\star \leqslant \left(1 - \frac{\mu_1^F}{L_1^F}\right)(G(\alpha_k) - G_\star).
\end{equation*}
\end{proposition}
\begin{proof}
    The proof is taken from \citet[Appendix~A.8.]{Nutini_2018thesis} and consists in an optimization step over all trajectories. The argument is inspired from randomized coordinate descent~\citep{2014Richtarik}).
\end{proof} 

 The RMP is a special case of method~\eqref{prox_1_algo}, and verifies the convergence guarantee of Proposition~\ref{regularized_MP_proof_strconv}. As a conclusion, it beats traditional boosting techniques converging sublinearly, such as coordinate descent with GS rule (with $1- \frac{\mu_2^F}{dL_2^F} \leqslant 1- \frac{\mu_1^F}{L_1^F}$), or the generalized conditional gradient that is also adapted to a gauge geometry. In addition, its linear guarantee only depends on the strong convexity and smoothness parameters of $F$ with respect to the $\ell_1$-norm. In the special case of the LASSO, the estimates established in Proposition~\ref{Exact_approx_GS}~and~\ref{big_theo_approx} still apply. In the overparametrized regime however, Figure~\ref{fig:sotopo_gsq_comparison} suggests that the method does not converge linearly (since it is stuck at an accuracy of about around $10^{-5}$).

\subsubsection{Smooth convex functions}

Let now $F$ be $L_1^F$-smooth, convex, but not strongly convex (which is verified in the overparametrized regime). Usually, guarantees for splitting methods, such as proximal gradient, states a sublinear convergence guarantee. Similarly in Proposition~\ref{sublinear_conv_regu_matching_pursuit}, we prove sublinear convergence for the GRMP. To our knowledge, there is no such result for sublinear convergence for SOTOPO~\citep{Song2017} or for coordinate descent with the Gauss-Southwell rule, which is very close to the GRMP.

\begin{proposition}
\label{sublinear_conv_regu_matching_pursuit}
    Let $(\alpha_k)$ be generated by the generalized regularized matching pursuit~\eqref{prox_1_algo}, starting from $\alpha_0 \in \R^d$
    \begin{equation*}
    G(\alpha_{k}) - G_\star \leqslant \frac{2L_1^F \mathcal{R}_{\alpha_0}^2}{k+1},
\end{equation*}
where $\mathcal{R}_{\alpha_0}^2 = \max_{\alpha \in \R^d} \max_{\alpha_{\star} \in \R^d} \{\|\alpha - \alpha_{\star}\|_1^2, \ {\rm s.t. }\ G(\alpha) \leqslant G(\alpha_0)\}$.
\smallbreak
\end{proposition}
\begin{proof}
This technique is inspired from a proof for sublinear convergence of randomized proximal coordinate descent established by Richtarik~and~Takac~\cite[Theorem 5]{2014Richtarik}. Let $\alpha_{k+1} \in \R^d$ be a minimizer of the smooth upper bound:
\begin{align*}
    G(\alpha_{k+1}) &\leqslant \inf_{\alpha \in \R^d} F(\alpha_k) + \langle \nabla F(\alpha_k), \alpha - \alpha_k \rangle + \frac{L_1^F}{2}\|\alpha - \alpha_k\|_1^2 + H(\alpha), \\
    & \leqslant \inf_{\alpha \in \R^d} F(\alpha) + H(\alpha) + \frac{L_1^F}{2}\|\alpha - \alpha_k\|_1^2 \ (= G(\alpha) + \frac{L_1^F}{2}\|\alpha - \alpha_k\|_1^2)\ (F {\rm\ convex }), \\
    & \leqslant \inf_{t \in [0, 1]} G(t\alpha_\star + (1 - t)\alpha_k) + \frac{L_1^F t^2}{2}\|\alpha_k - \alpha_\star\|_1^2, \\
    & \leqslant \inf_{t \in [0, 1]} G(\alpha_k) - t(G(\alpha_k) - G_\star) + \frac{L_1^F t^2}{2}\|\alpha_k - \alpha_\star\|_1^2 \ ({\rm convexity} \ {\rm of} \ H, F),\\
    G(\alpha_{k+1}) - G_\star &\leqslant \inf_{t \in [0, 1]} (1 - t)(G(\alpha_k) - G_\star) + \frac{L_1^F t^2}{2}\|\alpha_k - \alpha_\star\|_1^2.
\end{align*}
The solution of this minimization problem is given by $t_\star = \min(1, \frac{G(\alpha_k) - G_\star}{L_1^F \|\alpha_k - \alpha_\star\|_1^2})$. We conclude the minimization bound, depending on the sign of $G(\alpha_k) - G_\star - L_1^F \|\alpha_k - \alpha_\star\|_1^2$:
\begin{equation*}
    G(\alpha_{k+1}) - G_\star \leqslant \max\left(1 - \frac{G(\alpha_k) - G_\star}{2L_1^F \|\alpha_k - \alpha_\star\|_1^2}, \frac{1}{2}\right)(G(\alpha_k) - G_\star).
\end{equation*}
As a first conclusion, notice that $G(\alpha_k) - G_\star$ is nonincreasing. Recall now that $\mathcal{R}_{\alpha_0}^2 = \max_{\alpha \in \R^d} \max_{\alpha_{\star} \in \R^d} \{\|\alpha - \alpha_{\star}\|_1^2, \ {\rm s.t. }\ G(\alpha) \leqslant G(\alpha_0)\}$. Then, using the notation $\delta_k = G(\alpha_k) - G_\star$, an upper bound for $\delta_{k+1}$ is given by $\delta_{k+1} \leqslant \max\left( 1 - \frac{\delta_k}{2L_1^F\mathcal{R}_{\alpha_0}^2},  \frac{1}{2}\right)\delta_k$.
Assume now that $\delta_0 \leqslant L_1^F \mathcal{R}_{\alpha_0}^2$ and notice that $\delta_k \leqslant L_1^F \mathcal{R}_{\alpha_0}^2$ since $\delta_k$ is nonincreasing. If not, notice that the inequality satisfied at the next iteration $\delta_1 \leqslant \frac{1}{2}\mathcal{R}_{\alpha_0}^2$. Then, we have for $\omega = \frac{1}{2L_1^F\mathcal{R}_{\alpha_0}^2}$, $\delta_{k+1} \leqslant (1 - \delta_k\omega)\delta_k$. Following the same argument as in the proof for sublinear convergence of steepest coordinate descent, detailed in~Appendix~\ref{ap:matching_pursuit_gauge}, we arrive to a convergence guarantee $G(\alpha_{k}) - G_\star \leqslant \frac{2L_1^F \mathcal{R}_{\alpha_0}^2}{k+1}$.
\end{proof}

Proposition~\ref{sublinear_conv_regu_matching_pursuit} provides a sublinear convergence guarantee for the GRMP for non-strongly convex functions. To our knowledge, this is the first sublinear guarantee for a boosting algorithm under classical assumptions from convex optimization. This method does not benefit from linear convergence guarantee. Yet, as we see from the numerical experiments in the next section that the RMP does converge linearly in certain regimes in the case of the $\ell_1$-regularized model.

\subsection{A transition phase depending on $\lambda$: experimental results}

The RMP algorithm benefits from convergence guarantees similar to those for the proximal gradient: these methods converge linearly under strong convexity assumptions (underparametrized regime) but have sublinear guarantees for smooth convex problems (overparametrized regime). In the context of sparsity though, the proximal gradient descent benefits from linear convergence under additional assumptions on the problem classes such as restricted eigenvalue properties~\citep{raskutti2010}, and for a well-chosen parameter $\lambda$~\cite[Theorem 2]{2010agarwal}. In this section, our experiments reveals a transition phenomenon driven by $\lambda$ on a LASSO problem  $F(\alpha) = \frac{1}{2n}\|P\alpha - y\|_2^2 + \lambda \|x\|_1$, where $P$ are synthetic Gaussian data in the overparametrized setting as in Section~\ref{sec:random_features}. 

The convergence behavior of proximal gradient descent follows a transition phase, that can be divided into three phases: first, the method converges linearly according to the nonregularized trajectory, then it converges sublinearly, and it converges linearly once the support is identified. \citet[Theorem 1]{iutzeler2020nonsmoothness} prove the proximal gradient identifies the structure of the solution (described by manifolds, or sparsity patterns) after a certain number of steps. For strongly convex functions, \citet{Nutini_2018_activeset} bounded the `active-set'-complexity of the proximal gradient method. The regularized matching pursuit follows the same behavior. It appears that the sparsity of the solution, and the sparsity identification highly depends on the value of~$\lambda$: the larger~$\lambda$, the sparser the solution and the quicker the identification. Thanks to this observation, we derive a posteriori guarantee in Figure~\ref{fig:final_convergence_sparse}, based on the sparsity of the solution to the optimization problem. In Figure~\ref{fig:final_convergence_sparse}, local strong convexity parameters are given by the estimated of Corollary~\ref{under_over_corollary} and Proposition~\ref{Exact_approx_GS}. We recover that large parameter $\lambda$ both induces a stricter sparsity on the solution and a better convergence.

\begin{figure}
  \centering
  \includegraphics[height=140pt]{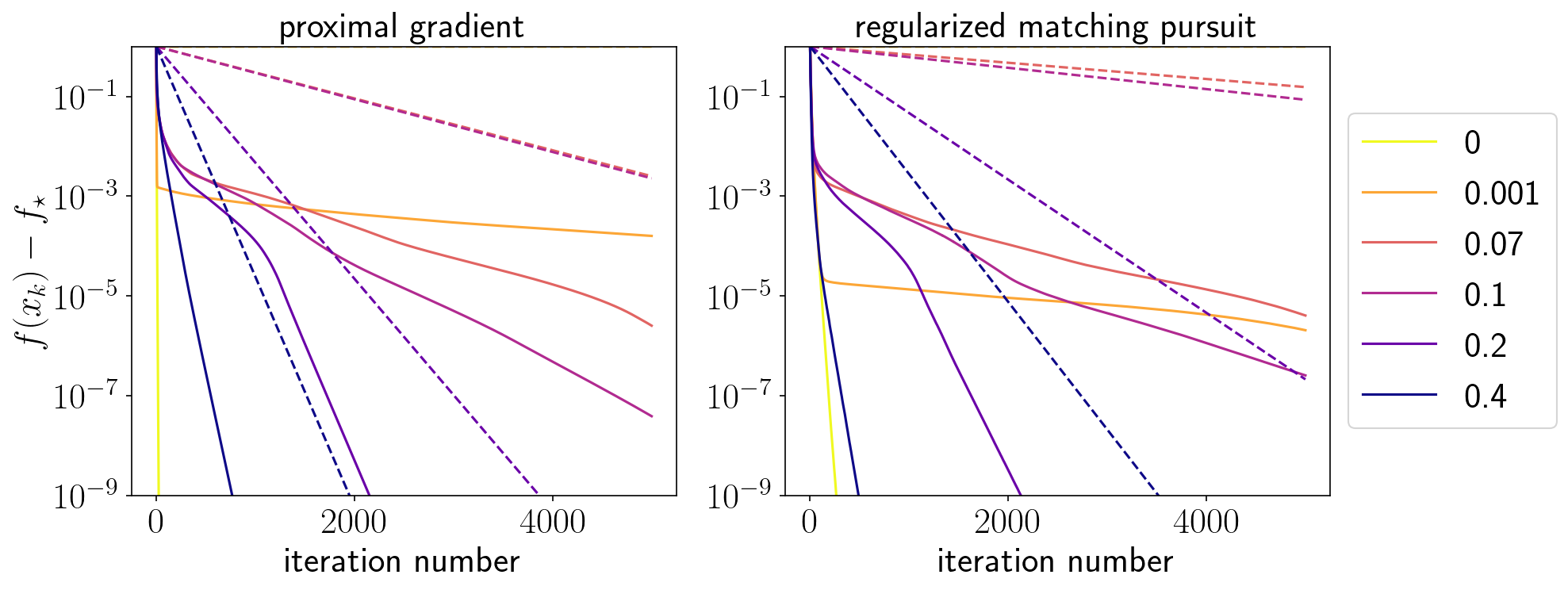}
\caption{Convergence in function values for the proximal gradient on the left and the regularized matching pursuit on the right for  $n = 50$, $d=500$ and a sparsity $s=8$ and for several penalty $\lambda$. Convergence is compared in dashed lines to local convergence guarantee, taken on the support $S$ on the last iterates and the SDP relaxation from Proposition~\ref{Exact_approx_GS}.}
\label{fig:final_convergence_sparse}
\end{figure} 

We describe this transition phase numerically in Figure~\ref{fig:epsilon_lasso} by plotting the $\epsilon$-curve (see Section~\ref{sec:random_features}) as a function of $\lambda$. For $\lambda = 0$, both methods converge linearly (as expected in the overparametrized regime for gradient descent and coordinate descent with the GS-rule). For `large' values of $\lambda$ for which the support is quickly identified, the convergence is linear too. In the intermediary phase however, convergence depends on the effective dimension of the trajectory, $d_{eff} \approx n$ by construction (and thus, on the effective strong convexity of $f$ long the trajectory). The $\epsilon$-curves can be seen as equivalent in the optimization perspective with the regularization path usually drawn in the context of statistical recovery. 

 \begin{figure}
    \centering
    \includegraphics[height=130pt]{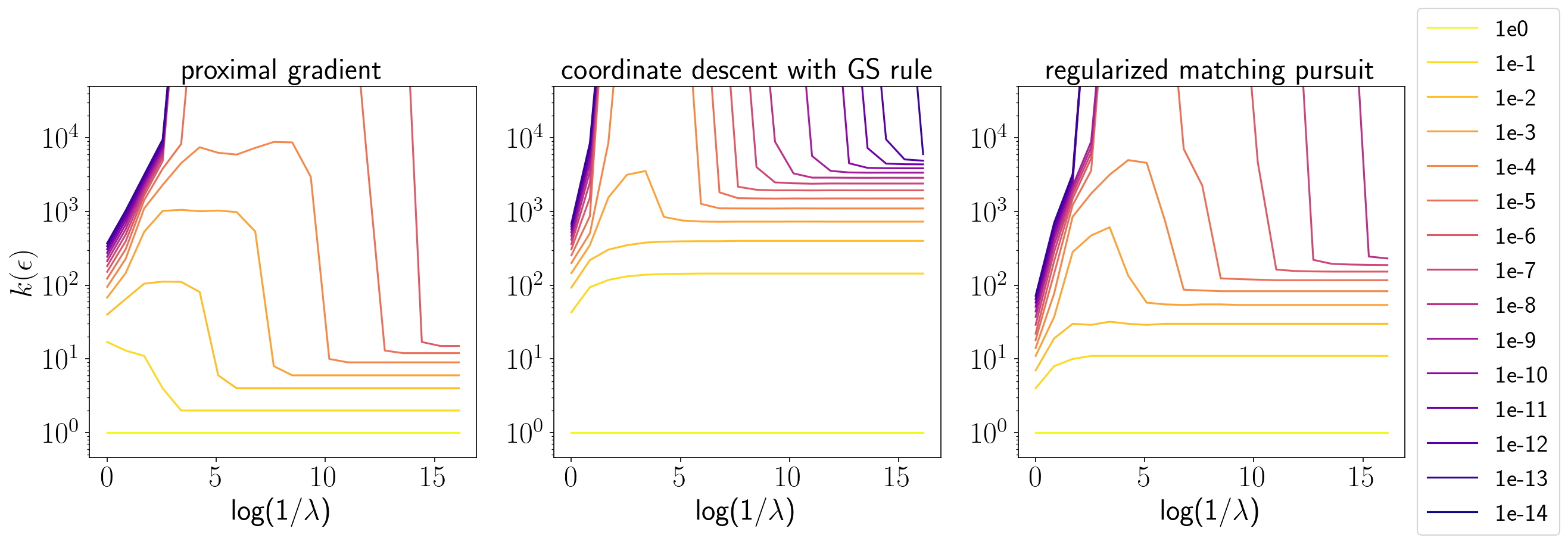}
    \caption{$\epsilon$-curve of the proximal gradient, coordinate descent with the GS rule and regularized matching pursuit for a LASSO problem with $d=500$, $n=50$, a sparsity level $s=8$, $\sigma= 0.5$, after $k=10 000$ iterations for several values of $\lambda$.}
    \label{fig:epsilon_lasso}
\end{figure}

The regularized matching pursuit Algorithm~\ref{Reguralized_MP} formulation allows some intuition regarding the interplay between $\lambda$ and the sparsity of the solution. Let $\mathcal{A} = \{i, \alpha_k^{(i)} > 0\}$ be the set of active atoms. Algorithm~\ref{Reguralized_MP} may reduce an active atom $i \in \mathcal{A}$ to zero if $\|\nabla F(\alpha)\|_{\infty} - \frac{\alpha_k^{(i)}}{|\alpha_k^{(i)}|}\nabla_i F(\alpha) \leqslant 2 \lambda$. The larger $\lambda$, the more active directions may be canceled out. The smaller $\lambda$ ($\lambda \ll \|\nabla F(\alpha)\|_{\infty}$), the closer is regularized matching pursuit to coordinate descent with GS rule (on the right in Figure~\ref{fig:epsilon_lasso}): indeed, only the linear minimization oracle may be added to the set of atoms without modifying other active atoms ($z_i \lesssim z_{\min}$). For $\lambda \approx 0$, the regularized matching pursuit thus converges linearly up to a certain iteration number, which appears with the parallel level lines in Figure~\ref{fig:epsilon_lasso}.

Based on the minimization of a smoothness upper bound with respect to the $\ell_1$-norm, we have developed a regularized matching pursuit algorithm, that benefits from linear convergence in the underparametrized regime (where $F$ is strongly convex), and sublinear convergence in the overparametrized regime (where $F$ is not strongly convex). Thanks to the $\epsilon$-curve, we numerically described the role of $\lambda$ on the convergence of the method (and on the sparsity). In the following section, we propose to develop a method suited to the gauge geometry in the overparametrized regime.

\subsection{An ultimate method adapted to the geometry of regularized models}
\label{sec:ultimate_method}
The regularized matching pursuit~\ref{Reguralized_MP} was derived from the $\ell_1$-geometry. In Section~\ref{section_matching_pursuit}, for non-regularized models, coordinate descent with GS-rule was interpreted as a matching pursuit algorithm in both the underparametrized and overparametrized regime. In what follow, we see that the regularized matching pursuit as developed above does not benefit from this formulation in the overparametrized regime. Instead, we propose an `ultimate method' for the gauge geometry, that benefits from linear convergence in the overparametrized regime but lacks a simple formulation.

Recall the equivalent regularized minimization problems~\eqref{eq:problem}~and~\eqref{eq:problem_n}, 
\begin{align*}
    \min_{\alpha \in \R^d} f(P\alpha) + \lambda \|\alpha\|_1
     = \min_{x \in \R^n} f(x) + \lambda \gamma_{\mathcal{P}}(x),
\end{align*}
where $\gamma_{\mathcal{P}}$ is a gauge function as defined in Section~\ref{section_matching_pursuit}, and $f$ is $L^f_{\gamma_{\mathcal{P}}}$-smooth and $\mu^f_{\gamma_{\mathcal{P}}}$-strongly convex with respect to the gauge. The problem in $\R^d$ is reformulated in $\R^n$, of lower dimension.

\begin{remark}
    This reformulation only requires $\gamma_{\mathcal{P}}$ to be a gauge function, but not specifically to be a norm. Reversely, minimizing a function penalized by a gauge function (or a semi-norm) can be reformulated as minimizing a linear function penalized by and $\ell_1$-norm. 
\end{remark}

As for the regularized matching pursuit, we formulate an optimization method as the minimization of the smoothness upper bound with respect to the gauge function, starting from $x_0 \in \R^n$:
\begin{equation}
    x_{k+1} = \argmin_{x \in \R^n} \langle \nabla f(x_k), x - x_k \rangle + \frac{L^f_{\gamma_{\mathcal{P}}}}{2}\gamma_{\mathcal{P}}(x - x_k)^2 + \lambda \gamma_{\mathcal{P}}(x).
    \label{ultimate_method}
\end{equation}
We refer to this method as the ultimate method for the gauge $\gamma_{\mathcal{P}}$, that is adapted to the geometry of the regularized problem~\ref{eq:problem}. Let us reformulate the minimization Problem~\eqref{ultimate_method} on $\R^n$ into a minimization problem in $\R^d$. Let $x_k = P\alpha_k$ with $\alpha_k \in \R^d$, then
\begin{align*}
    &\min_{x \in \R^n} \langle \nabla f(x_k), x - x_k\rangle + \frac{L^f_{\gamma_{\mathcal{P}}}}{2}\gamma_{\mathcal{P}}(x-x_k)^2 + \lambda \gamma_{\mathcal{P}}(x), \\
    &= \min_{\alpha, \nu \in \R^d} \langle \nabla f(P\alpha_k), P(\alpha - \alpha_k)\rangle + \frac{L^f_{\gamma_{\mathcal{P}}}}{2}\|\alpha - \alpha_k\|_1^2 + \lambda \|\nu\|_1, \ {\rm s.t.} \ x= P\alpha = P \nu, \\
    &= \min_{\alpha, \nu \in \R^d} \langle \nabla F(\alpha_k), \alpha - \alpha_k\rangle + \frac{L^f_{\gamma_{\mathcal{P}}}}{2}\|\alpha - \alpha_k\|_1^2 + \lambda \|\nu\|_1, \ {\rm s.t.} \ P\alpha = P \nu.
\end{align*}

When $P\alpha = P\nu$ implies $\alpha = \nu$, such as in the underparametrized regime where $P^\top P$ is invertible, the ultimate method for the gauge is equivalent with the regularized matching pursuit~\eqref{Reguralized_MP}. However, in the overparametrized regime, $P\alpha = P\nu$ does not imply $\alpha = \nu$ in general. This method does not belong to boosting algorithms due to the evaluation of the gauge function in $x$ and in $x-x_k$ in \eqref{ultimate_method}. In addition, this minimization problem admits neither a simple closed-form solution in general nor a solution based on the KKT conditions (as we did for regularized matching pursuit). While not directly computable in general, the minimization step~\eqref{ultimate_method} converges linearly to the optimum, as proven below in Proposition~\ref{thm:ump_overparam_convergence}.

\begin{proposition}
\label{thm:ump_overparam_convergence}
Let $f$ be $L^f_{\gamma_{\mathcal{P}}}$-smooth and $\mu^f_{\gamma_{\mathcal{P}}}$-strongly convex with respect to the norm $\gamma_{\mathcal{P}}(\cdot)$. The ultimate method~\eqref{ultimate_method} $(x_k)$ converges linearly with
\begin{equation*}
    f(x_k) - f_\star \leqslant \left(1 - \frac{\mu^f_{\gamma_{\mathcal{P}}}}{L^f_{\gamma_{\mathcal{P}}}}\right)^k (f(x_0) - f_\star).
    \label{conv_ultimate_method}
\end{equation*}
\end{proposition}
\begin{proof}
    The proof follows exactly the proof for Theorem~\ref{regularized_MP_proof_strconv}, replacing the function $F$ by $f$ and the norm $\|\cdot\|_1$ by $\gamma_{\mathcal{P}}(\cdot)$.
\end{proof}
Proposition~\ref{thm:ump_overparam_convergence} provides a linear convergence guarantee for the ultimate algorithm. We recover the convergence guarantee of coordinate descent with GS rule in the non regularized model~\eqref{lin_conv_GS_overparam}. 
\begin{figure}
    \centering
    \includegraphics[height=90pt]{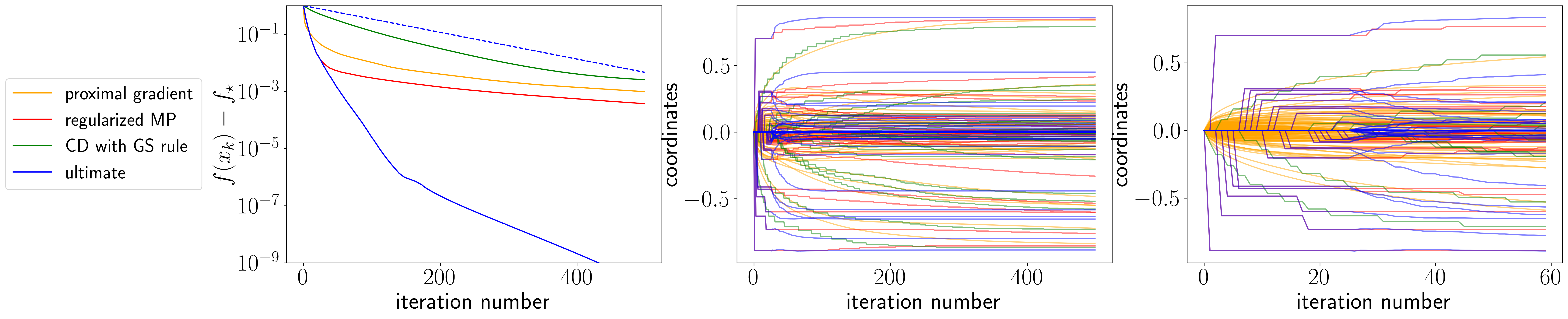}
    \caption{Convergence in function value and coordinates as a function of the iteration number for the proximal gradient descent, the proximal coordinate descent with GS rule, the regularized matching pursuit and the ultimate method on a LASSO problem, with $d=500$, $n=50$, $s=8$, $\lambda = 0.2$. The approximate guarantee in provided in dashed lines.}
    \label{fig:ultimate}
\end{figure}
 In Figure~\ref{fig:ultimate}, we solve the optimization step for the ultimate method for the gauge with the solver MOSEK~\citep{mosek} on a LASSO problem. It converges linearly in the overparametrized regime, while the other method are first stuck in a sublinear phase. Compared to the proximal gradient, proximal coordinate descent with GS rule, the regularized matching pursuit and the ultimate method starts with sparse solution, and differs after a small number of iteration (about 30 here). In the special case of the LASSO, it is possible to approximate its convergence guarantee as for the linear regression problem. Noticing that $L^f_{\gamma_{\mathcal{P}}} = L_1^F$ and $\mu^f_{\gamma_{\mathcal{P}}} = \mu_1^F$, the estimate of the convergence guarantee of coordinate descent with GS rule from Proposition~\ref{Exact_approx_GS} apply here. In the Appendix~\ref{sec:inner_loop_strategy}, we propose an inner loop strategy to avoid the use of an optimization solver, together with the convergence analysis of the outer loop given the precision of the inner loop.

\section*{Conclusion and future works}
In this paper, we developed a principled view for generating optimization algorithms from the minimization of a smoothness upper bound with respect to a well-chosen norm. For non-regularized models, this procedure leads to coordinate descent with GS-rule, that can be interpreted as a matching pursuit algorithm both in the $\ell_1$-geometry for underparametrized models, and in the $\gamma_{\mathcal{P}}$-geometry for overparametrized models. Building on these results, we derive a new regularized matching pursuit algorithm based on the minimization of smoothness with respect to the $\ell_1$-norm (whose counterpart is proximal gradient descent in the $\ell_2$-geometry). While strongly connected to proximal coordinate descent with GS-rule, the regularized matching pursuit cannot be interpreted as a matching pursuit algorithm in the gauge geometry for overparametrized models and does not converge linearly in this regime. We finally formulate an ultimate method adapted to overparametrized geometries. Yet, this method lacks a closed-form formulation. In numerical experiments, we approximate it using an inner-loop strategy.

From this approach, we obtain refined convergence guarantees for (resp.~regularized) matching pursuit (resp.~coordinate descent with GS rule), that are adapted to the geometry of the problem under consideration. For linear regression and the LASSO, we derive upper bounds (resp. high probability bounds) for convergence guarantees using SDP relaxations (resp. under statistical assumptions on the data). As a byproduct, convergence guarantees of both gradient descent and steepest coordinate descent applied to least-squares follow a transition phase from the underparametrized to the overparametrized regime. For $\ell_1$-regularized models, a similar transition phase for $\lambda$ appears, and allows to interpret it as a measure of the sparsity of the solution.

Building on these results, we believe it could be of interest to extend this principled approach to accelerated matching pursuit algorithms (and thus, to accelerated coordinate descent algorithms). Some accelerated techniques have already been developed relying on randomly selected coordinates, such as those of \citet{Nesterov_2017_accelerated} for nonregularized minimization and \citet{2015Fercoq_accelerated} or \citet[Section 3]{2018Locatello} for composite minimization problems, but. Another interesting line of research could be to understand the connections between the observed transition phase for optimization methods and the double descent phenomenon observed for the generalization error in machine learning.

\acks{This work was funded by MTE and the Agence Nationale de la Recherche as part of the “Investissements d’avenir” program, reference ANR-19-P3IA-0001 (PRAIRIE 3IA Institute). We also acknowledge support from the European Research Council (grant SEQUOIA 724063).}

\section*{Codes}
All codes for numerical results are provided at \url{https://github.com/CMoucer/Geometry_Dependent_Matching_Pursuit}.

\vskip 0.2in
\bibliography{references}

\begin{thebibliography}{66}
\providecommand{\natexlab}[1]{#1}
\providecommand{\url}[1]{\texttt{#1}}
\expandafter\ifx\csname urlstyle\endcsname\relax
  \providecommand{\doi}[1]{doi: #1}\else
  \providecommand{\doi}{doi: \begingroup \urlstyle{rm}\Url}\fi

\bibitem[Agarwal et~al.(2010)Agarwal, Negahban, and Wainwright]{2010agarwal}
Alekh Agarwal, Sahand Negahban, and Martin~J Wainwright.
\newblock Fast global convergence rates of gradient methods for high-dimensional statistical recovery.
\newblock In \emph{Advances in Neural Information Processing Systems}, volume~23, 2010.

\bibitem[ApS(2022)]{mosek}
MOSEK ApS.
\newblock \emph{The MOSEK optimization toolbox for MATLAB manual. Version 10.0.}, 2022.
\newblock URL \url{http://docs.mosek.com/9.0/toolbox/index.html}.

\bibitem[Bach(2015)]{Bach2015}
Francis Bach.
\newblock Duality between subgradient and conditional gradient methods.
\newblock \emph{SIAM Journal on Optimization}, 25\penalty0 (1):\penalty0 115--129, 2015.

\bibitem[Bach(2023)]{Bach2023_double_descnet}
Francis Bach.
\newblock High-dimensional analysis of double descent for linear regression with random projections.
\newblock Technical report, arXiv:2303.01372, 2023.

\bibitem[Bach et~al.(2012)Bach, Jenatton, Mairal, and Obozinski]{2012BachJenattonMairalObozinski}
Francis Bach, Rodolph Jenatton, Julien Mairal, and Guillaume Obozinski.
\newblock Optimization with sparsity-inducing penalties.
\newblock \emph{Foundations and Trends in Machine Learning}, 4\penalty0 (1):\penalty0 1--106, 2012.

\bibitem[Bai and Silverstein(2010)]{Bai2010}
Zhidong. Bai and Jack~.W. Silverstein.
\newblock Spectral analysis of large dimensional random matrices.
\newblock \emph{Springer Series in Statistics}, \penalty0 (2nde Edition), 2010.

\bibitem[Bauschke and Combettes(2017)]{Bauschke2017}
Heinz~H. Bauschke and Patrick~L. Combettes.
\newblock \emph{Convex Analysis and Monotone Operator Theory in Hilbert Spaces}.
\newblock 2017.

\bibitem[Beck and Tetruashvili(2013)]{2013Beck}
Amir Beck and Luba Tetruashvili.
\newblock On the convergence of block coordinate descent type methods.
\newblock \emph{SIAM Journal on Optimization}, 23\penalty0 (4):\penalty0 2037--2060, 2013.

\bibitem[Belkin et~al.(2019)Belkin, Hsu, Ma, and Mandal]{Belkin2018}
Mikhail Belkin, Daniel Hsu, Siyuan Ma, and Soumik Mandal.
\newblock Reconciling modern machine learning and the bias-variance trade-off.
\newblock \emph{Proceedings of the National Academy of Sciences}, 32\penalty0 (10):\penalty0 15849--15854, 2019.

\bibitem[Berthier et~al.(2020)Berthier, Bach, and Gaillard]{Berthier2020}
Rapha\"{e}l Berthier, Francis Bach, and Pierre Gaillard.
\newblock Accelerated gossip in networks of given dimension using {J}acobi polynomial iterations.
\newblock \emph{SIAM Journal on Mathematics of Data Science}, 2\penalty0 (1):\penalty0 24--47, 2020.

\bibitem[Bolte et~al.(2010)Bolte, Daniilinis, Ley, and Mazet]{2010Bolte}
Jérôme Bolte, Aris Daniilinis, Olivier Ley, and Laurent Mazet.
\newblock Characterization of Łojasiewicz inequalities: subgradient flows, talweg, convexity.
\newblock \emph{Transactions of the American Mathematical Society}, 362\penalty0 (6):\penalty0 3319--3363, 2010.

\bibitem[Borgwardt(1987)]{1987_Borgwardt}
K.H. Borgwardt.
\newblock The average number of pivot steps required by the simplex-method is polynomial.
\newblock \emph{Zeitschrift für Operations Research}, 26:\penalty0 157--177, 1987.

\bibitem[Bottou et~al.(2018)Bottou, Curtis, and Nocedal]{Bottou2018}
L\'{e}on Bottou, Frank~E. Curtis, and Jorge Nocedal.
\newblock Optimization methods for large-scale machine learning.
\newblock \emph{SIAM Review}, 60\penalty0 (2):\penalty0 223--311, 2018.

\bibitem[Boucheron et~al.(2013)Boucheron, Lugosi, and Massart]{2013_Boucheron}
Stéphane Boucheron, Gabòr Lugosi, and Pascal Massart.
\newblock \emph{{Concentration inequalities. A nonasymptotic theory of independence}}.
\newblock Oxford University Press. 2013.

\bibitem[Candes and Tao(2005)]{Candes2005}
E.J. Candes and T.~Tao.
\newblock Decoding by linear programming.
\newblock \emph{IEEE Transactions on Information Theory}, 51\penalty0 (12):\penalty0 4203--4215, 2005.

\bibitem[Combettes and Wajs(2005)]{2005Combettes}
Patrick~L. Combettes and Val\'{e}rie~R. Wajs.
\newblock Signal recovery by proximal forward-backward splitting.
\newblock \emph{Multiscale Modeling \& Simulation}, 4\penalty0 (4):\penalty0 1168--1200, 2005.

\bibitem[d'Aspremont et~al.(2018)d'Aspremont, Guzm\'{a}n, and Jaggi]{Daspremont2018_affine_invariant}
Alexandre d'Aspremont, Crist\'{o}bal Guzm\'{a}n, and Martin Jaggi.
\newblock Optimal affine-invariant smooth minimization algorithms.
\newblock \emph{SIAM Journal on Optimization}, 28\penalty0 (3):\penalty0 2384--2405, 2018.

\bibitem[d'Aspremont et~al.(2021)d'Aspremont, Scieur, and Taylor]{daspremont2021}
Alexandre d'Aspremont, Damien Scieur, and Adrien Taylor.
\newblock \emph{{Acceleration Methods}}, volume~5 of \emph{Foundations and Trends in Optimization}.
\newblock 2021.

\bibitem[Diakonikolas and Orecchia(2018)]{2018Diakonikolas}
Jelena Diakonikolas and Lorenzo Orecchia.
\newblock Alternating randomized block coordinate descent.
\newblock In \emph{Proceedings of the International Conference on Machine Learning}, volume~80 of \emph{Proceedings of Machine Learning Research}, pages 1224--1232, 2018.

\bibitem[Dudik et~al.(2012)Dudik, Harchaoui, and Malick]{2012dudik}
Miroslav Dudik, Zaid Harchaoui, and Jerome Malick.
\newblock Lifted coordinate descent for learning with trace-norm regularization.
\newblock In \emph{Proceedings of the Fifteenth International Conference on Artificial Intelligence and Statistics}, pages 327--336, 2012.

\bibitem[Fang et~al.(2020)Fang, Fan, Sun, and Friedlander]{Fang2020}
Huang Fang, Zhenan Fan, Yifang Sun, and Michael~P. Friedlander.
\newblock Greed meets sparsity: Understanding and improving greedy coordinate descent for sparse optimization.
\newblock In \emph{Proceedings of the International Conference on Artificial Intelligence and Statistics}, 2020.

\bibitem[Fercoq and Richt\'{a}rik(2015)]{2015Fercoq_accelerated}
Olivier Fercoq and Peter Richt\'{a}rik.
\newblock Accelerated, parallel, and proximal coordinate descent.
\newblock \emph{SIAM Journal on Optimization}, 25\penalty0 (4):\penalty0 1997--2023, 2015.

\bibitem[Frank and Wolfe(1956)]{1956Frank}
Marguerite Frank and Philip Wolfe.
\newblock An algorithm for quadratic programming.
\newblock \emph{Naval Research Logistics Quarterly}, 3\penalty0 (1‐2):\penalty0 95--110, 1956.

\bibitem[Freund and Schapire(1999)]{1999Freund}
Yoav Freund and Robert~E. Schapire.
\newblock A short introduction to boosting.
\newblock \emph{Japanese Society For Artifical Intelligence}, 14:\penalty0 771--780, 1999.

\bibitem[Friedlander et~al.(2014)Friedlander, Mac\^{e}do, and Pong]{2014friedlander}
Michael~P. Friedlander, Ives Mac\^{e}do, and Ting~Kei Pong.
\newblock Gauge optimization and duality.
\newblock \emph{SIAM Journal on Optimization}, 24\penalty0 (4):\penalty0 1999--2022, 2014.

\bibitem[Goemans and Williamson(1995)]{Goemans95}
Michel~X. Goemans and David~P. Williamson.
\newblock Improved approximation algorithms for maximum cut and satisfiability problems using semidefinite programming.
\newblock \emph{Journal of the ACM}, 42:\penalty0 1115--1145, 1995.

\bibitem[Golub et~al.(1999)Golub, Slonim, Tamayo, Huard, Gaasenbeek, Mesirov, Coller, Loh, Downing, Caligiuri, Bloomfield, and Lander]{Golub1999}
T.R. Golub, D.~K. Slonim, P.~Tamayo, C.~Huard, M.~Gaasenbeek, J.~P. Mesirov, H.~Coller, M.L. Loh, J.~R Downing, M.~A. Caligiuri, C.~D. Bloomfield, and E.~S. Lander.
\newblock Molecular classification of cancer: Class discovery and class prediction by gene expression monitoring.
\newblock \emph{Science}, 286:\penalty0 531--537, 1999.

\bibitem[Guille-Escuret et~al.(2021)Guille-Escuret, Goujaud, Girotti, and Mitliagkas]{2020guilleescuret}
Charles Guille-Escuret, Baptiste Goujaud, Manuela Girotti, and Ioannis Mitliagkas.
\newblock A study of condition numbers for first-order optimization.
\newblock In \emph{Proceedings of the International Conference on Artificial Intelligence and Statistics}, pages 1261--1269, 2021.

\bibitem[Hoffman(1957)]{1957Hoffman}
Alan~J. Hoffman.
\newblock On approximate solutions of systems of linear inequalities.
\newblock \emph{Journal Research of the National Bureau of Standards}, 49:\penalty0 263--264, 1957.

\bibitem[Iutzeler and Malick(2020)]{iutzeler2020nonsmoothness}
Franck Iutzeler and Jérôme Malick.
\newblock Nonsmoothness in machine learning: specific structure, proximal identification, and applications.
\newblock \emph{Set-Valued and Variational Analysis}, 28:\penalty0 661--678, 2020.

\bibitem[Jaggi(2013)]{jaggi2013}
Martin Jaggi.
\newblock Revisiting {Frank-Wolfe}: Projection-free sparse convex optimization.
\newblock In \emph{Proceedings of the International Conference on Machine Learning}, number~1, pages 427--435, 2013.

\bibitem[Karimi et~al.(2016)Karimi, Nutini, and Schmidt]{2016Karimi}
Hamed Karimi, Julie Nutini, and Mark Schmidt.
\newblock Linear convergence of gradient and proximal-gradient methods under the {P}olyak-Łojasiewicz condition.
\newblock In \emph{Machine Learning and Knowledge Discovery in Databases}, pages 795--811, 2016.

\bibitem[Karimireddy et~al.(2019)Karimireddy, Koloskova, Stich, and Jaggi]{Karimireddy2019}
Sai~Praneeth Karimireddy, Anastasia Koloskova, Sebastian~U. Stich, and Martin Jaggi.
\newblock Efficient greedy coordinate descent for composite problems.
\newblock In \emph{Proceedings of the International Conference on Artificial Intelligence and Statistics}, 2019.

\bibitem[Lacoste-Julien and Jaggi(2015)]{Lacoste2015}
Simon Lacoste-Julien and Martin Jaggi.
\newblock On the global linear convergence of frank-wolfe optimization variants.
\newblock In \emph{Advances in Neural Information Processing Systems}, 2015.

\bibitem[Locatello et~al.(2017)Locatello, Khanna, Tschannen, and Jaggi]{Locatello2017}
Francesco Locatello, Rajiv Khanna, Michael Tschannen, and Martin Jaggi.
\newblock A unified optimization view on generalized matching pursuit and {Frank-Wolfe}.
\newblock In \emph{Proceedings of the International Conference on Artificial Intelligence and Statistics}, pages 860--868, 2017.

\bibitem[Locatello et~al.(2018)Locatello, Raj, Karimireddy, Raetsch, Sch{\"o}lkopf, Stich, and Jaggi]{2018Locatello}
Francesco Locatello, Anant Raj, Sai~Praneeth Karimireddy, Gunnar Raetsch, Bernhard Sch{\"o}lkopf, Sebastian Stich, and Martin Jaggi.
\newblock On matching pursuit and coordinate descent.
\newblock In \emph{Proceedings of the International Conference on Machine Learning}, 2018.

\bibitem[Mallat and Zhang(1993)]{Mallat1993}
Stéphane~.G. Mallat and Zhifeng Zhang.
\newblock Matching pursuits with time-frequency dictionaries.
\newblock \emph{IEEE Transactions on Signal Processing}, 41\penalty0 (12):\penalty0 3397–3415, 1993.

\bibitem[Marchenko and Pastur(1967)]{1967Marchenko}
Vladimir~A. Marchenko and Loenid~A. Pastur.
\newblock Distribution of eigenvalues for some sets of random matrices.
\newblock \emph{Matematicheskii Sbornik}, 1\penalty0 (4):\penalty0 457 -- 483, 1967.

\bibitem[Mei and Montanari(2022)]{2022Mei}
Song Mei and Andrea Montanari.
\newblock The generalization error of random features regression: Precise asymptotics and the double descent curve.
\newblock \emph{Communications on Pure and Applied Mathematics}, 75\penalty0 (75):\penalty0 667--766, 2022.

\bibitem[Necoara et~al.(2019)Necoara, Nesterov, and Glineur]{2015Necoara}
Ion Necoara, Yurii Nesterov, and François Glineur.
\newblock Linear convergence of first order methods for non-strongly convex optimization.
\newblock \emph{Mathematical Programming}, 175:\penalty0 69–107, 2019.

\bibitem[Nesterov(2012)]{Nesterov2012}
Yuri Nesterov.
\newblock Efficiency of coordinate descent methods on huge-scale optimization problems.
\newblock \emph{SIAM Journal on Optimization}, 22\penalty0 (2):\penalty0 341--362, 2012.

\bibitem[Nesterov(1983)]{Nesterov1983}
Yurii Nesterov.
\newblock A method for solving the convex programming problem with convergence rate o(1/k2).
\newblock \emph{Proceedings of the USSR Academy of Sciences}, 269:\penalty0 543--547, 1983.

\bibitem[Nesterov and Stich(2017)]{Nesterov_2017_accelerated}
Yurii Nesterov and Sebastian~U. Stich.
\newblock Efficiency of the accelerated coordinate descent method on structured optimization problems.
\newblock \emph{SIAM Journal on Optimization}, 27\penalty0 (1):\penalty0 110--123, 2017.

\bibitem[Nutini et~al.(2018{\natexlab{a}})Nutini, Schmidt, and Hare]{Nutini_2018_activeset}
J.~Nutini, M~Schmidt, and W.~Hare.
\newblock "active-set complexity" of proximal gradient: How long does it take to find the sparsity pattern?
\newblock \emph{Optimization Letters}, 13:\penalty0 645--655, 2018{\natexlab{a}}.

\bibitem[Nutini(2018)]{Nutini_2018thesis}
Julie Nutini.
\newblock \emph{Greed is good: greedy optimization methods for large-scale structured problems}.
\newblock PhD thesis, University of British Columbia, 2018.

\bibitem[Nutini et~al.(2018{\natexlab{b}})Nutini, Schmidt, Laradji, Friedlander, and Koepke]{2018Nutini}
Julie Nutini, Mark Schmidt, Issam Laradji, Michael Friedlander, and Hoyt Koepke.
\newblock Coordinate descent converges faster with the {G}auss-{S}outhwell rule than random selection.
\newblock In \emph{Proceedings of the International Conference on Machine Learning}, pages 1632--1641, 2018{\natexlab{b}}.

\bibitem[Paquette et~al.(2023)Paquette, van Merriënboer, Courtney, and Pegregosa]{2023Paquette}
Courtney Paquette, Bart van Merriënboer, Elliot Courtney, and Fabian Pegregosa.
\newblock Halting time is predictable for large models: A universality property and average-case analysis.
\newblock \emph{Foundations of Computational Mathematics}, 23:\penalty0 597–673, 2023.

\bibitem[Parikh and Boyd(2013)]{Parikh2013ProximalA}
Neal Parikh and Stephen~P. Boyd.
\newblock Proximal algorithms.
\newblock \emph{Foundations and Trends Optimization}, 1:\penalty0 127--239, 2013.

\bibitem[Pedregosa and Scieur(2020)]{2020Pedregosa_acceleration}
Fabian Pedregosa and Damien Scieur.
\newblock Acceleration through spectral density estimation.
\newblock In \emph{Proceedings of the 37th International Conference on Machine Learning}, 2020.

\bibitem[Raskutti et~al.(2010)Raskutti, Wainwright, and Yu]{raskutti2010}
Garvesh Raskutti, Martin~J. Wainwright, and Bin Yu.
\newblock Restricted eigenvalue properties for correlated gaussian designs.
\newblock \emph{Journal of Machine Learning Research}, 11:\penalty0 2241--2259, 2010.

\bibitem[Richt\'{a}rik and Tak\'{a}\v{c}(2014)]{2014Richtarik}
Peter Richt\'{a}rik and Martin Tak\'{a}\v{c}.
\newblock Iteration complexity of randomized block-coordinate descent methods for minimizing a composite function.
\newblock \emph{Mathematical Programming}, 144, 2014.

\bibitem[Scieur and Pedregosa(2020)]{2020Scieur_Polyak}
Damien Scieur and Fabian Pedregosa.
\newblock Universal asymptotic optimality of {P}olyak momentum.
\newblock In \emph{Proceedings of the 37th International Conference on Machine Learning}, 2020.

\bibitem[Sheng~Chen and Luo(1989)]{Chen1989}
Stephen A.~Billings Sheng~Chen and Wan Luo.
\newblock Orthogonal least squares methods and their application to non-linear system identification.
\newblock \emph{International Journal of Control}, 50\penalty0 (5):\penalty0 1873--1896, 1989.

\bibitem[Song et~al.(2017)Song, Cui, Jiang, and Xia]{Song2017}
Chaobing Song, Shaobo Cui, Yong Jiang, and Shu-Tao Xia.
\newblock Accelerated stochastic greedy coordinate descent by soft thresholding projection onto simplex.
\newblock In \emph{Advances in Neural Information Processing Systems}, 2017.

\bibitem[Spielman and Teng(2001)]{Spielman_2001}
Daniel Spielman and Shang-Hua Teng.
\newblock Smoothed analysis of algorithms: Why the simplex algorithm usually takes polynomial time.
\newblock In \emph{Proceedings of the Thirty-Third Annual ACM Symposium on Theory of Computing}, page 296–305, 2001.

\bibitem[Sun and Bach(2022)]{Sun2020_1}
Yifan Sun and Francis Bach.
\newblock Safe screening for the generalized conditional gradient method.
\newblock \emph{Open Journal of Mathematical Optimization}, 3:\penalty0 1--35, 2022.

\bibitem[Taylor et~al.(2018)Taylor, Hendrickx, and Glineur]{Taylor2018}
Adrien~B. Taylor, Julien Hendrickx, and Fran{\c{c}}ois~and Glineur.
\newblock {Exact worst-case convergence rates of the proximal gradient method for composite convex minimization}.
\newblock \emph{Journal of Optimization Theory and Applications}, 178\penalty0 (2):\penalty0 455--476, 2018.

\bibitem[Tewari et~al.(2011)Tewari, Ravikumar, and Dhillon]{2012Tewari}
Ambuj Tewari, Pradeep Ravikumar, and Inderjit Dhillon.
\newblock Greedy algorithms for structurally constrained high dimensional problems.
\newblock In \emph{Advances in Neural Information Processing Systems}, 2011.

\bibitem[Tibshirani(1996)]{1996Tibshirani}
Robert Tibshirani.
\newblock Regression shrinkage and selection via the {L}asso.
\newblock \emph{Journal of the Royal Statistical Society. Series B}, 58\penalty0 (1):\penalty0 267--288, 1996.

\bibitem[Tropp(2004)]{2004Tropp}
Joel~A. Tropp.
\newblock Greed is good: algorithmic results for sparse approximation.
\newblock \emph{IEEE Transactions on Information Theory}, 50\penalty0 (10):\penalty0 2231--2242, 2004.

\bibitem[Tseng(2001)]{Tseng2001}
Paul Tseng.
\newblock Convergence of a block coordinate descent method for nondifferentiable minimization.
\newblock \emph{Journal of Optimization Theory and Applications}, 109\penalty0 (3):\penalty0 475–494, 2001.

\bibitem[Tseng and Yun(2009)]{Tseng2009}
Paul Tseng and Sangwoon Yun.
\newblock A coordinate gradient descent method for nonsmooth separable minimization.
\newblock \emph{Mathematical Programming}, B\penalty0 (117):\penalty0 387--423, 2009.

\bibitem[Vershynin(2018)]{vershynin_2018}
Roman Vershynin.
\newblock \emph{High-Dimensional Probability: An Introduction with Applications in Data Science}.
\newblock Cambridge Series in Statistical and Probabilistic Mathematics. 2018.

\bibitem[Zhang(2011{\natexlab{a}})]{Zhang2012_OMP}
Tong Zhang.
\newblock Sparse recovery with orthogonal matching pursuit under {RIP}.
\newblock \emph{IEEE Transactions on Information Theory}, 57\penalty0 (9):\penalty0 6215--6221, 2011{\natexlab{a}}.

\bibitem[Zhang(2011{\natexlab{b}})]{Zhang2012greedy}
Tong Zhang.
\newblock Adaptive forward-backward greedy algorithm for learning sparse representations.
\newblock \emph{IEEE Transactions on Information Theory}, 57\penalty0 (7):\penalty0 4689--4708, 2011{\natexlab{b}}.

\bibitem[Zhang et~al.(2012)Zhang, Schuurmans, and Yu]{2012Zhang_matrixboosting}
Xinhua Zhang, Dale Schuurmans, and Yao-liang Yu.
\newblock Accelerated training for matrix-norm regularization: A boosting approach.
\newblock In \emph{Advances in Neural Information Processing Systems}, 2012.

\end{thebibliography}

\appendix

\section*{Appendix}
In this document, we provide proofs for the main theorems of the paper as well as additional experiments offering a comprehensive overview of the main paper's results. Table~\ref{tab:summarize_appendix} summarizes the main contributions and results of the Appendix.

\begin{table}[htbp]{\renewcommand{\arraystretch}{1.3}
  \caption{Content of the appendices}
  \label{tab:summarize_appendix}
  \begin{center}
    \begin{tabular}{cc}
      \hline
      Appendix~\ref{ap:math_tools} & Mathematical tools appearing in the proofs: $\eta$-tricks  \\ & and computation of the basis of a kernel. \\
      \hline
      Appendix~\ref{proof_linreg_gd_coord} & Proof for linear convergence of matching pursuit (Proposition~\ref{conv_coordinate}) \\ & and gradient descent (Proposition~\ref{theolinreg_gd}).  \\
      \hline
      Appendix~\ref{estimates_linconv_linear} & Estimates for the convergence of steepest coordinate descent \\ & for least-squares: SDP relaxations and high-probability bounds \\ & for $\mu_1^F$, $\mu_{1, L}^F$ and $L_1^F$, with numerical comparisons to $\mu_2^F$ and $\mu_{2, L}^F$. \\
      \hline
      Appendix~\ref{ap:matching_pursuit_gauge} & Matching pursuit in the gauge geometry: \\ & properties of the gauge and sublinear convergence guarantee.  \\
      \hline
      Appendix~\ref{GS_matching_pursuit_appendix} & Formulation of steepest coordinate descent as a `nearly' \\ & matching pursuit algorithm.  \\
      \hline
      Appendix~\ref{sec:inner_loop_strategy} &  Efficiently computing the ultimate method: an inner loop strategy.  \\
      \hline
    \end{tabular}
  \end{center}}
\end{table}

\section{Mathematical tools}
\label{ap:math_tools}
\subsection{The $\eta$-tricks: reweighted least-squares formulations}
\label{eta_trick_appendix}
Due to non-smoothness, the $\ell_1$-norm is often seen as difficult to optimize. A common way to simplify regularization term containing an $\ell_1$-term, such as the LASSO, consists in formulating it as a reweighted least-square problems. We refer to the work of~\citet{2012BachJenattonMairalObozinski} There are three common formulations:
\begin{itemize}
    \item the variational formulation \cite[Section 5]{2012BachJenattonMairalObozinski}
    \begin{equation*}
        \|x\|_1 = \min_{\eta \in \R^n, \eta \geqslant 0} \frac{1}{2} \sum_{i=1}^n \frac{x_i^2}{\eta_i} + \frac{1}{2}\sum_{i=1}^n \eta_i,
    \end{equation*}
    \item the variational constrained formulation \cite[Section 1]{2012BachJenattonMairalObozinski}
    \begin{equation*}
        \|x\|_1^2 = \min_{\eta \in \Delta_n} \sum_{i=1}^n \frac{x_i^2}{\eta_i},
    \end{equation*}
    where $\Delta_n = \{\eta \in \R^n, \eta \geqslant 0, \sum_{i=1}^n \eta_i = 1\}$ is the simplex,
    \item the maximization problem,
\begin{equation*}
    \|x\|_1 = \max_{\|s\|_{\infty} \leqslant 1} \langle s, x \rangle.
\end{equation*}
\end{itemize}

Finally, we notice a useful variational trick, in dimension $1$. For $x \in \R^n$,
\begin{equation}
\label{z_trick}
    \frac{L}{2} \|x\|_1^2 = \max_{z \geqslant 0} \frac{-z^2}{2L} + z\|x\|_1,
\end{equation}
with the optimum $z_{\star} = L\|x\|_1$.

\subsection{Computing the basis of a kernel}
\label{section_basis_kernel}
 To compute the basis of $\text{Ker}(P)$, we perform a QR decomposition on $P$ such that $Q_1^\top Q_1 = I_{n}$, $Q_2^\top Q_2 = I_{d-n}$ and $Q_1^\top Q_2 = 0$: $P^\top = \begin{pmatrix}
            Q_1 & Q_2
        \end{pmatrix}\begin{pmatrix}
            R \\ 0
        \end{pmatrix}$, where $Q_2$ is a basis for the nullspace of $P$. Indeed, let $z \in \R^{d-n}$, and $AQ_2z = R^\top Q_1^\top Q_2 z = 0$.
        
\section{Proof for Propositions~\ref{theolinreg_gd},~\ref{conv_coordinate}}
\label{proof_linreg_gd_coord}

Let us prove linear convergence of gradient descent with fixed step size and coordinate descent with GS rule in a more general framework. This proof leads to the results of Propositions~\ref{theolinreg_gd},~\ref{conv_coordinate} for the $\ell_2$, $\ell_1$ norm respectively.
\smallbreak
Let $F$ be $L^F$-smooth with respect to a norm $\|\cdot\|$, (possibly) $\mu^F$-strongly convex with respect to $\|\cdot\|$ and verify the \L ojasiewicz inequality with parameter $\mu^{F}_L$. We consider a method $(\alpha_k)$ starting from $\alpha_0 \in \R^d$, and obtained by minimizing the smoothness quadratic upper bound:
\begin{align*}
    F(\alpha_{k+1}) &\leqslant  F(\alpha_k) + \min_{\alpha \in \R^d}\left(\langle \nabla F(\alpha_k), \alpha - \alpha_k \rangle + \frac{L^F}{2}\| \alpha - \alpha_k\|^2\right) \leqslant F(\alpha_k) - \frac{1}{2L^F}\|\nabla F(\alpha_k)\|_{\star}^2.
\end{align*}
If $F$ is $\mu^F$-strongly convex, then $F$ verifies the \L ojasiewicz inequality with parameter $\mu^F$: for all $\alpha \in \R^d$, $\mu^F (F(\alpha) - F_\star) \leqslant \frac{1}{2}\|\nabla F(\alpha)\|_{\star}^2$. Thus, substracting $F_\star$ on each side of the smoothness inequality, we have
\begin{equation*}
    F(\alpha_{k+1}) - F_\star \leqslant \left(1 - \frac{\mu^F}{L^F}\right)(F(\alpha_k) - F_\star).
\end{equation*}
Similarly, if $F$ satisfies the \L ojasiewicz inequality with parameter $\mu^{F}_L$, but is not strongly convex ($\mu^F = 0$).

\section{Estimates for the convergence of steepest coordinate descent for least-squares}
\label{estimates_linconv_linear}
In Proposition~\ref{conv_coordinate}, a sequence ($\alpha_k$) generated by steepest coordinate descent for a linear regression problem has a linear convergence rate in function values,
\begin{equation*}
        F(\alpha_k) - F_\star \leqslant \left( 1 - \frac{\max(\mu_1^F, \mu_{1, L}^{F})}{L_1^F}\right)^{k}(F(\alpha_0) - F_\star).
\end{equation*}
We assume that $F$ is a quadratic, that is $F(\alpha) = \frac{1}{2n}\|P\alpha-y\|_2^2$. We first derive SDP relaxations for the optimization problems characterizing $\mu_1^F$, $\mu_{1, L}^F$ and $L_1^F$, and numerically compare these estimates to $\mu_2^F$ and $\mu_{2, L}^{F}$. Then, we compute inequalities connecting $\mu_{1, L}^{F}$ and $\mu_1^{F}$ to $L_1^F$. Thanks to these inequalities, we prove concentration inequalities for $\mu_{1, L}^{F}$, $\mu_1^{F}$ and $L_1^F$, and derive approximate convergence guarantees of steepest coordinate descent.

\subsection{SDP relaxations for $\mu_1^F$ and $\mu_{1, L}^{F}$}
\label{exact_approximation}
We look for exact lower bounds for $\mu_1^F$ and $\mu_{1, L}^{F}$. Both in the overparametrized and underparametrized regime, we are going to reformulate the optimization problems defining $\mu_1^F, \mu_{1, L}^{F}$ into SDPs, and relax some rank constraints. Then, we compare these estimates to $\mu_2^F$ and $\mu_{2, L}^{F}$ in numerical experiments.

\subsubsection{Proof for Proposition~\ref{Exact_approx_GS}}
\noindent \textbf{SDP relaxation for $\mu_1^F$ in the underparametrized regime}.
Recall from Lemma~\ref{comparison_mu_pl} that $\mu_1^F$ is non zero in this regime and given by $\frac{1}{\sqrt{n\mu_1^F}} = \max_{\alpha, \|\alpha\|_{\infty}\leqslant 1} \max_{\nu, \|P\nu\|_2 \leqslant 1}  \alpha^\top \nu$. Since $P^\top P$ is invertible, we proceed to a change of variable in $P^\top P$:
\begin{align*}
    \frac{1}{\sqrt{n\mu_1^F}}  &= \max_{\alpha, \|P^\top P\alpha\|_{\infty}\leqslant 1}  \max_{\nu, \|Pz\|_2 \leqslant 1}  (P^\top Px) ^\top z, \\
    &= \max_{\alpha, \|P^\top P\alpha\|_{\infty}\leqslant 1}  \max_{\nu, \|P\nu\|_2 \leqslant 1}  (P\alpha) ^\top (P\nu) , \\
    &= \max_{\alpha, \|P^\top P\alpha\|_{\infty}\leqslant 1} \|P\alpha\|_2 , \\
    &= \max_{\alpha} \|P(P^\top P)^{-1}\alpha\|_2, \text{ s.t. } \|\alpha\|_{\infty} \leqslant 1.
\end{align*}
A first attempt to compute an exact solution to this problem is to consider $\alpha \in \{-1, 1\}^d$ and compute all possible values for the $\|\cdot\|_{\infty}$-norm. However this computation would require $2^d$ configurations. Instead, we compute an SDP relaxation of $\frac{1}{\sqrt{\mu_1^F}}$ in Lemma~\ref{SDP_underparametrized}.
\begin{lemma}
\label{SDP_underparametrized}
    Let the regime be underparametrized ($n \geqslant d$). $\mu_1^F$ has a $\frac{\pi}{2}$-approximation:
    \begin{equation*}
     \tilde{\mu}_1^F \leqslant \mu_1^F \leqslant \frac{\pi}{2}\tilde{\mu}_1^F.
    \end{equation*}
    where $\tilde{\mu}_1^F$ has an SDP-formulation $\frac{1}{\tilde{\mu}_1^F} = n \max_{X \succcurlyeq 0, {\rm diag}(X) \leqslant 1} {\rm Tr}(CX)$, with $C = (P^\top P)^{-1}$. 
\end{lemma}
\begin{proof} 
\label{proof_SDP_underparametrization}
Let us reformulate an SDP relaxation to this problem. The problem $OPT = \frac{1}{\mu_1^F}  = n \max_{\nu, \|\nu\|_{\infty}^2 \leqslant 1} \|P(P^\top P)^{-1}\nu\|_2^2  = n \max_{\nu, \nu_i^2 \leqslant 1} \nu^\top (P^\top P)^{-1}\nu$ can be relaxed into a SDP,
\begin{align}
\label{eq:SDP_relax_underparametrized}
  SDP=  \frac{1}{\tilde{\mu}_1^F} = n \max_{Z \succcurlyeq 0}  \text{Tr}(CZ), \  \ \ \text{s.t. }  \text{diag}(Z) \leqslant 1,
\end{align}
where $C = (P^\top P)^{-1} \succ 0$. This is exactly the Max-Cut SDP relaxation, for which \citet{Goemans95} proved the approximation $\frac{2}{\pi} SDP \leqslant OPT \leqslant SDP$. 
\end{proof}

\bigbreak
\noindent\textbf{SDP relaxation for $\mu_{1, L}^{F}$ in the overparametrized regime}. For $\mu_{1, L}^{F}$, the maximization problem does not correspond to Max-Cut. So it is possible to derive an SDP lower bound, but no SDP-approximation. Recall the formulation of $\mu_{1, L}^{F}$ as an optimization problem $\mu_{1, L}^{F} = \frac{1}{n} \inf_{\nu \in \R^d}\|P^\top P\nu\|_{\infty}^2,  \ \ { \rm s.t. }\  \|P\nu\|_2^2 = 1$. In Lemma~\ref{mu_1PL_SDP}, we derive a lower bound for $\mu_{1, L}^{F}$ formulated as a SDP.
\begin{lemma}
    \label{mu_1PL_SDP}
    Let $PP^\top$ be invertible, then,
    \begin{align*}
      \mu_{1, L}^{F} & \geqslant \frac{1}{n} \inf_{Z \succcurlyeq 0} \|P^\top P Z P^\top P\|_{\infty}, \ \ {\rm s.t. } \ { \rm Tr}(P^\top P Z) = 1,
    \end{align*}
\end{lemma}
\begin{proof}
    We introduce $Z = \nu \nu^\top \in \R^{d \times d}$, where ${\rm rank}(Z) = 1$, and reformulate the problem into $\mu_{1, L}^{F} = \frac{1}{n} \inf_{Z \succcurlyeq 0, {\rm rank}(Z) = 1, { \rm Tr}(P^\top P Z) = 1} \|P^\top P Z P^\top P\|_{\infty}$, which can be relaxed as an SDP $ \mu_{1, L}^{F} \geqslant \frac{1}{n} \inf_{Z \succcurlyeq 0, { \rm Tr}(P^\top P Z) = 1} \|P^\top P Z P^\top P\|_{\infty}$.
\end{proof}

\bigbreak
\noindent \textbf{Comparison of $\tilde{\mu}_1^F$ and $\tilde{\mu}_{1, L}^{F}$}. First, let us notice that norm $\|\cdot\|_{\infty}$ corresponds to the infinite norm on the diagonal of the matrices. Consequently, the constraint ${\rm diag}(X) \leqslant 1$ is equivalent with $\|X\|_{\infty} \leqslant 1$. In addition, in the underparametrized regime, a change of variable $X \rightarrow P^\top P Z P^\top P$ leads to the reformulation of~\eqref{eq:SDP_relax_underparametrized} as,
\begin{equation*}
    \frac{1}{\tilde{\mu}_1^F} = n \max_{X \succcurlyeq 0} \ {\rm Tr}(P^\top P X), \ \ \ {\rm s.t. \ } \|P^\top P X P^\top P \|_{\infty} \leqslant 1.
\end{equation*}
Thus, we conclude with the fact that $(P^\top P)$ is not invertible in the overparametrized regime, and thus that $\frac{1}{\tilde{\mu}_1^{F}} \geqslant \frac{1}{\tilde{\mu}_{1, L}^{F}}$.

\subsubsection{Numerical comparison}
In this section, we compare the estimate $\tilde{\mu}_1^{F}$ (resp. $\tilde{\mu}_{1, L}^{F}$) for $\mu_1^F$ (resp. $\mu_{1, L}^{F}$) to $\mu_2^F$ and $\mu_2^F/d$ (resp. $\mu_{2, L}^F$ and $\mu_{2, L}^F/d$). We consider Gaussian matrices $X \in \R^{n, d}$ such that $X_i \sim \mathcal{N}(0, \Sigma)$ are i.i.d., given four different diagonal variances: a uniform variance $\Sigma = I_d$, a non-uniform variance $\Sigma = {\rm Diag}(1, \ldots, 1/d)$, a variance with only one small value $\Sigma = {\rm Diag}(1, \cdots, 1, \frac{1}{100})$, a variance with only one large value $\Sigma = {\rm Diag}(1, \ldots, 1, 100)$. These variances are inspired from the work of~\citet[Section 4.1]{2018Nutini}, who computed explicitly $\mu_1^F$ for separable quadratics.

\noindent \textbf{Underparametrized regime: comparison for $\mu_1^F$}. In this regime, $\mu_1^F$ has a SDP approximation given by $\tilde{\mu}_1^F$, as proven in Appendix~\ref{exact_approximation}. By norm equivalence, $\frac{\mu_2^F}{d} \leqslant \mu_1^F \leqslant \mu_2^F$, as proven by~\citet[Appendix 4]{2018Nutini}. The value of the SDP approximation for $\mu_1^F$ in Figure~\ref{fig:Comparing_mu} is very close to its lower bound $\frac{\mu_2^F}{d}$, except for the variance where only one diagonal element of the variance is very large.

\begin{figure}[h!]
    \centering
    \includegraphics[height=220pt]{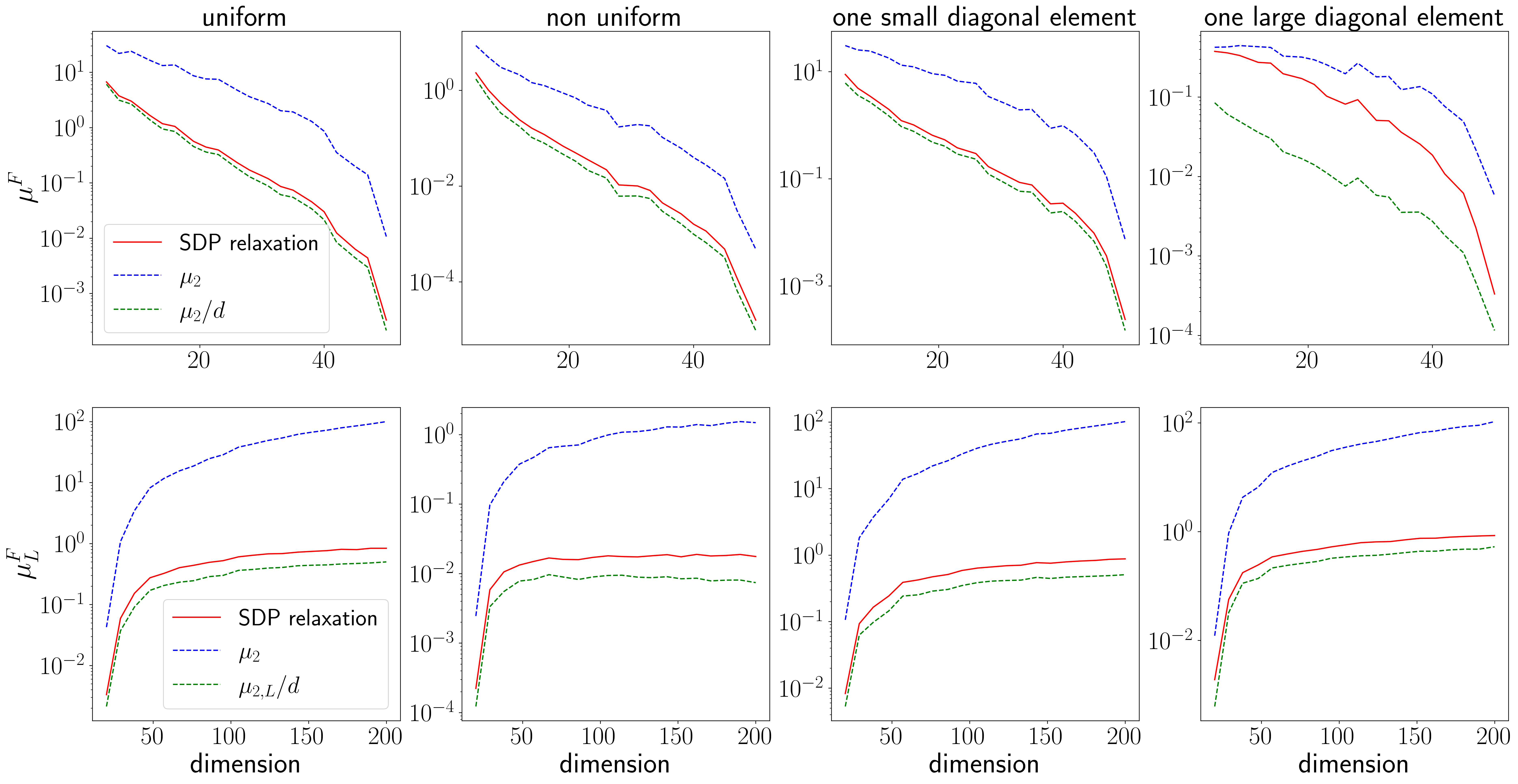}
    \caption{Comparison of $\mu_1^F$ on the top (resp. $\mu_{1, L}^F$ on the bottom) with $\mu_2^F$ and $\mu_2^F/d$ (resp. $\mu_{2, L}^F$ and $\mu_{2, L}^{F}$) for $n=50$ (resp. $n=20$) averaged on $5$ trials, for Gaussian random data with a variance equal from the left to the right to $\Sigma = I_d$, $\Sigma = {\rm diag}(1, 1/2, \ldots, 1/d)$, $\Sigma = {\rm diag}(1, \ldots, 1, 1/100)$ and ${\rm diag}(1, \ldots, 1, 100)$.}
    \label{fig:Comparing_mu}
\end{figure}

\begin{figure}[!h]
    \centering
    \includegraphics[height=120pt]{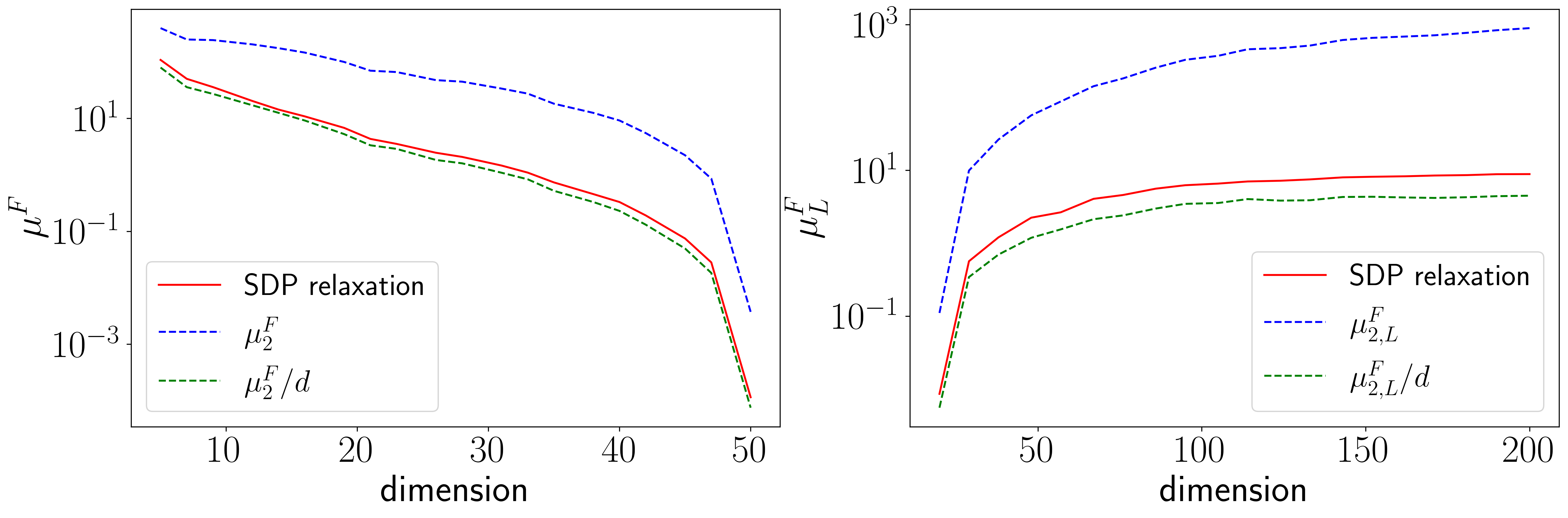}
    \caption{Comparison of the SDP relaxation for $\mu_1^F$ on the left (resp. $\mu_{1, L}^F$ on the right) with $\mu_2^F$ and $\mu_2^F/d$ (resp. $\mu_{2, L}^F$ and $\mu_{2, L}^F/d$) for a random feature model generated from the Leukemia dataset with a number of samples $n = 50$ (resp. $n=20$).}
    \label{fig:Comparing_mu_random_features}
\end{figure}

\noindent \textbf{Overparametrized regime: comparison for $\mu_{1, L}^{F}$}. In this regime, $\mu_{1, L}^{F}$ is lower bounded by $\tilde{\mu}_{1, L}^{F}$ defined by a SDP (see Appendix~\ref{exact_approximation}. By norm equivalence, we have that $\frac{\mu_{2, L}^{F}}{d} \leqslant \mu_{1, L}^{F} \leqslant \mu_{2, L}^{F}$. The SDP relaxation provides a lower bound for $\mu_{1, L}^{F}$, that is always close to its lower bound $\frac{\mu_{2, L}^{F}}{d}$ as observed in Figure~\ref{fig:Comparing_mu}. We observe similar results for a random features model generated from the Leukemia dataset (libsvm) in Figure~\ref{fig:Comparing_mu_random_features}.

\subsection{High-probability bounds for $\mu_1^F$, $\mu_{1, L}^{F}$ and $L_1^F$: proof for Proposition~\ref{big_theo_approx}}

\label{proof:big_theo_approx}

In this proof, we look for concentration inequalities on $\mu_1^F$, $\mu_{1, L}^{F}$ and $L_1^F$. To this end, we first establish deterministic inequalities connecting $\mu_1^F$, $\mu_{1, L}^{F}$ and $L_1^F$. In a second part, we assume the data $P$ are generated randomly, such that each row is subgaussian. Under some mild assumptions, we derive concentration inequalities on these parameters.

\subsubsection{Deterministic inequalities connecting $\mu_1^F$, $\mu_{1, L}^{F}$ and $L_1^F$}
In this proof, we look for inequalities connecting $L_1^F$, $\mu_1^F$ and $\mu_{1, L}^{F}$ to minimal and maximal eigenvalues of $P^\top P$ and $PP^\top$. We first propose to establish deterministic inequalities characterizing these parameters.

We consider the special case of least-squares, for which $F(\alpha) = \frac{1}{n}\|P\alpha - y\|_2^2$ with $P \in \R^{n \times d}$. We have seen that:
\begin{itemize}
    \item $L_1^F = \frac{1}{n}\max_{i=1, \ldots, d} \|P_{:, i}\|_2^2$,
    \item $\mu_1^{F} = \frac{1}{n} \inf_{\eta \in \R^d} \|P\eta\|_2^2 {\rm \ such \ that \ } \|\eta\|_1^2 \geqslant 1$, 
    \item $\mu_{1, L}^{F} = \frac{1}{n} \sup_{\eta \in \R^d} \|P^\top P \eta\|_{\infty}^2 {\rm such \ that \ } \|P\eta\|_2^2 = 1$.
\end{itemize}

In Lemma~\ref{lemma:det_appro_mu_1}, we establish deterministic lower and upper bounds for both $\mu_1^F$ and $\mu_{1, L}^{F}$.

\begin{lemma} 
\label{lemma:det_appro_mu_1}
Let $F(\alpha) = \frac{1}{n}\|P\alpha - y\|_2^2$ with $P \in \R^{n \times d}$. Denote $\mu_1^F$ (resp.~$\mu_{1, L}^{F}$) the strongly convex (resp.~\L ojasiewicz) parameter of $F$ with respect to the $\ell_1$-norm. Then, $\mu_1^F$ verifies,
    \begin{align*}
        \frac{1}{n} \frac{\lambda_{\min}(P^\top P)}{d} \leqslant \ &\mu_1^F \ \leqslant \frac{1}{n}\frac{\mathbf{1}_d^\top P^\top P \mathbf{1}_d}{d^2}, \\
        \frac{1}{n} \frac{\lambda_{\min}(PP^\top)}{d} \leqslant \ &\mu_{1, L}^{F} \ \leqslant \frac{L_1^F}{n}.
    \end{align*}
\end{lemma}
\begin{proof}
    From result of Nutini~et~al.\cite[Appendix 4]{2018Nutini}, we have by the norm equivalence $\frac{\mu_2^F}{d} \leqslant \mu_1^F$. In the special case of least-squares, $\mu_2^F = \frac{\lambda_{\min}(P^\top P)}{n}$.  In addition, we have that $\mu_1^{F} = \frac{1}{n} \inf_{\eta \in \R^d, \|\eta\|_1^2 = 1} \|P\eta\|_2^2 \leqslant \frac{1}{n}\frac{\mathbf{1}_d^\top P^\top P \mathbf{1}_d}{d^2}$ by taking $\eta = \frac{\mathbf{1}_d}{d}$.

    Let us reformulate $\mu_{1, L}^{F}$. Using the trick $\forall \nu\in \R^d, \|\nu\|_1^2 = \inf_{\gamma \in \Delta_d} \sum_{i=1}^d \frac{\nu_i^2}{\gamma_i}$,
\begin{align*}
    \mu_{1, L}^{F} &= \frac{1}{n} \inf_{\nu \in \R^n}\|P^\top \nu\|_{\infty}^2,  \ \ { \rm s.t. }\  \|\nu\|_2^2 = 1, \\
       & = \frac{1}{n} \inf_{\nu \in \R^n} \max_{\eta \in \Delta_d} \nu^\top P{\rm Diag}(\eta) P^\top \nu,  \ \ { \rm s.t. }\  \|\nu\|_2^2 = 1, \\
       &= \frac{1}{n} \max_{\eta \in \Delta_d} \inf_{\nu \in \R^n}  \nu^\top P{\rm Diag}(\eta) P^\top \nu,  \ \ { \rm s.t. }\  \|\nu\|_2^2 = 1, \\
       &= \frac{1}{n} \max_{\eta \in \Delta_d} \lambda_{\min}\left(P{\rm Diag}(\eta)P^\top\right).
\end{align*}
For $\eta = \frac{1}{d}\mathbf{1}_d$, we have that $\mu_{1, L}^{F} \geqslant \frac{1}{n} \frac{\lambda_{\min}(PP^\top)}{d}$. In addition, we can reformulate $\mu_1^{L,F}$ as follows: $\mu_1^{L,F} = \frac{1}{n} \max_{\eta \in \Delta_d} \lambda_{\min}\left(P{\rm Diag}(\eta)P^\top\right) = \frac{1}{n} \max_{\eta \in \R^d} \sup_{M \succcurlyeq 0, {\rm Tr}(M) = 1} {\rm Tr}(P{\rm Diag}(\eta)P^\top M) = \frac{1}{n} \sup_{M \succcurlyeq 0, {\rm Tr}(M) = 1} \|P^\top MP\|_{\infty} \leqslant \frac{1}{n^2} \|P^\top P\|_\infty, \ \ {\rm for \ } M = \frac{I_d}{n} = \frac{L_1^F}{n}$.
\end{proof}

As expected by the norm equivalence (and already highlighted by Nutini~et~al.~\cite[Appendix 4]{2018Nutini}), strong convexity parameters verify $\mu_1^F \geqslant \frac{\mu_2^F}{d}$. Lemma~\ref{lemma:det_appro_mu_1} states a similar comparison for \L ojasiwiecz parameters with $\mu_{1, L}^{F} \geqslant \frac{\mu_{2, L}^{F}}{d}$. We recover that $\mu_1^F \geqslant 0$ (resp.~$\mu_{1, L}^{F} \geqslant 0$) in the overparametrized (resp.~underparametrized) regime. Under subgaussian data, we now build high probability bounds for $\mu_{1, L}^{F}$, $\mu_1^{F}$ and $L_1^F$ based on these lower and upper bounds.

\subsubsection{Generalities on subgaussian data}

Before deriving concentration bounds, we detail the subgaussian assumptions on $P$.

\begin{definition}{\cite[Definition 2.5.2, Proposition 2.5.2]{vershynin_2018}}
\label{def:subgaussian}
    Let $X$ be a subgaussian random with $\tau > 0$, and $\E[X] = 0$. Then, there exists absolute constants $c_1, c_2 > 0$ such that, 
    \begin{align*}
       \forall t \in \R, \ &\E[e^{t X}] \leqslant e^{c_1 \|X\|_{\Psi_2}^2t^2}, \\
       \forall t \in \R, \ & \mathbb{P}(|X| \geqslant t) \leqslant 2\exp(-\frac{c_2}{\|X\|_{\Psi_2}^2} t^2),
    \end{align*}
In addition, $\|X\|_{\Psi_2} = \inf\{t > 0, \E[e^{X^2/t^2}] \leqslant 2\}$ is the subgaussian norm of $X$.
\end{definition}
Usually, a subgaussian variable $X$ with $\E[X] = 0$, $\E[x^2] = \sigma^2$ is defined with a parameter $\tau > 0$ such that for all $t \in \R$, $\E[\exp(t X)] \leqslant \exp(\frac{t^2 \tau^2}{2})$. Then, $\sigma \leqslant \tau = \sqrt{2 c_1}\|X\|_{\Psi_2}$. As for gaussian data, Definition~\ref{def:subgaussian} can be extended to vectors.

\begin{definition}{\cite[Definition 3.4.1]{vershynin_2018}}
    A random vector $X \in \R^d$ is called subgaussian if the one-dimensional marginals $\langle X, x\rangle$ are subgaussian random variables for all $x \in \R^d$. The subgaussian norm of $X$ is defined as $\|X\|_{\psi_2} = \sup_{x \in \mathbb{S}^{d-1}} \|\langle X, x\rangle\|_{\Psi_2}$.
\end{definition}

Throughout this section, we assume the data $P$ to be subgaussian as follows:
\begin{assumption}[Subgaussian data]
\label{assumption:subgaussian}
     $P_1, \ldots, P_n \in \R^{d}$ are i.i.d. subgaussian random vectors with $\E[P_i] = 0$, $\E[P_{i, j}^2] = \sigma^2$ and $K = \max_i \|P_i\|_{\Psi_2}$. 
\end{assumption}

\subsubsection{A concentration inequality for $\sqrt{L_1^F}$} Our goal is to obtain a concentration inequality for $L_1^F = \frac{1}{n} \max_{i=1, .., d} \|P_{:, i}\|_2^2$, where the data $P$ is generated as in Assumption~\ref{assumption:subgaussian}. We rather focus on $\sqrt{L_1^F} = \frac{1}{\sqrt{n}} \max_{i=1, \ldots, d} \|P_{:, i}\|_2$.

\begin{theorem}{\cite[Theorem 3.1.1, equation (3.3)]{vershynin_2018}}
\label{thm:concentration_norm}
    Let $p = (p_1, \ldots, p_n) \in \R^n$ be a random vector with independent subgaussian coordinates $p_i$ that satisfy $\E[p_i^2] = 1$, and $K = \max_i \|p_i\|_{\Psi_2}$. Then, there exists $C> 0$ an absolute constant such that, for all $t \in \R$
    \begin{equation*}
        \mathbb{P}(|\|p\|_2 - \sqrt{n}| \geqslant t) \leqslant 2 {\rm exp}\left( - \frac{C}{K^4} t^2 \right).
    \end{equation*}
\end{theorem}
Equivalently, the concentration result of Theorem~\ref{thm:concentration_norm} can be extended by a linearity argument to random variables with variance $\E[p_i^2] = \sigma^2$: for all $t \in R$,
\begin{equation*}
    \mathbb{P}(|\|p\|_2 - \sigma \sqrt{n}| \geqslant t) \leqslant 2 {\rm exp}\left( - \frac{C}{\sigma^2 K^4} t^2\right).
\end{equation*}
In Theorem~\ref{thm:concentration_norm}, vectors $p$ correspond to columns of $P$ as defined in Assumption~\ref{assumption:subgaussian}. We conclude with the concentration of each norm of the columns $(P_{:, i})$ around $\sigma \sqrt{n}$. More precisely, according to Definition~\ref{def:subgaussian}, $\|P_{:, i}\|_2^2 - \sigma \sqrt{n}$ are i.i.d. subgaussian variables. In Lemma~\ref{lemma:bounds_maxima}, we provide a lower and an upper bound for the expectation of $\max_{i=1, \ldots, d} \|P_{:, i}\|_2$.
\begin{lemma}
\label{lemma:bounds_maxima}
    Let $P \in \R^{n \times d}$ be a collection of subgaussian elements as in Assumption~\ref{assumption:subgaussian}. There exist $C_1, C_2 > 0$ absolute constant such that, 
    \begin{align*}
        C_2 K^2 \sigma \leqslant & \E[\max_{i= 1, \ldots, d} \left(\|P_{:, i}\|_2 - \sigma\sqrt{n}\right)] \leqslant 2 \sigma K^2\sqrt{C_1 \log(d)}.
    \end{align*}
\end{lemma}
\begin{proof}
From Theorem~\ref{thm:concentration_norm} and Definition~\ref{def:subgaussian}, there exists $C_1 > 0$ an absolute constant such that $Y = \|P\|_2 - \sigma \sqrt{n}$ is subgaussian with for all $t \in \R$, $\E[\exp(tY] \leqslant \exp(C_1t^2K^4 \sigma^2) = \exp(v^2 t^2 / 2)$. Boucheron~et~al.~\cite[Theorem 2.5]{2013_Boucheron} derived an upper bound for the expectation to the maximum of independent subgaussian random variables (and more generally, to subgamma random variables): $\E[Y] \leqslant \sqrt{2\log(d)}$.
For the left-hand side, first we have the following inequality $\E[\max_{i= 1, \ldots, d} \|P_{:, i}\|_2] \geqslant \E[\|P_{:, i}\|_2]$. This expectation can be lower bounded using \cite[Theorem 3.1.1]{vershynin_2018} and $1 + x \leqslant e^x$ for all $x>0$.
\end{proof}

\begin{proposition}
\label{prop:concentration_L_1}
    Let $P$ be a subgaussian matrix generated as in Assumption~\ref{assumption:subgaussian} and $L_1^F = \frac{1}{n}\max_{i=1, \ldots} \|P_{:, i}\|_2^2$. Then, there exists absolute constant $C,C_1, C_2 > 0$ such that for all $t \geqslant 2 \sigma K^2 \sqrt{\frac{C_1 \log(d)}{n}}$, 
    \begin{equation*}
        \Prob\left(|\sqrt{L_1^F} - \E[\sqrt{L_1^F}]| \geqslant t\right) \leqslant e^{-\frac{C}{\sigma^2K^4}\min(u_1(t), u_2(t))}.
    \end{equation*}
    where $u_1(t) = \log(d)(t + \frac{C_2K^2\sigma}{\sqrt{n}})^2$ and $u_2(t) = d(t - 2 \sigma K^2 \sqrt{\frac{C_1 \log(d)}{n}})^2$.
\end{proposition}
\begin{proof}
    Recall that $\sqrt{L_1^F} = \max_{i=1, \ldots, d} \|P_{:, i}\|_2$, where $\frac{1}{\sqrt{n}}\|P_{:, i}\|_2 - \sigma$ is subgaussian. For the right side event, notice that $ \Prob\left(\sqrt{L_1^F} \geqslant \E[\sqrt{L_1^F}]  + t\right) \leqslant d \Prob\left(\frac{\|P_{:, i}\|_2}{\sqrt{n}} \geqslant \E[\sqrt{L_1^F}]  + t\right)$ by the union bound. In addition, using the bounds on $\E(\max_{i=1, \ldots, d}\|P_{:, i}\|_2)$ from Lemma~\ref{lemma:bounds_maxima} and concentration from Theorem~\ref{thm:concentration_norm}, there exists $C_2, C > 0$ such that $\Prob\left(\frac{\|P_{:, i}\|_2}{\sqrt{n}} \geqslant \E[\sqrt{L_1^F}]  + t\right) \leqslant d \Prob\left(\frac{\|P_{:, i}\|_2}{\sqrt{n}} - \sigma \geqslant \frac{C_2K^2\sigma}{\sqrt{n}} + t\right) \leqslant \exp\left(-\frac{C\log(d)}{\sigma^2 K^4}(t + \frac{C_2K^2\sigma}{\sqrt{n}})^2\right)$. For the left side event, using independence of the $P_{i,j}$, we conclude from Theorem~\ref{thm:concentration_norm} that there exists an absolute constant $C_1 > 0$ such that $\Prob\left(\sqrt{L_1^F} \leqslant -t + \E[\sqrt{L_1^F}]\right) \leqslant \Prob\left(\sqrt{L_1^F} \leqslant -t + \sigma + 2 \sigma K^2\sqrt{\frac{C_1 \log(d)}{n}}\right) \leqslant \Prob\left(\frac{\|P_{:, i}\|_2}{\sqrt{n}} - \sigma \leqslant -t + 2 \sigma K^2\sqrt{\frac{C_1 \log(d)}{n}}\right)^d \leqslant \exp\left( - \frac{Cd}{\sigma^2K^4}(t - 2 \sigma K^2 \sqrt{\frac{C_1 \log(d)}{n}})^2\right)$. 
\end{proof}

We conclude from Proposition~\ref{prop:concentration_L_1} that $\sqrt{L_1^F}$ concentrates around its mean, that admits lower and upper bounds as in Lemma~\ref{lemma:bounds_maxima}. More precisely, there exist absolute constant $C, C_1, C_2 > 0$ such that for all $t \geqslant 2 \sigma K^2 \sqrt{\frac{C_1 \log(d)}{n}}^2$,
\begin{equation*}
    C_2K^2\sigma\frac{1}{\sqrt{n}} + \sigma - t \leqslant \sqrt{L_1^F} \leqslant \sigma + 2 \sigma K^2 \sqrt{\frac{C_1\log(d)}{n}} + t,
\end{equation*}
holds with probability $1 - \exp(-\frac{C}{\sigma^2K^4}u(t))$, where $u(\cdot)$ is quadratic by part. 

\subsubsection{Concentration inequality for $\mu_{1, L}^{F}$ and $\mu_1^F$ under subgaussian data.}

As we have seen in Lemma~\ref{lemma:det_appro_mu_1}, $\mu_1^F$ and $\mu_{1, L}^{F}$ have deterministic approximants, closely related to the minimal eigenvalues of $PP^\top$ and $P^\top P$. Under subgaussian Assumption~\ref{assumption:subgaussian}, we provide concentration inequalities for $\mu_1^F$ and $\mu_{1, L}^{F}$.

More precisely, recall from Lemma~\ref{lemma:det_appro_mu_1} that $\mu_1^F$ (resp.~$\mu_{1, L}^{F}$) is lower bounded by $\frac{1}{d}\frac{\lambda_{\min(P^\top P)}}{n}$ (resp.~$\frac{1}{n}\frac{\lambda_{\min}(PP^\top)}{d}$). Let us begin by calling a concentration inequality~\cite[Theorem 4.6.1]{vershynin_2018} for eigenvalues of such subgaussian matrices.

\begin{theorem}{\cite[Theorem 4.6.1]{vershynin_2018}}
\label{thm:concentration_eigenvalues}
    Let $P \in \R^{n \times d }$ be a subgaussian matrix generated as in Assumption~\ref{assumption:subgaussian}. Then, there exists an absolute constant $C_3 > 0$ such that, for all $t \geqslant 0$, 
    \begin{equation*}
        \sqrt{n} - C_3K^{2}(\sqrt{d} + t) \leqslant \sqrt{\lambda_{\min}(P^\top P)} \leqslant \sqrt{\lambda_{\max}(P^\top P)} \leqslant \sqrt{n} + C_3K^{2}(\sqrt{d} + t), 
    \end{equation*}
with probability at least $1 - 2 \exp(-t^2)$, with $K = \max_{i}\|P_i\|_{\phi_2}$.
\end{theorem}
Similarly, we deduce from Theorem~\ref{thm:concentration_eigenvalues} that there exists $C_4 > 0$ such that for all $t \geqslant 0$,
\begin{equation*}
        \sqrt{d} - C_4K^{2}(\sqrt{n} + t) \leqslant \sqrt{\lambda_{\min}(PP^\top)} \leqslant \sqrt{\lambda_{\max}(PP^\top)} \leqslant \sqrt{d} + C_4K^{2}(\sqrt{n} + t), 
\end{equation*}
holds with probability at least $1 - 2 \exp(-t^2)$. In particular, it is possible to derive bounds for quadratics of subgaussian data from Theorem~\ref{thm:concentration_eigenvalues}. We provide a concentration result in Proposition~\ref{prop:chi_subgaussian} for $\mathbf{1}_d^\top P^\top P \mathbf{1}_d$ that appears in Lemma~\ref{lemma:det_appro_mu_1}.
\begin{proposition}
    \label{prop:chi_subgaussian}
    Let $P$ be generated as in Assumption~\ref{assumption:subgaussian}. Then, there exists $C_3 > 0$ such that, for all $t \geqslant 0$, 
    \begin{equation*}
       \frac{1}{\sqrt{d}} - C_3 K^2\left(\sqrt{\frac{1}{nd}} + \frac{t}{d\sqrt{n}} \right) \leqslant \sqrt{\frac{\mathbf{1}_d^\top P^\top P \mathbf{1}_d}{d^2 n}} \leqslant \frac{1}{\sqrt{d}} + C_3 K^2\left(\sqrt{\frac{1}{nd}} + \frac{t}{d\sqrt{n}} \right),
    \end{equation*}
    holds with probability $1 - 2\exp(-t^2)$.
\end{proposition}
\begin{proof}
    This result is directly obtained from Vershynin~\cite[Theorem 4.6.1]{vershynin_2018}. Under the same assumptions than in Theorem~\ref{thm:concentration_eigenvalues}, there exists an absolute constant $C_3 > 0$ such that, for all $t \geqslant 0$, $\|\frac{1}{n}P^\top P - I_d\| \leqslant K^2 \max(\delta, \delta^2)$, where $\delta = C_3K^2(\sqrt{\frac{d}{n}} + \frac{t}{\sqrt{n}})$. From this, Vershynin concludes an approximate isometry for $P$~\cite[Lemma 4.1.5]{vershynin_2018}. For all $x \in \R^d$, $(1 - \delta) \|x\|_2 \leqslant \|Px\|_2 \leqslant (1 + \delta)\|x\|_2$. Note that this is exactly the same constant than in Theorem~\ref{thm:concentration_eigenvalues}.
\end{proof}

Using Theorem~\ref{thm:concentration_eigenvalues}, Proposition~\ref{prop:chi_subgaussian} and Lemma~\ref{lemma:det_appro_mu_1}, we get the expected concentration bounds n $\mu_1^F$ and $\mu_{1, L}^{F}$.

\subsection{Proof for Corollary~\ref{coro:limiting_mu_1_L_1}}
\label{proof:coro_limiting_mu_1_L_1}

Given the convergence guarantee for coordinate descent with GS-rule from Theorem~\ref{conv_coordinate}, we conclude concentration for the convergence rate. We provide the proof for convergence in the underparametrized regime, where the convergence guarantee is given by $1 - \frac{\mu_1^F}{L_1^F}$, and let the overparametrized regime to the reader. Recall from Proposition~\ref{proof:big_theo_approx} that, there exists absolute constant $C, C_1, C_2, C_3, K > 0$ such that,
\begin{equation*}
  1 - \frac{1}{d}\frac{\left(1 + C_3K^2(\frac{1}{\sqrt{n}} + \frac{t}{\sqrt{nd}} \right)^2}{\left(1 + C_2K^2\frac{1}{\sqrt{n}} +  \frac{t}{\sqrt{n}}\right)^2}  \leqslant 1 - \frac{\mu_1^F}{L_1^F} \leqslant 1 - \frac{1}{d}\frac{\left(1 - C_3K^2(\frac{1}{\sqrt{n}} + \frac{t}{\sqrt{nd}} \right)^2}{\left(1 + 2K^2\sqrt{C_1\frac{\log(d)}{n}} + \frac{t}{\sqrt{n}}\right)^2},
\end{equation*}
with probability $p(t) = 1 - 4 \exp \left(-\min(t^2, \frac{d\sigma^2}{n}(t - 2K^2\sqrt{C_1\log(d)})^2, \frac{\log(d)\sigma^2}{n}(t+C_2^2K^2)^2)\right)$. Applying a limited development in $\frac{1}{\sqrt{n}}$, the left term is equal to
\begin{equation*}
    1 - \frac{1}{d}\left(1 + \frac{2}{\sqrt{n}}[(C_3 - C_2)K^2 + \frac{t}{\sqrt{d}} - t] + o(\frac{1}{\sqrt{n}})\right),
\end{equation*}
and the right term,
\begin{equation*}
    1 - \frac{1}{d}\left(1 - \frac{2}{\sqrt{n}}[K^2(C_3 + 2\sqrt{C_1\log(d)}) + \frac{t}{\sqrt{d}} + t] + o(\frac{1}{\sqrt{n}})\right).
\end{equation*}
These two limited development allows to conclude to the limiting convergence rate for coordinate descent with GS rule in the underparametrized regime.

\section{Matching pursuit in the gauge geometry}
\label{ap:matching_pursuit_gauge}
\subsection{The gauge is a norm: proof for Lemma~\ref{norm_lemma}}
\label{proof_norm_dual}
Since $\alpha \rightarrow P\alpha$ is surjective, the function $\gamma_{\mathcal{P}}(x)$ is well-defined. Let us prove subbaditivity, positive definiteness and absolute homogeneity.
\begin{itemize}
    \item Let $t > 0$ and $x \in \mathrm{R}^n$,
    \begin{align*}
        \gamma(tx) = \inf_{\alpha\in \mathrm{R}^d, tx = P\alpha} \|\alpha\|_1 = \inf_{\tilde{\alpha} (=\frac{\alpha}{t}) \in \mathrm{R}^d, x = P\tilde{\alpha}}   \|t\tilde{\alpha}\|_1 = t\inf_{\tilde{\alpha} \in \mathrm{R}^d, x = P\tilde{\alpha}} \|\tilde{\alpha}\|_1 = t \gamma(x).
    \end{align*}
    Since $\mathcal{P}= {\rm conv}(P)$ is centrally symmetric, we conclude that for all $t \neq 0$, $\gamma(tx) = |t|\gamma(x)$. Finally, letting $t \rightarrow 0$, we conclude that $\gamma(0) = 0$.
    \item Let $x \in \mathrm{R}^n$ be such that $\gamma(x) = 0$. We have $0 = \inf_{\alpha\in \mathrm{R}^d} \|\alpha\|_1, \text{s.t. } x = P\alpha$. There exists $(\alpha_k)$ a sequence in $\mathrm{R}^d$ such that $x = P\alpha_k$ and $\|\alpha_k\|_1 \rightarrow 0$, meaning that $\alpha_k \rightarrow 0$. By linearity of $P\alpha$, we obtain $x = 0$.
    \item Let $x, y \in \mathrm{R}^n$, $\gamma(x + y) = \inf_{\eta} \|\eta\|_1, \text{s.t. } x + y = P\eta$.
    Let $\alpha, \beta$ be the minimal representation for $x, y$, such that $x = P\alpha$ and $y = P\beta$. We have that
    \begin{align*}
        \gamma(x + y) &= \inf_{\eta, P(\alpha + \beta)  = P\eta,} \|\eta\|_1  \leqslant \|\alpha + \beta\|_1 \leqslant \|\alpha\|_1 + \|\beta\|_1 = \gamma(x) + \gamma(y).
    \end{align*}
\end{itemize}

If $\gamma(\cdot)$ is a norm, we compute its dual norm $\gamma^\star(z) = \sup_{x, \gamma(x) \leqslant 1} \langle z, x \rangle = \sup_{x} \inf_{\lambda \geqslant 0} \langle z, x \rangle + \lambda - \lambda \gamma(x) = \sup_x \inf_{\lambda \geqslant 0} \sup_{\alpha, x = P\alpha} \langle z, x \rangle + \lambda - \lambda \|\alpha\|_1 = \inf_{\lambda \geqslant 0} \lambda + \sup_{\alpha} \langle \alpha, P^\top z\rangle - \lambda \|\alpha\|_1 = \inf_{\lambda \geqslant 0, \|P^\top z\|_{\infty} \leqslant \lambda} \lambda = \|P^\top z\|_{\infty} $.

\subsection{Proof for sublinear convergence of matching pursuit}
\label{proof_conv_sublinar_MP}
We have seen in Section~\ref{section_matching_pursuit} that matching pursuit converges linearly in both the underparametrized and overparametrized regime. This result improves the sublinear guarantee of matching pursuit proven by \citet[Theorem 3]{2018Locatello} letting a sublevel set radius appear (as usual in the coordinate descent literature).
\begin{theorem}
\label{conv_sublinear_MP}
Let $f$ be convex, $L_{\gamma_{\mathcal{P}}}^f$-smooth with respect to the norm $\gamma_{\mathcal{P}}(\cdot)$. Then, the sequence verifies for $\mathcal{R} = \max_{x_\star \in X_\star}\max_{x \in \R^d} \gamma_{\mathcal{P}}(x-x_\star), \ {\rm s.t.} \ f(x) \leqslant f(x_0) < +\infty$,
\begin{equation*}
    f(x_k) - f_\star \leqslant \frac{2L_{\gamma_{\mathcal{P}}}^f \mathcal{R}^2}{k-1}.
\end{equation*}
\end{theorem}
\begin{proof}
The sequence $(x_k)$ verifies the descent lemma:
    \begin{equation*}
        f(x_{k+1}) - f(x_k) \leqslant - \frac{1}{2L_{\gamma_{\mathcal{P}}}^f}\sigma_{\mathcal{P}}(\nabla f(x_k))^2.
    \end{equation*}
We introduce $\delta_{k} = f(x_k) - f_\star$, so that $\delta_{k+1} \leqslant \delta_k - \frac{1}{L_{\gamma_{\mathcal{P}}}^f} \sigma_{\mathcal{P}}(\nabla f(x_k))^2$. By convexity and Cauchy-Schwarz inequality, we have $\delta_k \leqslant \langle x_k - x_\star, \nabla f(x_k) \rangle \leqslant \gamma_{\mathcal{P}}(x_k - x_\star)\sigma_{\mathcal{P}}(\nabla f(x_k))$. Then,
\begin{equation*}
    \delta_{k+1} \leqslant \delta_k - \frac{1}{L_{\gamma_{\mathcal{P}}}^f \gamma_{\mathcal{P}}(x_k - x_\star)^2}\delta_k^2 \leqslant \delta_k - \frac{1}{L_{\gamma_{\mathcal{P}}}^f \mathcal{R}^2}, \delta_k^2.
\end{equation*}
where $\mathcal{R}=\max_{x_\star \in X_\star, x \in \R^d} \gamma_{\mathcal{P}}(x_k - x_\star),  {\rm s.t.} f(x) \leqslant f(x_0)$ is assumed to be finite. Denoting $\omega = \frac{1}{L_{\mathcal{P}} \mathcal{R}^2}$, and dividing by $\delta_k$, we have that $\frac{1}{\delta_k} + \omega \frac{\delta_k}{\delta_{k+1}} \leqslant \frac{1}{\delta_{k+1}}$. Since $\delta_k$ is nonincreasing with $k$, $\frac{1}{\delta_{k+1}} \geqslant \frac{1}{\delta_{k}} + \omega$. By summation, $\frac{1}{\delta_k} \geqslant \omega (t-1)$ and the result follows.
\end{proof}

\section{Steepest coordinate descent is `nearly' a matching pursuit algorithm}
\label{GS_matching_pursuit_appendix}
Coordinate descent with a Gauss-Southwell rule can be formulated using an LMO. Indeed, recall its formulation,
\begin{align*}
         \alpha_{k+1} &= \argmin_{\alpha \in \R^d} \langle \nabla_{i_k} F(\alpha_k) e_{i_k}, \alpha - \alpha_k\rangle + \frac{L_2^F}{2}\|\alpha - \alpha_k\|_2^2 + \lambda |\alpha_{i_k}|,
\end{align*}
where $i_k= \argmin_k \min_{t \in \R} P_{:, k}^\top \nabla  f(P\alpha)(t - \alpha_k) + \frac{L_2^F}{2} (t - \alpha_k)^2 + \lambda |t|$. We introduce the gauge function $\gamma_{\mathcal{P}}(x) = \inf_{\alpha \in (\R^+)^d} \sum_{i=1}^d \alpha_i P_i$. Then, applying the GS-rule can be formulated as,
\begin{align*}
    i_k &= \argmin_k \min_{u + \alpha_k \geqslant 0} P_{:, k}^\top \nabla  f(P\alpha)u + \frac{L_2^F}{2} u^2 + \lambda |u + \alpha_k|, \\
    &= \argmin_k \Delta_k = \frac{1}{2L_2^F}(P_{:, k}^\top \nabla  f(P\alpha) + \lambda - L_2^F\alpha_k)_+^2 - \frac{1}{2L_2^F}(P_{:, k}^\top \nabla  f(P\alpha) + \lambda)^2.
\end{align*}
Let $\mathcal{P}$ be the set of atoms, $\mathcal{P}_k = \{k, \alpha_k \neq 0\}$ the set of visited atoms at iteration $k$. The algorithm consists in computing:
\begin{itemize}
    \item for non visited atoms $k \in \mathcal{P}\setminus \mathcal{P}$ ($\alpha_k = 0$), $\Delta_k = -\frac{1}{2L_2^F}(-P_{:, k}^\top \nabla  f(P\alpha) - \lambda)_+^2$ and this value can be computed using the ${\rm LMO}_{\mathcal{P}\setminus \mathcal{P}_k}(\nabla f(P\alpha_k)) = p_{{\rm out}}$, leading to $\Delta_{{\rm out}} = \frac{1}{2L_2^F}(-p_{{\rm out}}^\top \nabla f(P\alpha_k) - \lambda)_+^2 $, which costs $O(|\mathcal{P}\setminus \mathcal{P}_k|)$,
    \item for visited atoms, the objective needs to be computed completely in at most $|\mathcal{P}_k|$ iterations: $\Delta_{{\rm in}} = \frac{1}{2L_2^F} \sup_{p \in \mathcal{P}_k}(p^\top \nabla f(P\alpha_k) + \lambda)^2 - (p^\top \nabla f(P\alpha_k) + \lambda - L\alpha_k^p)_+^2$.
\end{itemize}
Then, we compute $i_k$ corresponding to the minimizer of $\min(\Delta_{{\rm in}}, \Delta_{{\rm out}})$, and compute the update $\alpha_{k+1}^{i_k} = (\alpha_k^{i_k} - \frac{1}{L_2^F}(p_{i_k}^\top \nabla f(P\alpha_k) + \lambda))_+$.

\section{The ultimate method: an inner loop strategy}
\label{sec:inner_loop_strategy}
The ultimate method is defined as a minimization problem, that has no closed-form solution. To overcome this issue, at each iteration $k$, we will solve iteratively the inner minimization problem using a randomized alternating minimization technique:
\begin{equation*}
   \min_{\eta, \beta \in \R^d} \langle P^\top\nabla f(P\alpha), \beta - \alpha \rangle + \frac{L_{\gamma_{\mathcal{P}}}^f}{2} \|\beta - \alpha\|_1^2 + \lambda \|\nu\|_1, \text{ such that } P\beta = P \nu.
\end{equation*}
First, we give the guarantees of an inner loop strategy, and in a second part, the result when applying a randomized alternating minimization technique to the inner loop.

Let us provide the convergence guarantees of the inner loop strategy, following the example of \citet[Section 5.2]{daspremont2021}. Let $L(\beta, \nu) = \langle P^\top\nabla f(P\alpha), \beta - \alpha \rangle + \frac{L_{\gamma_{\mathcal{P}}}^f}{2} \|\beta - \alpha\|_1^2 + \lambda \|\nu\|_1$ and $L_\star^\alpha = \min_{\eta, \beta \in \R^d} L(\beta, \nu), \text{ such that } P\beta = P \nu$. Given an approximate solution to this problem at iteration $k$, Theorem~\ref{theo_inner_loop} provides a guarantee on the outer loop.

\begin{theorem}
\label{theo_inner_loop}
    Let $(\tilde{\beta}_k, \tilde{\nu}_k)$ with $\tilde{x}_k = P\tilde{\beta}_k = P\tilde{\nu}_k$ be an approximate solution of $x_k$ produced by the ultimate method~(3.8) such that $L^k(\tilde{\beta}_k, \tilde{\nu}_k) - L_\star^k \geqslant \epsilon_k$, then we have
    \begin{equation*}
    f(\tilde{x}_{k}) - f_\star \leqslant \left(1- \frac{\mu_{\gamma_{\mathcal{P}}}^f}{L_{\gamma_{\mathcal{P}}}^f}\right)^k (f(x_0) - f_\star) + \sum_{i=0}^{k-1} \left(1- \frac{\mu_{\gamma_{\mathcal{P}}}^f}{L_{\gamma_{\mathcal{P}}}^f}\right)^i \epsilon_{k-i}.
\end{equation*}
\end{theorem}
\begin{proof}
\vspace{-0.2cm}
    \begin{align*}
    f(\tilde{x}_{k+1}) &\leqslant f(\tilde{x}_k) + L^k(\tilde{\eta}_{k+1}, \tilde{\beta}_{k+1}), \\
   & \leqslant f(\tilde{x}_k) + L^k_\star + \epsilon_k, \\
    &\leqslant f(\tilde{x}_k) - \frac{\mu_{\gamma_{\mathcal{P}}}^f}{L_{\gamma_{\mathcal{P}}}^f}(f(\tilde{x}_k) - f_\star) + \epsilon_k, \\
    f(\tilde{x}_{k+1}) - f_\star &\leqslant \left(1- \frac{\mu_{\gamma_{\mathcal{P}}}^f}{L_{\gamma_{\mathcal{P}}}^f}\right) (f(\tilde{x}_k) - f_\star) + \epsilon_k.
\end{align*}
The result is obtained by a direct summation.
\end{proof}
The precision of the outer loop depends on the precision $\epsilon_k$ of the inner loop at each iteration $k$. Given an iterative method with iteration number $t$ to obtain $(\tilde{\beta}_k^t, \tilde{\nu}_k^t)$, Theorem~\ref{theo_inner_loop} ensures that the better the precision of the inner loop, the better the convergence guarantee of the outer loop. More precisely, if the error is constant $\epsilon_k = \epsilon$, the global convergence guarantee becomes $f(\tilde{x}_{k}) - f_\star \leqslant \left(1- \frac{\mu_{\gamma_{\mathcal{P}}}^f}{L_{\gamma_{\mathcal{P}}}^f}\right)^k (f(x_0) - f_\star) + \epsilon \frac{L_{\gamma_{\mathcal{P}}}^f}{\mu_{\gamma_{\mathcal{P}}}^f}\left(1 - (1 - \frac{\mu_{\gamma_{\mathcal{P}}}^f}{L_{\gamma_{\mathcal{P}}}^f})^{k}\right)$. For a linearly decreasing inner precision $\epsilon_i = \left( 1 - \frac{\mu_{\gamma_{\mathcal{P}}}^f}{L_{\gamma_{\mathcal{P}}}^f}\right)^i$, the global convergence rate is exactly $1 - \frac{\mu_{\gamma_{\mathcal{P}}}^f}{L_{\gamma_{\mathcal{P}}}^f}$. Finally, for a sublinearly inner precision, the global convergence guarantee is given by $F(\tilde{x}_{k}) - F_\star \leqslant \left(1- \frac{\mu_{\gamma_{\mathcal{P}}}^f}{L_{\gamma_{\mathcal{P}}}^f}\right)^k (F(x_0) - F_\star) + \sum_{i=0}^k-1 \left(1- \frac{\mu_{\gamma_{\mathcal{P}}}^f}{L_{\gamma_{\mathcal{P}}}^f}\right)^i \frac{1}{(k - i + 1)^\alpha}$. The inner precision $\epsilon_k$ can slow down the global convergence of the ultimate method. However, achieving a low precision in the inner loop can be very costly, especially if the inner method converges sublinearly.

\subsection{Alternating randomized block coordinate descent}
\label{section_AR_BCD}

We propose to solve this minimization problem that defines matching pursuit using an alternating minimization technique. 
\citet{2018Diakonikolas} developed the alternating randomized block coordinate descent (AR-BCD), that generalizes alternating minimization to more than two blocks. 
\begin{theorem}{\citep[Theorem 3.4]{2018Diakonikolas}}
\label{AR-BCD_convergence}
    Let $x_k$ be generated by AR-BCD with a distribution $(p_i)_{i=1, \ldots, n-1}$ over $n-1$ $L^{f, i}$-smooth blocks, and $n$ be the non-smooth block. Then, for $R_i = \max_{x \in \R^d, f(x) \leqslant f(x_0)} \|x_{\star, i} - x_i\|^2$,
    \begin{equation*}
        \E[f(x_{k+1})] - f_\star \leqslant \frac{2}{k+3}\left(\sum_{i=1}^{n - 1} \frac{L^{f, i}}{p_i}R_i\right).
    \end{equation*}
\end{theorem}
\vspace{-0.1cm}
The AR-BCD algorithm converges sublinearly for a non-smooth objective, according to Theorem~\ref{AR-BCD_convergence}. Let us first reformulate the minimization problem as an unconstrained problem, calling $\alpha = \alpha_k$:
\begin{align*}
    &\min_{\eta, \beta \in \R^d} \langle P^\top\nabla f(P\alpha), \beta - \alpha \rangle + \frac{L_{\gamma_{\mathcal{P}}}^f}{2} \|\beta - \alpha\|_1^2 + \lambda \|\eta\|_1, \text{ such that } P\beta = P \eta, \\
   & =\min_{\eta, k \in \R^d} \langle P^\top\nabla f(P\alpha), \eta - \alpha \rangle + \frac{L_{\gamma_{\mathcal{P}}}^f}{2} \|\eta + k - \alpha\|_1^2 + \lambda \|\eta\|_1, \text{ such that } k \in \text{Ker}(P), \\
   &=\min_{\eta \in \R^d, z \in \R^{d_{Q}}} \langle P^\top\nabla f(P\alpha), \eta - \alpha \rangle + \frac{L_{\gamma_{\mathcal{P}}}^f}{2} \|\eta + Qz - \alpha\|_1^2 + \lambda \|\eta\|_1,
\end{align*}
where $Q$ is a basis for ${\rm Ker}(P)$ and $d_Q = {\rm dim(Ker(Q))}$, obtained using a QR decomposition. We now use the eta-trick on $\|\cdot\|_1^2$ and an other $\eta$-trick on $\|\cdot\|_1$ (see~\ref{eta_trick_appendix}) , which leads to the equivalent minimization problem,
\begin{align*}
    \min_{\eta \in \R^d, z \in \R^{r}} \min_{\gamma \in \Delta_d, \theta \geqslant 0}  G(\eta, z, \gamma, \theta) &= \langle P^\top\nabla f(P\alpha), \eta - \alpha \rangle + \frac{\lambda}{2} \left(\eta^\top\text{Diag}(\theta)^{-1} \eta + \text{Diag}(\theta) \mathrm{1}\right)\\
    & + \frac{L_{\gamma_{\mathcal{P}}}^f}{2} (\eta + Qz - \alpha)^\top \text{Diag}(\gamma)^{-1}(\eta + Qz - \alpha) .
\end{align*}
\begin{remark}
    Note that without applying the second $\eta$-trick on $\|\cdot\|_1$, the objective function is smooth with respect to $z$, no non-smooth with respect to $\eta, \gamma$, which prevents us from using the alternating minimization technique.
\end{remark}
Each coordinate can be solved as follows:
\begin{align*}
    z_{opt} &= \left(Q^\top \text{Diag}(\gamma_k)^{-1}Q\right)^{-1}Q^\top \text{Diag}(\gamma_k)^{-1} (\alpha - \eta), \\
    \eta_{opt} &= \text{Diag}(L_{\gamma_{\mathcal{P}}}^f / \gamma + \lambda / \theta)^{-1}\left(L_{\gamma_{\mathcal{P}}}^f \text{Diag}(\gamma)^{-1}(\alpha - Qz) - P^\top \nabla f(P\alpha) \right), \\
    \gamma_{opt} &= \frac{|\eta^i + (Qz)^i - \alpha^i|}{\|\eta + (Qz) - \alpha\|_1}, \\
    \theta_{opt} &= |\eta|.
\end{align*}
Let us rewrite our problem into $\min_{\eta \in \R^d, z \in \R^{r}} \min_{\gamma \in \Delta_d, \theta \geqslant 0} G(\eta, z, \gamma, \theta) = G(\eta, z, \beta) = G(\xi)$
where $G(\cdot)$ is smooth with respect to $\eta, z$ but non smooth with respect to $\beta = (\gamma, \theta)^\top$. We perform AR-BCD with probabilities $p_1$ for $\eta$ (respectively $p_2 = 1 - p_1$ for $z$). Let us rewrite $(\eta, z, \beta) = (\eta_1, \eta_2, \eta_3)$, and let $S_i(\xi)$ be the set of points that differs from $\xi$ only over block $i$. AR-BCD is given at iteration $k$ by:
\begin{align*}
    & \text{Pick } i_k \in \{1, 2\} \text{ with probability }  p_{i_k}, \\
    & \tilde{\xi}_{k+1} = \argmin_{\xi \in S_{i_k}(\xi_k)} G(\xi), \\
    & \xi_{k+1} = \argmin_{\xi \in S_3(\tilde{\xi}_{k+1})} G(\xi).
\end{align*}

\subsection{Alternating minimization}
\label{section_alternating_minization}
\begin{align*}
    \min_{\eta \in \R^d, z \in \R^{d_Q}} \min_{\gamma \in \Delta_d}\langle P^\top\nabla f(P\alpha), \eta - \alpha \rangle + \frac{L_{\gamma_{\mathcal{P}}}^f}{2} (\eta + Qz - \alpha)^\top \text{Diag}(\gamma)^{-1}(\eta + Qz - \alpha) + \lambda \|\eta\|_1.
\end{align*}
The objective function $F(\eta, z, \gamma)$ is jointly convex in $(\eta, z, \gamma)$ (for $\gamma >0$). This problem can be solved using alternating minimization (which is stronger than coordinate descent). Starting from $\eta_0 \in \R^d, z_0 \in \R^r (r = d - \text{rg}(P) = r, \gamma_0 \in \Delta_d$,
\begin{align*}
    z_{k+1} &= \argmin_{z}F(\eta_k, z, \gamma_k), \\
    \gamma_{k+1} &= \argmin_{\gamma}F(\eta_{k}, z_{k+1}, \gamma), \\
    \eta_{k+1} &= \argmin_{\eta}F(\eta, z_{k+1}, \gamma_{k+1}).
\end{align*}
In this context, it corresponds to computing 
\begin{align*}
    z_{k+1} &= \left(Q^\top \text{Diag}(\gamma_k)^{-1}Q\right)^{-1}Q^\top \text{Diag}(\gamma_k)^{-1} (\alpha - \eta_{k}), \\
    \eta_{k+1} &= S_{\lambda/L_{\gamma_{\mathcal{P}}}^f \gamma_k}\left(\alpha - \frac{1}{L_{\gamma_{\mathcal{P}}}^f} \text{Diag}(\gamma_k)P^\top \nabla f(P\alpha) - Qz)\right), \\
    \gamma_{k+1} &= \frac{|\eta^i_{k+1} + (Qz_{k+1})^i - \alpha^i|}{\|\eta_{k+1} + (Qz_{k+1}) - \alpha\|_1}.
\end{align*}

\subsection{Experimental results}

In Figure~\ref{fig:ar-bcd_inner_loop}, we apply an alternating randomized block coordinate descent (AR-BCD) technique  to solve a well-chosen inner-loop optimization problem, which was developed by~\citet{2018Diakonikolas}. The inner loop strategy is developed in~\ref{section_AR_BCD}, as well as its convergence guarantee. We also derive an alternating minimization method on three blocks, as detailed in~\ref{section_alternating_minization}, simpler than AR-BCD, for which there is no convergence guarantee.

\begin{figure}[!h]
    \centering
\includegraphics[height=135pt]{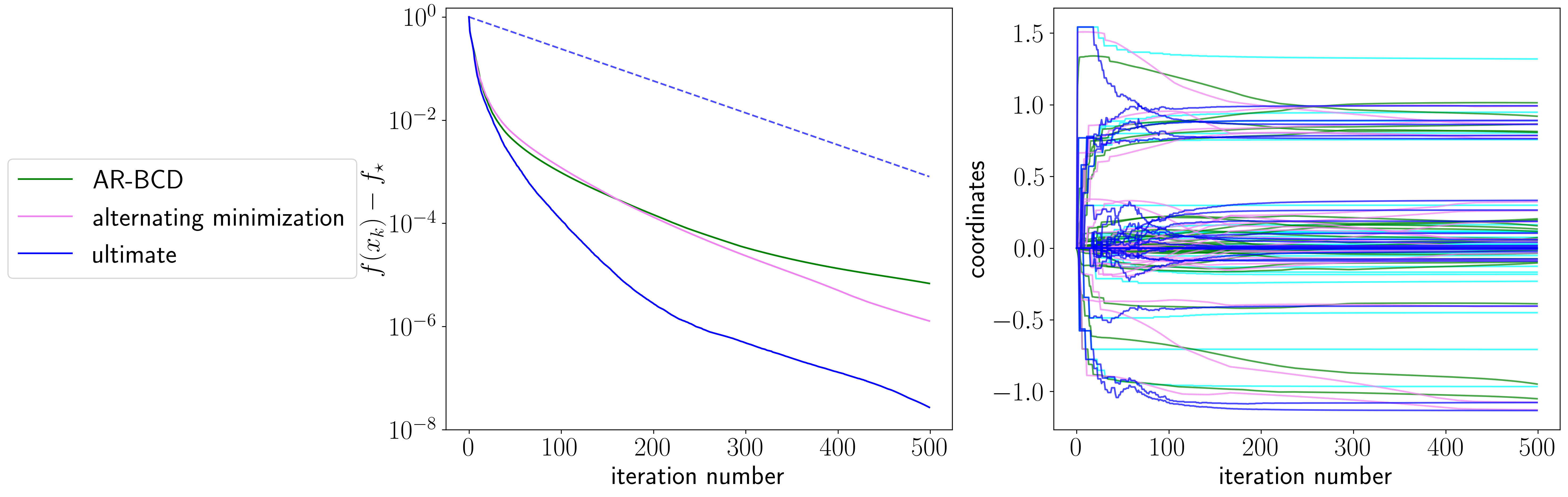}
\hfill
    \caption{Convergence in function value for the ultimate method computed with an alternating minimization technique, AR-BCD, and with the solver MOSEK, compared to the regularized matching pursuit. Parameters are given by $d=50$, $n=20$, $s=8$ and $\lambda = 0.001$, with an inner loop with exponential precision.}
    \label{fig:ar-bcd_inner_loop}
\end{figure}

\end{document}